\documentclass{amsart}
\usepackage[utf8]{inputenc}
\usepackage{amsfonts}
\usepackage{amsthm}
\usepackage{amssymb}
\usepackage{amsmath}
\usepackage{mathtools}
\usepackage{eucal}
\usepackage{mathrsfs}
\usepackage{wasysym}
\usepackage{comment}
\usepackage[centering,bottom=3cm,margin=2.7cm,top=3cm]{geometry}
\usepackage{tikz-cd}
\usetikzlibrary{matrix,arrows,decorations.pathmorphing}
\usepackage[backend=bibtex, maxnames=4]{biblatex}
\usepackage{xcolor}
\usepackage{hyperref}
\usepackage{cleveref}

\newcommand\myshade{85}
\colorlet{mylinkcolor}{violet}
\colorlet{mycitecolor}{blue}
\colorlet{myurlcolor}{green}

\hypersetup{
  linkcolor  = mylinkcolor!\myshade!black,
  citecolor  = mycitecolor!\myshade!black,
  urlcolor   = myurlcolor!\myshade!black,
  colorlinks = true,
}

\newcommand{\sssec}[0]{\subsubsection}

\newcommand{\on}{\operatorname}
\newcommand{\nc}{\newcommand}
\nc{\TV}{Toen--Vezzosi}
\nc{\et}{{\on{\acute{e}t}}}
\nc{\virg}[1]{``#1"}
\nc{\bigt}[1]{\big( #1 \big) }
\nc{\biggt}[1]{\bigg( #1 \bigg) }
\nc{\Bigt}[1]{\Big( #1 \Big) }
\nc{\Biggt}[1]{\Bigg( #1 \Bigg) }
\nc{\biggq}[1]{\bigg[ #1 \bigg]}
\nc{\bigq}[1]{\big[ #1 \big] }
\nc{\Bigq}[1]{\Big[ #1 \Big] }
\nc{\wt}{\tilde}

\nc{\violet}[1]{\textcolor{violet}{#1}}
\nc{\green}[1]{\textcolor{green}{#1}}
\nc{\orange}[1]{\textcolor{orange}{#1}}

\theoremstyle{plain}
\newtheorem{thm}[subsubsection]{Theorem}
\newtheorem{mainthm}{Theorem}

\newtheorem{lem}[subsubsection]{Lemma}

\newtheorem{cor}[subsubsection]{Corollary}
\newtheorem{conj}[subsubsection]{Conjecture}
\theoremstyle{definition}

\newtheorem{defn}[subsubsection]{Definition}

\theoremstyle{remark}
\newtheorem{rmk}[subsubsection]{Remark}

\newtheorem{rem}[subsubsection]{Remark}

\newtheorem{example}[subsubsection]{Example}

\theoremstyle{plain}

\theoremstyle{definition}

\newcommand{\ccC}[0]{\mathcal{C}}

\newcommand{\ccE}[0]{\mathcal{E}}

\newcommand{\ccH}[0]{\mathcal{H}}

\newcommand{\ccM}[0]{\mathcal{M}}

\newcommand{\ccO}[0]{\mathcal{O}}

\newcommand{\ccS}[0]{\mathcal{S}}

\newcommand{\bbA}[0]{\mathbb{A}}

\newcommand{\bbE}[0]{\mathbb{E}}

\newcommand{\bbL}[0]{\mathbb{L}}

\newcommand{\bbN}[0]{\mathbb{N}}

\newcommand{\bbQ}[0]{\mathbb{Q}}

\newcommand{\bbU}[0]{\mathbb{U}}

\newcommand{\bbZ}[0]{\mathbb{Z}}

\newcommand{\ffF}[0]{\mathfrak{F}}

\newcommand{\ffL}[0]{\mathfrak{L}}

\newcommand{\ffT}[0]{\mathfrak{T}}

\newcommand{\fff}[0]{\mathfrak{f}}
\newcommand{\ffg}[0]{\mathfrak{g}}
\newcommand{\ffh}[0]{\mathfrak{h}}
\newcommand{\ffi}[0]{\mathfrak{i}}

\newcommand{\ffk}[0]{\mathfrak{k}}

\newcommand{\fft}[0]{\mathfrak{t}}
\newcommand{\ffu}[0]{\mathfrak{u}}

\newcommand{\rmB}[0]{\mathrm{B}}

\newcommand{\scrH}[0]{\mathscr{H}}

\newcommand{\scrK}[0]{\mathscr{K}}

\newcommand{\scrY}[0]{\mathscr{Y}}

\nc{\sA}{{\mathsf{A}}}
\nc{\sB}{{\mathsf{B}}}
\nc{\sC}{{\mathsf{C}}}
\nc{\sD}{{\mathsf{D}}}
\nc{\sE}{{\mathsf{E}}}
\nc{\sF}{{\mathsf{F}}}
\nc{\sG}{{\mathsf{G}}}
\nc{\sK}{{\mathsf{K}}}
\nc{\sM}{{\mathsf{M}}}
\nc{\sN}{{\mathsf{N}}}
\nc{\sO}{{\mathsf{O}}}
\nc{\sW}{{\mathsf{W}}}
\nc{\sQ}{{\mathsf{Q}}}
\nc{\sP}{{\mathsf{P}}}
\nc{\sR}{{\mathsf{R}}}
\nc{\sS}{{\mathsf{S}}}
\nc{\sT}{{\mathsf{T}}}
\nc{\sU}{{\mathsf{U}}}
\nc{\sV}{{\mathsf{V}}}
\nc{\sZ}{{\mathsf{Z}}}

\newcommand{\up}[1]{\on{#1}}
\newcommand{\ul}[1]{\underline{#1}}
\newcommand{\ol}[1]{\overline{#1}}

\newcommand{\usotimes}{\otimes}
\nc{\enh}{{\on{enh}}}
\nc{\eZ}{\on{e-}\!\red{Z}}
\nc{\Getp}{\uG_0\on{-etp}}
\nc{\at}{\on{at}}
\nc{\colim}{\varinjlim}
\renewcommand{\Bar}[0]{\on{Bar}}
\newcommand{\inj}[0]{\on{inj}}
\renewcommand{\epsilon}{\varepsilon}
\newcommand{\eps}{\epsilon}
\newcommand{\op}[0]{{\on{op}}}

\nc{\cat}{\on{cat}} 

\newcommand{\etp}[0]{\on{etp}}
\newcommand{\tp}[0]{\on{tp}}
\newcommand{\oo}[0]{\infty}

\newcommand{\uG}[0]{\on{G}}
\newcommand{\uH}[0]{\on H}
\newcommand{\uI}[0]{\on{I}}
\newcommand{\uK}[0]{\on{K}}
\newcommand{\uL}[0]{\on{L}}
\newcommand{\uR}[0]{\on{R}}
\newcommand{\GLK}[0]{\on{G}_{\on{L/K}}}

\newcommand{\Tr}[0]{\on{Tr}}

\newcommand{\coFib}[0]{\on{coFib}}
\newcommand{\Fib}[0]{\on{Fib}}
\nc{\act}{{\on{act}}}
\nc{\rev}{{\on{rev}}}
\nc{\env}{{\on{env}}}
\newcommand{\angles}[1]{\langle #1 \rangle}
\newcommand{\Angles}[1]{\bigl \langle #1 \bigr \rangle}
\newcommand{\can}[0]{\on{can}}

\newcommand{\xto}[1]{\xrightarrow{#1}}
\newcommand{\tto}{\twoheadrightarrow}
\newcommand{\hto}{\hookrightarrow}
\DeclareMathOperator{\ev}{ev} 
\DeclareMathOperator{\coev}{coev} 
\DeclareMathOperator{\id}{id} 
\newcommand{\longto}{\longrightarrow}

\newcommand{\IK}{{\uI_{\uK}}}
\newcommand{\IL}{{\uI_{\uL}}}

\newcommand{\CAlg}{\on{CAlg}}
\newcommand{\sCAlg}{\on{CAlg}^{\Delta^{\op}}}
\newcommand{\Alg}{\on{Alg}}
\newcommand{\Ql}[1]{\bbQ_{\ell #1}}
\newcommand{\BU}[0]{\rmB \bbU}
\newcommand{\End}[0]{\up{End}}
\newcommand{\Qell}{\bbQ_{\ell}}
\nc{\sBenv}{\sB^{\env}}
\nc{\sCenv}{\sC^{\env}}
\newcommand{\QellI}{\bbQ_\ell^{\uI}}

\newcommand{\QellIGred}{\bbQ_\ell^{\uI}(\GLK)}

\nc{\Benv}{\CB^{\env}}
\nc{\Sym}{\on{Sym}}

\nc{\red}[1]{#1_{\on{red}}}
\newcommand{\Spec}[1]{\on{Spec}(#1)}
\newcommand{\dSch}[0]{\on{dSch}}
\newcommand{\Sch}[0]{\on{Sch}}
\nc{\bareta}{\bar{\eta}}

\DeclareMathOperator{\Hom}{Hom} 
\DeclareMathOperator{\Map}{Map} 
\newcommand{\pr}[1]{\on{pr}_{#1}}
\nc{\olpr}{\ol{\pr}}
\nc{\tdelta}{\widehat\delta}

\nc{\Ind}[0]{\on{Ind}}

\newcommand{\Tor}[0]{\on{Tor}}

\newcommand{\sTor}[0]{\on{sTor}}

\newcommand{\HH}[0]{\on{HH}}

\newcommand{\HK}[0]{\on{HK}}
\newcommand{\G}[0]{\on{G}}

\newcommand{\Set}[0]{\on{Set}}
\newcommand{\sSet}[0]{\on{Set}^{\Delta^{\op}}}

\newcommand{\dgCat}{\on{dgCat}}
\newcommand{\dgCAT}[0]{\on{dgCAT}}

\newcommand{\Mod}[0]{\on{Mod}}
\newcommand{\Fun}[0]{\on{Fun}}

\newcommand{\dgMod}[0]{\on{Mod}_{\on{dg}}}
\newcommand{\cofdgMod}[0]{\on{Mod}_{\on{dg}}^{\on{cof}}}

\newcommand{\Cohlevel}[1]{\mathsf{D^{\leq #1}_{coh}}}
\newcommand{\Cohlevelplus}[1]{\mathsf{D^{\geq #1}_{coh}}}
\DeclareMathOperator{\Coh}{\mathsf{D^b_{coh}}} 
\DeclareMathOperator{\Cohminus}{\mathsf{D^{--}_{coh}}} 
\DeclareMathOperator{\Cohinfty}{\mathsf{D_{coh}}}
\DeclareMathOperator{\Cohext}{\mathsf{D^{b,\boxtimes}_{coh}}} 
\DeclareMathOperator{\Perf}{\mathsf{D_{perf}}} 

\DeclareMathOperator{\Perfext}{\mathsf{D^{\boxtimes}_{perf}}} 
\DeclareMathOperator{\QCoh}{\mathsf{D_{qcoh}}} 
\DeclareMathOperator{\Sing}{\mathsf{D_{sg}}} 
\DeclareMathOperator{\Singext}{\mathsf{D^{\boxtimes}_{sg}}} 
\DeclareMathOperator{\MF}{\on{MF}}
\DeclareMathOperator{\MFcoh}{\on{MF}^{\on{coh}}}

\newcommand{\SH}[0]{\ccS \ccH}

\newcommand{\Shv}[0]{\on{Shv}_{\Ql{}}}

\newcommand{\Mv}[0]{\ccM^{\vee}}
\newcommand{\rl}[0]{\on{r}^{\ell}}
\newcommand{\chern}[0]{\ccC\!\on{h}^{\ell}}
\newcommand{\rell}{\rl}

\newcommand{\Sw}[0]{\on{Sw}}
\newcommand{\dimtot}[0]{\on{dimtot}}

\newcommand{\ar}[0]{\on{ar}_{\uL\!/\!\uK}}

\nc{\arcat}{\on{ar}_{\uL\!/\!\uK}^{\on{cat}}}
\nc{\idcat}{\id^{\cat}}
\nc{\Bl}{\on{Bl}} 
\nc{\Ar}{\on{Ar}} 

\nc{\hB}{\on{h}_{\sB}}
\nc{\hC}{\on{h}_{\sC}}

\nc{\restr}[2]{\left. #1 \right |_{#2}}

\nc{\wideprime}[1]{#1'}

\renewcommand{\setminus}{\smallsetminus}
\renewcommand{\sim}{\simeq}
\nc{\aug}{\on{aug}}
\nc{\rot}{\on{rot}}

\nc{\comm}{\on{comm}}

\title{Categorification of the localized intersection product and Bloch conductor formula}

\author{Dario Beraldo \and Massimo Pippi}

\address[D.~Beraldo]{Department of Mathematics, Université de Genève, Genève, Suisse}
\email{dario.beraldo@unige.ch}

\address[M.~Pippi]{Univ Angers, CNRS-UMR 6093, LAREMA, SFR MATHSTIC, F-49000 Angers,France}
\email{massimo.pippi@univ-angers.fr}

\begin{document}

\begin{abstract}
We categorify the localized intersection product 
on arithmetic schemes defined by Kato--Saito in \cite{katosaito04}.
As an application, we prove a generalization of Bloch conductor conjecture.
\end{abstract}

\maketitle

\tableofcontents


\section{Introduction}

In 1985, S. Bloch proposed a conjectural formula which relates numerical invariants of different nature of an arithmetic scheme, see \cite{bloch87} and Conjecture \ref{conj:Bloch} below. In this paper, we prove this conjecture together with its generalization, Conjecture \ref{conj:gBCC}.

Our method amounts to revisiting the Kato--Saito localized intersection product (\cite[Definition 5.1.5]{katosaito04}), relating this with the non-commutative approach by \TV{} (\cite{toenvezzosi22}), and then combining these two ingredients with our earlier work (\cite{beraldopippi24}).

\subsection{The generalized Bloch conductor conjecture}

We recall the statement of the conjecture proposed by Bloch in \cite{bloch87} and the generalization we prove in this paper.

\sssec{}

Let $A$ be a complete discrete valuation ring (DVR) with algebraically closed residue field $k$. Denote by $\uK=\on{Frac}(A)$ the fraction field and fix a separable closure $\ol \uK$.
We will use the standard notations
$$
  s := \Spec{k}
    \hto 
  S := \Spec{A}
    \hookleftarrow
  \eta := \Spec{\uK}
    \leftarrow
  \bareta := \Spec{\ol\uK}.
$$

\sssec{}\label{hypothesis gBCC}

Let $p:X\to S$ be an arithmetic $S$-scheme. 
By this, we mean that $X$ is a flat, generically smooth and regular $S$-scheme of finite type. 
We further assume that the singular locus $Z$ of $p:X\to S$ is proper over $S$.
We will denote $X_s$, $X_{\eta}$ and $X_{\bareta}$ the special, generic and geometric generic fiber, respectively.

\sssec{}

Let us fix a prime number $\ell$ which is different from the residue characteristic of $A$. 
Let 
$$
  \Phi
    :=
  \Phi_p(\Ql{,X})
$$
be the $\ell$-adic complex of vanishing cycles of $p:X\to S$ (\cite[Expos\'e I]{sga7i}, \cite[Expos\'e XIII]{sga7ii}).

\sssec{}

The \emph{generalized Bloch Conductor Conjecture}, to be referred to in the sequel simply as \virg{gBCC}, relates three numerical invariants of $X/S$: 

\begin{itemize}
    \item the Bloch intersection number 
          $
          \Bl(X/S)
          $, 
          an \virg{algebro-geometric} invariant related to algebraic differential forms of $X/S$;
    \item the Euler characteristic 
          $
          \chi \bigt{\uH^*_{\et}(X_s,\Phi)}
          $, a \virg{topological} invariant;
    \item the Swan conductor 
          $
          \Sw \bigt{\uH^*_{\et}(X_s,\Phi)}
          $, an \virg{arithmetic} invariant related to the action of the wild inertia group on 
          $
          \Phi
          $.
\end{itemize}

We refer the reader to \cite{bloch87, katosaito04, beraldopippi24} for details on the definitions of such integers.

\begin{conj}[Generalized Bloch Conductor Formula] \label{conj:gBCC}
Let $p: X \to S$ be an arithmetic $S$-scheme 
with $S$-proper singular locus $Z$. 
Then there is an equality of integers
\begin{equation}
\tag{gBCF}
\label{eqn:gBCC in intro}
  \Bl(X/S)
    =
  -
  \chi \bigt{\uH^*_{\et}(X_s,\Phi)}
  - 
  \Sw \bigt{\uH^*_{\et}(X_s,\Phi)}.   
\end{equation}
\end{conj}

\sssec{}

The original \emph{Bloch conductor formula}, to be referred to as \eqref{eqn:Bloch-formula}, is the special case when $p: X \to S$ is proper. 
In this case, the singular locus of $p$ is automatically proper over $S$; moreover, the vanishing cohomology and the Swan conductor simplify to give the following conjectural formula.

\begin{conj}[Bloch Conductor Formula]\label{conj:Bloch}
Let $p:X \to S$ be a proper arithmetic $S$-scheme. 
Then there is an equality of integers
\begin{equation}
\tag{BCF}
\label{eqn:Bloch-formula}
  \Bl(X/S) 
    =
  \chi(X_s;\Qell) -\chi(X_{\bareta};\Qell) 
    -
  \Sw\bigt{\uH^*_{\et}(X_{\bareta},\Qell)}.
\end{equation}
\end{conj}

\sssec{}

As mentioned, the latter conjecture was stated by Bloch in 1985 in \cite{bloch87}. On the other hand, Conjecture \ref{conj:gBCC} first appeared in print in \cite{orgogozo03} and it was taken up by us in \cite{beraldopippi22,beraldopippi23, beraldopippi24}.
There, following a program initiated by \TV{} in \cite{toenvezzosi17,toenvezzosi19,toenvezzosi22},
we proved a categorical version of \eqref{eqn:gBCC in intro}: this categorical version is an equality of classes in 
\begin{equation} \label{eqn:place of cat BCC}
 \uH^0_{\et}\Bigt{S, \rl_S \bigt{\HH(\sB/A)}},    
\end{equation}
the $\Qell$-vector space of global sections of a certain complex of $\ell$-adic sheaves that we are going to review below.
In this paper, we show how Conjecture \ref{conj:gBCC} can be deduced from this categorical version.

\begin{rmk}

Several cases of Conjecture \ref{conj:Bloch} have already been established.

\begin{itemize}
    \item In \cite[Expos\'e XVI]{sga7ii}, Deligne states (an equivalent reformulation of, see \cite{orgogozo03}) the conjecture in the case of an isolated singularity and proves it in relative dimension $0$, in the case of ordinary quadratic singularities and in the case of pure characteristic.
    \item In \cite{bloch87}, Bloch states his conjectural formula and proves it in the case of relative dimension $1$.
    \item In \cite{kapranov91}, Kapranov essentially provides a proof in the case of characteristic $0$.
    \item In \cite{katosaito04}, Kato--Saito gave a proof under the assumption that the reduced special fiber is a normal crossing divisor.
    \item In \cite{saito21}, Saito proved the formula in the pure characteristic case assuming $X/S$ projective.
    In \cite{abe22a,abe22b}, this approach is extended to the case where $X/S$ is proper, still in pure characteristic.
    \item In \cite{beraldopippi24}, we proved Conjecture \ref{conj:gBCC} in the mixed-characteristic case under the assumption that $X$ is a Cartier divisor in a smooth $S$-scheme, as well as in the pure characteristic case in general.
\end{itemize}
\end{rmk}

\subsection{The categorical generalized Bloch conductor formula}

\sssec{}

Let us recall the various objects appearing in \eqref{eqn:place of cat BCC}:

\begin{itemize}
    \item 
    $\sB$ is the monoidal dg-category of singularities of the derived self-intersection 
    $$
      G:=s\times_Ss
    $$
    of $s\hto S$, see Section \ref{sec: differential forms on two-periodic complexes} for details;
    \item 
    $\HH(\sB/A)$ is the Drinfeld cocenter (or Hochschild homology) of the monoidal $A$-linear dg-category $\sB$, see Section \ref{sec: differential forms on two-periodic complexes}; 
    \item 
    $\rl_S$ is the $\infty$-functor of $\ell$-adic realization of dg-categories introduced in \cite{brtv18} and quickly reviewed in Section \ref{ssec: cohomology of dg-categories}. It associates a complex of $\ell$-adic sheaves to any $A$-linear dg-category. 
\end{itemize}

\sssec{}

As mentioned above, the proof of gBCC relies on its categorical counterpart proved in \cite{beraldopippi24}.
We can summarize a large part of the content of \emph{loc. cit.} as follows.
Let $p:X \to S$ be as in Section \ref{hypothesis gBCC}.
In \cite{toenvezzosi22}, \TV{} introduced a dg-functor
$$
  \ev^{\HH}: 
  \Sing(X\times_SX) 
    \to 
  \HH(\sB/A)
$$
and considered the induced homomorphism of $\Qell$-vector spaces
$$
  \uH^0_{\et}\Bigt{S,\rl_S\bigt{\Sing(X\times_SX)}}
    \to
  \uH^0_{\et}\Bigt{S,\rl_S\bigt{\HH(\sB/A)}}.
$$
Here $\Sing(X\times_SX)$ denotes the $A$-linear dg-category of singularities associated to $X\times_SX$, see Section \ref{ssec: singularity categories} below. 
The definition of $\ev^{\HH}$ will be recalled in Section \ref{sssec: defn evHH}. 
It amounts to combining three facts: 

\begin{itemize}
    \item the dg-category of singularities 
          $
          \sT := \Sing(X_s)
          $ 
          is naturally a dualizable $(\sB,A)$-bimodule; 
    \item the dual of $\sT$ is the $(A,\sB)$-bimodule $\sT^{\op}$; 
    \item there is a \virg{K\"unneth formula} 
          $
          \sT^{\op} \otimes_\sB \sT 
            \simeq 
          \Sing(X \times_S X)
          $.
\end{itemize}
Then the dg-functor $\ev^{\HH}$ is obtained from the evaluation of this pair of dual objects.

\sssec{} 

The corresponding coevaluation is determined by the object $\Delta_X \in \Sing(X\times_SX)$ associated to the structure sheaf of the diagonal $X \to X \times_S X$. This object defines a cohomology class 
$$
  \chern ([\Delta_X]) 
    \in 
  \uH^0_{\et}\Bigt{S,\rl_S\bigt{\Sing(X\times_SX)}},
$$
which can be 
pushed along the previous map
to obtain the \emph{categorical Bloch intersection class}
$$
  \Bl^{\on{cat}}(X/S)
    :=
  \chern \bigt{\bigq{\ev^{\HH}(\Delta_X)}}
    \in 
  \uH^0_{\et}\Bigt{S,\rl_S\bigt{\HH(\sB/A)}},
$$
see \cite[Definition 5.2.1]{toenvezzosi22}.
Here $\chern$ is the non-commutative Chern character, also defined by \TV{} (see Section \ref{sssec: nc Chern character}).

\sssec{}

In \cite[Section 5.7.5]{beraldopippi24}, we proved that the class $\chern ([\Delta_X])$ is the sum of two other naturally defined cohomology classes: 
$$
  \chern ([\Delta_X])
    =
  \gamma_{\sT}
    +
  \gamma_{\sU}.
$$
In fact, for every finite Galois extension $\uK\subseteq \uL$ 
whose inertia group $\IL$ acts unipotently\footnote{The existence of $\uL$ is guaranteed by Grothendieck's monodromy theorem.} on $\Phi$, there is a splitting of the $\ell$-adic sheaf $\rl_S\bigt{\Sing(X\times_SX)}$ as the direct sum of two sub-sheaves.
This induces the above mentioned decomposition of $\chern ([\Delta_X])$.

Moreover, the restriction of 
$$
  \uH^0_{\et}\Bigt{S,\rl_S\bigt{\Sing(X\times_SX)}}
    \to
  \uH^0_{\et}\Bigt{S,\rl_S\bigt{\HH(\sB/A)}}
$$
to each of these two sub-sheaves factors through two homomorphisms, denoted
$$
  \idcat: 
  \Qell 
    \to 
  \uH^0_{\et}\Bigt{S,\rl_S\bigt{\HH(\sB/A)}}
$$
and
$$
  -\arcat:  
  \Qell(\GLK) 
    \to 
  \uH^0_{\et}\Bigt{S,\rl_S\bigt{\HH(\sB/A)}}
$$
in \cite{beraldopippi24}.
Here $\GLK$ is the Galois group of $\uK\subseteq \uL$ and 
$$
  \Qell(\GLK)
    :=
  \frac{\Qell[\GLK]}{\Angles{\sum_{g\in \GLK}g}}
$$
denotes the associated reduced group algebra.

\sssec{}

The image of $\gamma_{\sT}$ in $\Qell$ is known to be 
$$
  \Tr_{\QellI}\bigt{\id, \uH^*_{\et}(X_s,\Phi^{\IK}[-1])},
$$
while the image of $\gamma_{\sU}$ in $\Qell(\GLK)$ is known to be 
$$
  \Tr_{\QellI(\GLK)}\bigt{\id, \uH^*_{\et}(X_s,\Phi^{\IL}/\Phi^{\IK})}
    =
  \frac{1}{|\GLK|}
  \sum_{g\in \GLK} 
  \Tr_{\QellI}\bigt{g, \uH^*_{\et}(X_s,\Phi^{\IL}/\Phi^{\IK})}
  \angles{g^{-1}}.
$$
This yields the following equality of classes in the $\Qell$-vector space $\uH^0_{\et}\Bigt{S,\rl_S\bigt{\HH(\sB/A)}}$, which we named \emph{categorical generalized Bloch conductor formula} (\cite[Theorem A]{beraldopippi24}):
\begin{align}
\tag{cat-gBCF}
\label{eqn: categorical gBCF}
  \Bl^{\on{cat}}(X/S)
    = 
    &
  -\idcat \Bigt{\Tr_{\QellI}\bigt{\id, \uH^*_{\et}(X_s,\Phi^{\IK})}}
  \\
  \nonumber
  &
  -\arcat \Bigt{\Tr_{\QellI(\GLK)}\bigt{\id, \uH^*_{\et}(X_s,\Phi^{\IL}/\Phi^{\IK})}}.
\end{align}

\sssec{}

It is our opinion that the formula displayed above conceptually explains why Bloch conductor conjecture holds. 
To corroborate this claim, we need to extract an equality of numbers from \eqref{eqn: categorical gBCF}.
Thus, it becomes a natural and important task to investigate the $\Qell$-vector space in which this formula takes place.
Indeed, as highlighted in \cite[Remark 5.2.3]{toenvezzosi22}, we cannot a priori exclude the possibility that the vector space $\uH^0_{\et}\Bigt{S,\rl_S\bigt{\HH(\sB/A)}}$ is zero (or that the maps $\idcat$ and $\arcat$ are zero).
This would lead to the disappointing conclusion that the formula \eqref{eqn: categorical gBCF} is just the equality $0=0$.

\sssec{} 

In \cite{beraldopippi22, beraldopippi24}, we found a way to avoid this issue, at the cost of imposing the additional hypothesis that $X$ is a Cartier divisor in a smooth $S$-scheme. 
As a consequence, we proved there several instances of \eqref{eqn:gBCC in intro}, including the case of isolated singularities. 
The latter goes under the name of Deligne--Milnor formula, see \cite[Expos\'e XVI]{sga7ii}.

\sssec{} \label{sssec:pure-char}

Furthermore, the above problem does not occur when $A$ has pure characteristic $p\geq 0$. In this case, $\sB$ is a \emph{symmetric} monoidal dg-category; in particular, the canonical dg-functor $\sB \to \HH(\sB/A)$ admits a retraction $\on{m}: \HH(\sB/A)\to \sB$ induced by the multiplication of $\sB$.
Hence, we can push \eqref{eqn: categorical gBCF} along the homomorphism\footnote{When confusion is unlikely, we abusively denote a dg-functor and its $\ell$-adic realization by the same symbol.}
$$
  \on{m}:
  \uH^0_{\et}\Bigt{S,\rl_S\bigt{\HH(\sB/A)}}
    \to 
  \uH^0_{\et}\bigt{S,\rl_S(\sB)}.
$$
By \cite[Proposition 4.27]{brtv18}, the latter vector space equals $\Qell$ and, as proven in \cite[Section 6.3]{beraldopippi24}, this procedure indeed yields \eqref{eqn:gBCC in intro} in pure characteristic.

\begin{rem}
The monoidal structure on $\sB$ fails to be symmetric in mixed characteristic. In particular, the previous method does not make sense and new ideas are required. 
\end{rem}

\subsection{Contents of this work}

The above remark is the starting point of the present paper.

\sssec{}

As mentioned, the goal is to deduce the generalized Bloch conductor conjecture (Conjecture \ref{conj:gBCC}) from its categorical version \eqref{eqn: categorical gBCF}: this amounts to understanding \eqref{eqn:place of cat BCC} and ultimately the dg-category $\HH(\sB/A)$. 
The latter dg-category is a fundamental invariant of the discrete valuation ring $A$, hence we expect that its study will be of interest beyond the remit of the Bloch conductor formulas.

\sssec{}

In practice, we will construct a map 
\begin{equation}
  \oint: \uH^0_{\et}\Bigt{S, \rl_S\bigt{\HH(\sB/A)}}
    \to
  \Qell
\end{equation}
that transforms the known formula \eqref{eqn: categorical gBCF} into the desired formula \eqref{eqn:gBCC in intro}.

\begin{rmk}

We chose the notation $\oint$ guided by some vague analogies: the Drinfeld cocenter is a categorification of Hochschild homology, which gives back differential forms in the commutative case, while $G$ can be regarded as the derived scheme of (infinitesimal and algebraic) loops in $S$ based at $s$.
Therefore, we think of $\oint$ as (a decategorification of) a map which \virg{integrates differential forms on the space of loops}.
\end{rmk}

\sssec{}

To construct $\oint$, we first work categorically. 
The Drinfeld cocenter $\HH(\sB/A)$ is the colimit of a simplicial dg-category,
the cyclic bar construction:
\begin{equation*}
\begin{tikzcd}
    \cdots
    \;\;
    \sB\otimes_A\sB\otimes_A\sB\otimes_A\sB 
    \arrow[r,shift right=3]
    \arrow[r,shift right=1]
    \arrow[r,shift left=1]
    \arrow[r,shift left=3]
    &
    \sB\otimes_A\sB\otimes_A\sB
    \arrow[r,shift right=2]
    \arrow[r]
    \arrow[r,shift left=2]
    &
    \sB\otimes_A\sB
    \arrow[r,shift right=1]
    \arrow[r,shift left=1]
    &
    \sB.
\end{tikzcd}
\end{equation*}
We approximate $\HH(\sB/A)$ by taking the colimits $F_n\HH(\sB/A)$ of the various truncations of this simplicial diagram.
Informally, $F_n\HH(\sB/A)$ can be characterized as the universal dg-category where the $(n+1)$-fold product on $\sB$ becomes invariant with respect to the canonical action of the cyclic group $C_{n+1}$.
This procedure yields a natural filtration
$$
  \HH(\sB/A)
    \simeq 
  \varinjlim_{\bbN}
  F_n\HH(\sB/A).
$$

\sssec{}

We use the latter equivalence to construct a filtered dg-functor $\xi$ out of $\HH(\sB/A)$.
To this end, the key observation is that considering the infinitesimal neighbourhoods
of $s\hto S$ \emph{creates enough space for certain homotopies to exist}. It follows from this observation that the natural target of $\xi$ is the dg-category $\sB(\infty)$, whose definition is sketched below.

\sssec{} 

Fix a uniformizing element $\pi \in A$ once and for all. For $n\geq 1$, set
$k(n) := A/\pi^n$ and consider the closed embedding
$$
  s(n)
    :=
  \Spec{k(n)}
    \hto 
  S
    :=
  \Spec{A}.
$$
Notice that $k(1) = k$ and $s(1) =s$.

\sssec{} 

For each $n \geq 0$, we will construct a dg-functor
$$
  \xi_n:
  F_n\HH(\sB/A)
    \to 
  \sB(n+2),
$$
where the target is a 
certain 
dg-category of singularities associated to the derived scheme 
$$
  \bbA^1_{s(n+2)}
    :=
  s(n+2) \times_{\bbA^1_{s(n+2)}}s(n+2).
$$
The dg-functor
$$
  \xi:
  \HH(\sB/A)
    \to 
  \sB(\infty)
    := 
  \varinjlim_{\bbN} \, \sB(n+2)
$$
will then be defined simply by taking the colimit of the $\xi_n$'s.

\sssec{} 

By the universal property of $F_n\HH(\sB/A)$ mentioned above, the construction of $\xi_n$ amounts to showing that the natural dg-functor $\sB:= \sB(1) \to \sB(n+1)$ coequalizes the actions of the cyclic groups $C_m$ on the $m$-fold product $\sB^{\otimes_Am}\to \sB$ for $m\leq n+1$ (in a homotopy-coherent manner).
This task can be performed geometrically, by first exhibiting a family of coherent homotopies
\begin{equation}\label{eqn: coherent homotopies in intro}
  \Delta^{n}
    \otimes 
  s^{\times_S(n+1)}
    \to 
  s(n+1)
\end{equation}

among the $n+1$ morphisms 
$$
  s^{\times_S(n+1)}
    \xto{\on{rotate}^j}
  s^{\times_S(n+1)}
    \xto{\on{multiply}}
  s
    \xto{\on{include}}
  s(n+1),
$$
where the left arrow cyclically rotates the copies of $s$ by $j$ spots (for $0 \leq j \leq n$).
This is the content of Theorem \ref{thm: the homotopy}, the combinatorial core of this paper.

Moreover, the morphisms \eqref{eqn: coherent homotopies in intro} are compatible with the lower homotopies
in a precise sense, see Theorem \ref{thm: master homotopies}.
This guarantees that the dg-functors $\xi_n$ define an $\bbN$-indexed diagram.

\begin{rmk}
The existence of the homotopies described above shows that, for each $n\geq 0$, there is a morphism of derived $S$-schemes
$$
  \varinjlim_{i=1,\dots,n+1}\bigt{
                                s^{\times_S(n+1)}
                                  \xto
                                  {\pr{i}}s
                                  }
                                  \to
                                  s(n+1).
$$
We suspect this might actually be an equivalence.
In any case, the existence of this map is deeply connected to the Hopf algebroid structure on $G$, which is the main theme of this paper.

In particular, while $G$ admits two distinct $s$-scheme structures (the two projections $G\to s$ are not homotopic in the mixed characteristic situation), it does have a canonical $s(2)$-scheme structure: take $n=1$ in the above formula.
\end{rmk}

\sssec{} 

Having $\xi$ at our disposal, we could apply the decategorification $\rl_S$ and thus study the complex of $\Qell$-vector spaces\footnote{By construction, $\rl_S$ commutes with filtered colimits: this yields the isomorphism in the displayed formula.}
$$
  \uH^*_{\et} 
  \Bigt{ 
        S , \rell_S \bigt{ \sB(\infty)}
       }
    \simeq
  \varinjlim_{\bbN}
  \uH^*_{\et} 
  \Bigt{ 
        S , \rell_S \bigt{ \sB(n+1)}
        }.
$$
However, the following extra categorical step yields a further simplification.

\sssec{} 

Let $\bbA^1_S[-1]:= S \times_{0,\bbA^1_S,0} S $ and observe that $G(n)$ sits in the fiber product
\begin{equation}
\nonumber
\begin{tikzpicture}[scale=1.5]
\node (00) at (0,0) {$S $};
\node (10) at (2,0) {$\bbA^1_S.$ };
\node (01) at (0,1) {$
G(n)$};
\node (11) at (2,1) {$ \bbA^1_S[-1]$}; 
\path[->,font=\scriptsize,>=angle 90]
(00.east) edge node[above] {$\pi^n$}  (10.west); 
\path[->,font=\scriptsize,>=angle 90]
(01.east) edge node[above] {$r(n) $} (11.west); 
\path[->,font=\scriptsize,>=angle 90]
(01.south) edge node[right] {$ $} (00.north);
\path[->,font=\scriptsize,>=angle 90]
(11.south) edge node[right] {$ $} (10.north);
\end{tikzpicture}
\end{equation}
Pushing forward along the arrows $r(n)$ yields a dg-functor
$$
  r(\oo)_*
    := 
  \colim_{\bbN} r(n)_*
    :
  \sB(\infty) 
    \to 
  \Sing(\bbA^1_S[-1])_s,
$$
where $\Sing(\bbA^1_S[-1])_s$ is the singularity category of $\bbA^1_S[-1]$ with support on $s$ (see Section \ref{ssec: singularity categories}).
Combining this with the previously defined $\xi:\HH(\sB/A) \to \sB(\infty)$, we obtain a dg-functor
$$
  \oint:
  \HH(\sB/A) 
    \to 
  \Sing(\bbA^1_S[-1])_s.
$$

\begin{rmk}
Strictly speaking, for the purpose of deducing the generalized Bloch conductor formula \eqref{eqn:gBCC in intro} from \eqref{eqn: categorical gBCF}, it is not necessary to construct the filtered dg-functor $\xi=\varinjlim_n \xi_n$.
In fact, it is possible to construct the dg-functor
$$
  \HH(\sB/A)
    \to
  \Sing(\bbA^1_S[-1])_s
  $$
directly; for this, the homotopies of Theorem \ref{thm: the homotopy} are not necessary.
Nevertheless, we believe that the convolution monoidal structure on $\sB$ is a fundamental object of study in mixed characteristic, and thus it deserves an analysis as fine as possible.
\end{rmk}

\sssec{} 

The $\ell$-adic realization of $\Sing(\bbA^1_S[-1])_s$ is easily computed from \cite[Proposition 4.27]{brtv18}: we obtain a canonical isomorphism
$$
  \uH^0_{\et} 
  \Bigt{ 
        S , \rell_S
                   \bigt{
                         \Sing(\bbA^1_S[-1])_s
                        }
        }
    \simeq 
  \Qell
$$
and thus a homomorphism
$$
  \oint: 
  \uH^0_{\et}\Bigt{
                   S, \rl_S\bigt{
                                 \HH(\sB/A)
                                 }
                    } 
    \to 
  \uH^0_{\et} 
             \Bigt{ 
                   S , \rell_S
                              \bigt{
                                    \Sing(\bbA^1_S[-1])_s
                                    }
                  }
    \simeq 
  \Qell.
$$

\sssec{}

Next, we apply $\oint$ to \eqref{eqn: categorical gBCF} and obtain an equality of $\ell$-adic numbers. It remains to identify these numbers with the invariants appearing in \eqref{eqn:gBCC in intro}.
Of course, we expect the formulas
$$
  \oint 
  \Bigt{ 
        \Bl^{\on{cat}}(X/S)
       }
    = 
  \Bl(X/S),
    \hspace{.4cm}
  \oint \circ \idcat 
    = 
  \id,
    \hspace{.4cm}
  \oint \circ \arcat 
    = 
  \ar.
$$
The last two formulas are the content of Lemma \ref{lem: idcat and arcat}.
The proof of the first formula, to which Sections \ref{sec: categorical intersection theory} and \ref{sec: proof of gBCC} are dedicated, is more involved.

\sssec{}

More precisely, Corollary \ref{cor: intersection with the diagonal} proves that the composition
$$
  \G_0(X\times_SX)
    \to
  \uH^0_{\et}\Bigt{S,\rl_S \bigt{\Sing(X\times_SX)}}
    \xto{\ev^{\HH}}
  \uH^0_{\et}\Bigt{S, \rl_S\bigt{\HH(\sB/A)}}
    \xto{\oint}
  \Qell,
$$
where the left map is 
a shortcut for
$$
  \G_0(X\times_SX)
    \xto{\chern}
  \uH^0_{\et}\Bigt{S,\rl_S \bigt{\Coh(X\times_SX)}}
    \to
  \uH^0_{\et}\Bigt{S,\rl_S \bigt{\Sing(X\times_SX)}},
$$
sends a G-theory class $[M]\in \G_0(X\times_SX)$ to the integer
$$
  [\![M,\Delta_X]\!]_S
    :=
  \sum_{i=0,1}(-1)^{i}\deg \bigq{
                                 \ul \sTor_i^{X\times_SX}(M,\Delta_X)
                                 } 
    \in 
  \bbZ 
    \subseteq 
  \Qell.
$$

\sssec{} 
Here $\ul \sTor_i^{X\times_SX}(M,\Delta_X)$ denotes the $i^{th}$ stable Tor sheaf, i.e. the sheaf
$$
  \ccH^{-2n-i}\bigt{
                    M
                    \otimes^{\bbL}_{X\times_SX}
                    \Delta_X
                    }, 
  \;\;\;\; 
  n\gg 0.
$$
It is known after Kato--Saito that 
$
  \ccH^{-2n-i}\bigt{
                    M
                    \otimes^{\bbL}_{X\times_SX}
                    \Delta_X
                    }
$ is a coherent $\ccO_Z$-module for $n\gg0$ and that its G-theory class does not depend on $n$ (\cite[Theorem 3.2.1]{katosaito04}).
Furthermore,  \cite[Corollary 3.4.5]{katosaito04} guarantees that 
$$
  \Bl(X/S)
    =
  \sum_{i=0,1}(-1)^{i}\deg \bigq{
                                \ul \sTor_i^{X\times_SX}(\Delta_X,\Delta_X)
                                }.
$$

\sssec{} 

To summarize the above discussion, consider the \emph{integration dg-functor}
$$
  \int_{X/S}:
  \Sing(X \times_S X)
    \xto{\ev^{\HH}}
  \HH(\sB/A)
    \xto{\oint}
  \Sing(\bbA^1_S[-1])_s
$$
discussed above and introduced properly in Section \ref{sec: categorical intersection theory}. 
On the one hand, it is a categorification of Kato--Saito's localized intersection product; on the other hand, by definition, it involves \TV's categorical intersection product. As explained above, the former interpretation gives us access to the number $\Bl(X/S)$, while the latter lets us employ the categorical methods of \cite{beraldopippi24}. The interplay between these two points of view is what ultimately yields the proof of Conjecture \ref{conj:gBCC}.

\begin{rmk}
As observed by Toen, much of the reasoning in this paper does not require $A$ to be a DVR: the important points work essentially unchanged for $S$ a regular scheme and $s = Z(f) \subseteq S$ the zero locus of a regular section $f \in \uH^0(S, \ccO_S)$. We have not investigated this research direction.
\end{rmk}

\subsection{Outline}

This paper is organized as follows.
\begin{itemize}
    \item In Section \ref{sec: preliminaries and notation}, we first fix some notation related to simplicial sets and simplicial algebras, which play a crucial technical role for the main computations performed afterwards. 
    Then we recall the needed notions in the theory of dg-categories, including their realizations.
    
    \item In Section \ref{sec: differential forms on two-periodic complexes}, we introduce our main object of study, the monoidal dg-category $\sB$, and we consider the standard filtration on its cyclic bar complex whose colimit computes the Drinfeld cocenter $\HH(\sB/A)$.
    Furthermore, we describe the nodes of the cyclic bar complex in terms of certain singularity categories.
    In a brief digression, we explain why the monoidal structure is symmetric in the pure-characteristic case.
    
    \item
    In Section \ref{sec: combinatorial preliminaries}, we provide cofibrant models for the infinitesimal neighbourhoods of $s \hto S$ and $G \hto S$. 
    
    \item Section \ref{sec: the E_2 structure on G} lies at the very heart of this paper. Here we produce the coherent higher homotopies \eqref{eqn: coherent homotopies in intro} we referred to in the introduction.
    We do this by working with strict models: we need to produce some explicit morphisms of simplicial rings in the ordinary category of simplicial $A$-algebras, endowed with the usual model structure introduced by Quillen.
    
    \item In Section \ref{sec: integrating differential forms on two-periodic complexes}, we construct the dg-functor $\oint$ that \virg{integrates} differential forms on $\sB$. 
    Relying on Section \ref{sec: the E_2 structure on G}, we proceed to construct a dg-functor for each level of the filtration of the cyclic bar complex.
    We use this construction to partially answer a question of \TV{}, see \cite[Remark 5.2.3]{toenvezzosi22}.
    
    \item In Section \ref{sec: categorical intersection theory} we study the composition $\int_{X/S} := \oint \circ \ev^{\HH}$, our categorical localized intersection product, and we write it in two other equivalent ways.
    
    \item Finally, in Section \ref{sec: proof of gBCC}, we show that the dg-functor $\oint \circ \ev^{\HH}$ categorifies the localized intersection product introduced by Kato--Saito. As a consequence, we deduce the generalized Bloch conductor formula from its categorical version.
    
\end{itemize}

\subsection*{Acknowledgements}

We thank Bertrand Toen and Gabriele Vezzosi for many insightful discussions related to this work and for their support over the years.
We are also grateful to Benjamin Hennion, Valerio Melani, Mauro Porta and Marco Robalo for numerous valuable conversations.
This project has received funding from PEPS JCJC 2024 and the ANR project DAG-Arts.


\section{Preliminaries and notation}\label{sec: preliminaries and notation}

In this section, we introduce some notions that will play important roles in the main body of this paper.
First we recall some facts about simplicial sets, simplicial rings and their homotopy theory.
This gives us the chance to fix some of the notation appearing later in the paper.
Next we recall some basic facts concerning the Morita theory of dg-categories and then the definitions of some dg-categories of sheaves on derived schemes.
We also introduce some non-standard dg-categories associated to fiber products of derived schemes: these dg-categories will naturally appear in the study of the Drinfeld cocenter of $\sB$.
Then we introduce eventually two-periodic complexes and relate them to (coherent) matrix factorizations.
Finally, we fix the notation for the motivic and $\ell$-adic realizations of dg-categories.

\sssec*{Convention}

We will freely use the language of $\oo$-categories as developed in \cite{luriehtt, lurieha}, using for the most part the same notations.

\subsection{Simplicial sets}

\sssec{The simplex category}

For $n\geq 0$, denote by $[n]$ the set $\{0,\dots,n\}$, considered with its standard linear order.
The simplex category $\Delta$ is the (ordinary) category whose objects are the linearly ordered finite sets $[n]$ (for $n\geq 0$) and whose morphisms are order-preserving functions.
Among order-preserving functions, particularly relevant are the functions
$$
  \delta^n_i: 
  [n-1]
    \to 
  [n] 
  \;\;\;\; 
  0
    \leq 
  i 
    \leq 
  n,
$$
$$
  \sigma^n_i: 
  [n+1]
    \to 
  [n] 
  \;\;\;\; 
  0 
    \leq 
  i 
    \leq 
  n.
$$
The function 
$
  \delta^n_i: 
  [n-1]
    \to 
  [n]
$ 
is uniquely determined by the property that the image is 
$$
  \{
    0,\dots,i-1,i+1,\dots,n
  \}
  \subseteq 
  [n],
$$
while $\sigma^n_i: [n+1]\to [n]$
by the fact that it is surjective and the element $i \in \{0,\dots,n\}$ has two pre-images.
The morphisms $\delta^n_i$ and $\sigma^m_j$ satisfy a set of relations (depending on $n,m,i$ and $j$):
\begin{equation*}
    \begin{cases}
    \delta_j^{n+1}\circ \delta_i^{n}=\delta^{n+1}_{i}\circ \delta^{n}_{j-1} & \text{if } i<j;
    \\
    \sigma_j^{n-1}\circ \delta_i^n= \delta^{n-1}_i \circ\sigma^{n-2}_{j-1} & \text{if } i<j;
    \\
    \sigma_i^{n-1} \circ \delta_i^n=\id = \sigma^{n-1}_i \circ  \delta^{n}_{i+1} &
    \\
    \sigma_j^{n-1}\circ\delta_i^n=\delta^{n-1}_{i-1}  \circ \sigma^{n-2}_{j}& \text{if } i>j+1;
    \\
    \sigma_j^{n-1} \circ \sigma_i^n=\sigma^{n-1}_i \circ \sigma^{n}_{j+1}& \text{if } i\leq j.
    \end{cases}
\end{equation*}
Their importance stems from the fact that every map in $\Delta$ can be written as a composition of such morphisms.

\sssec{Simplicial sets}

A simplicial set $X$ is a functor
$$
  X: 
  \Delta^{\op}
    \to 
  \Set
$$
$$
  (f:[p]\to [q])
    \mapsto 
  (f^*: X_q \to X_p).
$$
This is equivalent to the datum of a collection of sets 
$
  \{X_n\}_{n\geq 0}
$ 
together with a collection of morphisms
$$
  \text{(face maps)} 
    \;\;\;\;\;\;\;\;\;\;\;\;\;\;\; 
  d^p_i
    =
  (\delta^p_i)^*:
  X_p 
    \to 
  X_{p-1}
    \;\;\;\; 
  0\leq i \leq p,
$$
$$
  \text{(degeneracy maps)} 
    \;\;\;\; 
  s^p_i
    =
  (\sigma^p_i)^*:
  X_p 
    \to
  X_{p+1}
    \;\;\;\; 
  0\leq i \leq p,
$$
which verify relations analogous to those satisfied by the $\delta^p_i$ and $\sigma^q_j$'s:
\begin{equation*}
    \begin{cases}
    d_i^p\circ d_j^{p+1}=d^p_{j-1}\circ d^{p-1}_i & \text{if } i<j;
    \\
    d_i^p\circ s_j^{p-1}=s^{p-2}_{j-1}\circ d^{p-1}_i & \text{if } i<j;
    \\
    d_i^p\circ s_i^{p-1}=\id = d^{p}_{i+1}\circ s^{p-1}_i &
    \\
    d_i^p\circ s_j^{p-1}=s^{p-2}_{j}\circ d^{p-1}_{i-1} & \text{if } i>j+1;
    \\
    s_i^p\circ s_j^{p-1}=s^{p}_{j+1}\circ s^{p-1}_i & \text{if } i\leq j.
    \end{cases}
\end{equation*}
We will write $\sSet$ for the category of simplicial sets.

\sssec{} 

Consider the Yoneda embedding $\scrY:\Delta \to \sSet$.
Then the simplicial sets
$$
  \Delta^n
    := 
  \scrY([n])
$$
are called simplices.

The morphism $\delta^n_i:[n-1]\to [n]$ corresponds to the inclusion of the $i^{th}$ face $\Delta^{n-1}\hto \Delta^n$.
We set
$$
  \partial \Delta^n
    := 
  \underbrace{
              \Delta^{n-1}\amalg_{\Delta^{n-2}}\Delta^{n-1}\dots \amalg_{\Delta^{n-2}}\Delta^{n-1}
              }_{
                 n+1 \;\; \text{ copies}
                 } 
    \hto 
  \Delta^n
$$
where each copy of $\Delta^{n-1}$ corresponds to one of the $n+1$ faces of dimension $n-1$ in $\Delta^n$.

The simplicial set $\Lambda^n_i$ ($0\leq i \leq n$) is obtained by gluing all the faces of dimension $n-1$ in $\Delta^n$ except for the $i^{th}$ one, corresponding to the face in $\Delta^n$ opposite to the vertex $\{i\}$.

\sssec{Weak-equivalences and fibrations}

By a theorem due to Quillen, the category of simplicial sets $\sSet$ admits a simplicial model structure called \emph{classical (or Quillen) model structure} (see \cite{quillen67}).
This model structure is cofibrantly generated (\cite{hovey99}).
The set of generating cofibrations is $\{\partial \Delta^n \hto \Delta^n\}_{n \geq 0}$ and the set of generating trivial cofibrations is $\{ \Lambda^n_i\hto \Delta^n\}_{n\geq 0,\, 0\leq i \leq n}$.

The $\oo$-category underlying this simplicial model category is $\ccS$, the $\oo$-category of spaces.
See \cite{luriehtt} for details.

\subsection{Derived algebraic geometry}

We refer to \cite{luriedag1,toenvezzosi08,toen14} for more details on derived algebraic geometry.
Here we just recall some facts about the model category of simplicial rings, which provides a model for the $\oo$-category of derived affine schemes.

\sssec{Simplicial \texorpdfstring{$A$}{A}-algebras}

Let $\CAlg_{A/}$ denote the (ordinary) category of commutative $A$-algebras.
A simplicial $A$-algebra $R$ is a functor
$$
  R: 
  \Delta^{\op}
    \to 
  \CAlg_{A/}
$$
$$
  \bigt{ 
        f:[p]\to[q]
        }
    \mapsto
  \bigt{ 
        f^*:R_q \to R_p 
        }.
$$
Just like simplicial sets, such a functor is determined by a sequence of  $A$-algebras $\{R_p\}_{p\geq 0}$ and by the face and degeneracy maps:
$$
  d^p_i
    :=
  (\delta^p_i)^*:
  R_{p}
    \to 
  R_{p-1}
    \;\;\;\; 
  (i=0,\dots,p),
$$
$$
  s^p_i
    :=
  (\sigma^p_i)^*:
  R_{p-1}
    \to 
  R_{p}
    \;\;\;\; 
  (i=0,\dots,p-1).
$$
These maps satisfy the same simplicial identities that we made explicit for simplicial sets.

We will write $\sCAlg_{A/}$ for the category of simplicial commutative $A$-algebras.

\sssec{Simplicial enrichment}

For $R\in \sCAlg_{A/}$ and $K \in \sSet$, one can define the simplicial $A$-algebra
$$
  K\otimes R: 
  \Delta^\op 
    \to 
  \CAlg_{A/}
$$
$$
  [p]
    \mapsto 
  \bigotimes_{K_p}R_p. 
$$
Any morphism $f:[p]\to [q] \in \Delta$ induces a morphism of commutative rings
$$
  f^*: 
  (K\otimes R)_q
    =
  \bigotimes_{K_q}R_q
    \to
  \bigotimes_{K_p}R_p=(K\otimes R)_p
$$
which corresponds to the coproduct of the maps
$$
\underbrace{R_q}_{j^{th} \text{ copy}} \to \underbrace{R_p}_{f(j)^{th} \text{ copy}} 
$$
induced by the simplicial structure of $R$.

\begin{rmk}

One uses the fact that $\sCAlg_{A/}$ is tensored over $\sSet$ to enrich the former category over the latter. 
Namely, for $R,R' \in \sCAlg_{A/}$, the \emph{simplicial mapping space} 
$$
  \Map_{\sCAlg_{A/}}(R,R'): 
  \Delta^{\op}
    \to 
  \Set
$$
is the simplicial set defined on objects as
$$
  [n] 
    \mapsto 
  \Hom_{\sCAlg_{A/}}(\Delta^n \otimes R,R')
$$
and in the obvious way on morphisms.
\end{rmk}

\sssec{Weak-equivalences and fibrations}

Every simplicial $A$-algebra
has an underlying simplicial set, which is always cofibrant.

A morphism $f:R\to R'$
in $\sCAlg_A$
is a \emph{weak-equivalence} if it induces a weak-equivalence at the level of the underlying simplicial sets.
In other words, $f:R\to R'$ is a weak equivalence if it induces an isomorphism of $A$-algebras
$$
  \pi_0(R)
    \to 
  \pi_0(R')
$$
and, for every $r\in R$ and for every $i\geq 1$, an isomorphism of abelian groups
$$
  \pi_i(R,r)
    \to 
  \pi_i(R',f(r)).
$$
In fact, it suffices to verify the latter condition for $r=0$.

A morphism $f:R\to R'$ is a \emph{fibration} if it induces a fibration at the level of the underlying simplicial sets.

\sssec{Homotopy theory of simplicial rings}

The starting point of derived algebraic geometry is the following result, due to Quillen:
the category $\sCAlg_{A/}$ admits the structure of a simplicial model category where the classes of weak-equivalences and fibrations are those specified above.

\sssec{}

The $\oo$-category of \emph{derived affine $S$-schemes}, to be denoted $\dSch_{/S}^{\on{aff}}$, is by definition the opposite of the $\oo$-localization $\sCAlg_{A/}[W_{\on{we}}^{-1}]$. The object corresponding to a simplicial $A$-algebra $R$ is called $\Spec{R}$. The adjunction between $A$-algebras and simplicial $A$-algebras induces an adjunction
$$
  \iota: 
  \Sch_{/S}^{\on{aff}}
    \leftrightarrows 
  \dSch_{/S}^{\on{aff}}: 
  \pi_0,
$$
with $\iota$ fully faithful. Fiber products exist in $\dSch_{/S}^{\on{aff}}$ and correspond to derived tensor products.

\sssec{Derived schemes} 

More generally, a derived scheme over $S$ is a pair $(Y,\ccO_Y)$ where $Y$ is a topological space and $\ccO_Y$ is a sheaf of simplicial $A$-algebras (up to weak-equivalences) which is locally equivalent to $\bigt{\Spec{\pi_0(R)},R}$ for some simplicial commutative $A$-algebra $R$ and such that $\pi_i\bigt{\ccO_Y}$ is a quasi-coherent $\pi_0\bigt{\ccO_Y}$-module for each $i\geq 1$.

Derived $S$-schemes are naturally organized in an $\oo$-category $\dSch_{/S}$.

\sssec{} 

Just as for affine schemes, the assignment $(Y,\ccO_Y)\mapsto (Y,\pi_0(Y))$ defines an $\oo$-functor
$$
  \dSch_{/S}
    \to 
  \Sch_{/S}
$$
which is right adjoint to a fully-faithful embedding
of the ordinary category of $S$-schemes in $\dSch_{/S}$.

\subsection{Dg-categories}

\sssec{}

An $A$-linear dg-category is a category enriched over the category of $A$-dg-modules.
In this paper, every dg-category is $A$-linear, unless stated otherwise.

\sssec{}

We will work with (small) dg-categories up to Morita equivalences.
By this we mean that we will work in the $\oo$-category $\dgCat_A$, which is the $\oo$-localization of the ordinary category of small dg-categories against the class of Morita equivalences.
We will say that a dg-functor is \emph{Karoubi essentially surjective} if its essential image generates the target dg-category under finite (homotopy) colimits and (homotopy) retracts. 
We refer to \cite{toen07} for more details.

\sssec{}

It is proven in \cite{toen07} that the $\dgCat_A$ admits an internal theory of localizations and quotients. 
Let $\sT$ be a dg-category and let $\on{W}$ denote a class of morphisms in $\sT$ which is stable under composition and that contains the identities.
Denote by $\Map_{\dgCat_A}^{\on{W}}(\sT,\sU)\subseteq \Map_{\dgCat_A}(\sT,\sU)$  the full sub-groupoid consisting of those dg-functors $\sT \to \sU$ that send every morphism in $\on{W}$ to an equivalence in $\sU$.
It is proved in \emph{loc. cit.} that the $\oo$-functor
$$
  \dgCat_A 
    \to 
  \ccS
$$
$$
  \sU 
    \mapsto 
  \Map_{\dgCat_A}^{\on{W}}(\sT,\sU)
$$
is representable. 
We will denote by $L_{\on{W}}(\sT)$ the dg-category which represents this functor:
$$
  \Map_{\dgCat_A}\bigt{L_{\on{W}}(\sT),\sU}
    \simeq
  \Map_{\dgCat_A}^{\on{W}}(\sT,\sU).
$$
The existence of dg-localizations in $\dgCat_A$ implies the existence of dg-quotients as well.
Let $\sT_0\subseteq \sT$ denote a full embedding in $\dgCat_A$. 
We can consider the class $\on{W}_{\sT_0}$ of morphisms in $\sT$ whose (homotopy) cone belongs to $\sT_0$. 
Then the dg-quotient of $\sT_0\subseteq \sT$ is defined as
$$
  \sT/\sT_0 
    := 
  L_{\on{W}_{\sT_0}}(\sT).
$$
We will say that
$$
  \sT_0 
    \hto 
  \sT 
    \to 
  \sT/\sT_0
$$
is a localization sequence.

\sssec{}

Finally recall that there is a non-full faithful embedding
$$
  \widehat{(-)}:
  \dgCat_A 
    \hto  
  \dgCAT_A,
$$
induced by ind-completion,
where $\dgCAT_A$ denotes the $\oo$-category of compactly generated presentable dg-categories and continuous dg-functors. Dg-localizations and dg-quotients exist in $\dgCAT_A$ as well.

\subsection{Singularity categories}\label{ssec: singularity categories}

\sssec{}

Let $\Spec{R}\in \dSch_{/S}^{\on{aff}}$.
Consider the dg-algebra $N(R)$ which corresponds to $R$ via the Dold-Kan equivalence. By definition, the derived category of $\Spec{R}$ is
$$
  \QCoh(X)
    :=
  \Mod_{N(R)} 
    \in 
  \dgCAT_A.
$$
This is the dg-category of $N(R)$-dg-modules up to quasi-isomorphisms, that is, the localization of the ordinary category of $N(R)$-dg-modules along the class of quasi-isomorphisms.

For an arbitrary derived scheme (or stack) $X$, we set
$$
  \QCoh(X)
    := 
  \varprojlim_{\Spec{R}\to X}
  \QCoh(\Spec{R}) 
    \in 
  \dgCAT_A.
$$

\sssec{}

Assume that $X$ is a bounded noetherian derived $S$-scheme. 
We consider the full subcategories 
$$
  \Perf(X)
    \subseteq 
  \Coh(X)
    \subseteq 
  \QCoh(X).
$$
Here $\Perf(X)$ denotes the dg-category of perfect complexes and $\Coh(X)$ that of quasi-coherent complexes with coherent total cohomology.
Under the above hypotheses on $X$, one defines the singularity category of $X$ as the dg-quotient
$$
  \Sing(X)
    := 
  \frac{\Coh(X)}{\Perf(X)}.
$$
We refer the reader to \cite[\S 2]{brtv18} for details on singularity categories.

\sssec{} 

If $Y\subseteq X$ is a closed embedding, we write $\Perf(X)_Y$ (respectively: $\Coh(X)_Y$, $\Sing(X)_Y$, $\QCoh(X)_Y$) for the subcategory of $\Perf(X)$ (respectively: $\Coh(X)$, $\Sing(X)$, $\QCoh(X)$) consisting of objects supported on $Y$.
Notice that
$$
  \Sing(X)_Y 
    \simeq 
  \frac{\Coh(X)_Y}{\Perf(X)_Y}.
$$

\subsection{External products}
\label{ssec: categories external products}

In the main body of the paper, we will need to deal with certain dg-categories Karoubi generated by external products of coherent sheaves. 
We record their definitions and the main properties here.

\begin{defn}
Let $Y_1, \dots,Y_m$ be $S$-schemes of finite type.
Denote by
$$
  \Cohext(\times_SY_i)
    \subseteq 
  \QCoh(\times_SY_i)
$$
the full subcategory Karoubi-generated by the external tensor products $\boxtimes_SE_i$, where $E_i \in \Coh(Y_i)$ for every $i$.
Denote by
$$
  \Perfext(\times_SY_i)
    \subseteq 
  \Cohext(\times_SY_i)
$$
the full subcategory Karoubi-generated by the external tensor products $\boxtimes_SE_i$, where at least one of the $E_i$'s is perfect.
Finally, set
$$
  \Singext(\times_SY_i)
    :=
  \frac{\Cohext(\times_SY_i)}{\Perfext(\times_SY_i)}.
$$
\end{defn}

\sssec{}

Now assume we have other $S$-schemes $Z_1, \ldots, Z_m$ and morphisms $f_i: Y_i \to Z_i$ of $S$-schemes.
Consider the product morphism
$$
  F
    := 
  f_1 \times \dots \times f_m:
  Y_1 \times_S \cdots \times_S Y_m
    \to 
  Z_1 \times_S \cdots \times_S Z_m.
$$
The following useful formula is a pleasant exercise with base-change:
\begin{equation}\label{eqn:push of boxtimes}
  F_*
  \bigt{ 
        E_1 \boxtimes_S \cdots \boxtimes_S E_m
       }
    \simeq
  f_{1*} (E_1) \boxtimes_S \cdots \boxtimes_S 
  f_{m*} (E_m).
\end{equation}

\begin{lem} \label{lem:Cohext contained in Coh}

With the above notation, suppose that each $Z_i$ is a regular $S$-scheme
and that each $f_i$ is a quasi-smooth closed embedding. 
Then the inclusion
$$
  \Cohext(\times_SY_i)
    \subseteq 
  \QCoh(\times_SY_i)
$$
factors through $\Coh(\times_SY_i)$.
\end{lem}

\begin{proof}
Since $F$ is a quasi-smooth closed embedding, the dg-functor $F_*$ detects coherency. 
Hence, it suffices to prove that the right-hand side of \eqref{eqn:push of boxtimes} is coherent. 
We see that this object is perfect since each $f_{i*} (E_i)$ is coherent and hence perfect ($Z_i$ is regular). 
The regularity assumption finally ensures that the derived $S$-scheme $\times_S Z_i$ is bounded, and so $\Perf(\times_S Z_i)$ is contained in $\Coh(\times_S Z_i)$.
\end{proof}

\sssec{Variants with support}

Suppose that, for each $i=1,\dots m$, we are given a closed subset $W_i\hto Y_i$.
We denote by
$$
  \Cohext(\times_SY_i)_{\times_SW_i}
    \subseteq 
  \QCoh(\times_SY_i)
$$
the full subcategory Karoubi-generated by the external tensor products $\boxtimes_SE_i$, where $E_i \in \Coh(Y_i)_{W_i}$.
Similarly, we denote by 
$$
  \Perfext(\times_SY_i)_{\times_SW_i} 
    \subseteq 
  \Cohext(\times_SY_i)_{\times_SW_i}
$$
the full subcategory Karoubi-generated by external tensor products $\boxtimes_S E_i$, with $E_i \in \Coh(Y_i)_{W_i}$ and at least of them is perfect. 
Notice that 
$
  \Perfext(\times_SY_i)_{\times_SW_i} 
    \subseteq 
  \Perfext(\times_SY_i)
$ 
and 
$
  \Cohext(\times_SY_i)_{\times_SW_i} 
    \subseteq 
  \Cohext(\times_SY_i)
$.
We will use the following notation:
$$
  \Singext(\times_SY_i)_{\times_SW_i} 
    :=
  \frac{\Cohext(\times_SY_i)_{\times_SW_i}}{\Perfext(\times_SY_i)_{\times_SW_i}}.
$$

\subsection{Eventually two-periodic complexes}\label{ssec: etp complexes}

\sssec{}

We now recall the notion of \emph{eventually two-periodicity}.

\sssec{}

Let $Y$ be a noetherian $S$-scheme and $u$ a free variable in cohomological degree $2$. Consider the algebra object
$$
  \ccO_Y[u]
    :=
  \Sym_{\ccO_Y}(\ccO_Y[-2])
$$
in $\QCoh(Y)$ and denote by
$$
  \Mod_{\ccO_Y[u]}
    := 
  \Mod_{\ccO_Y[u]}\bigt{\QCoh(Y)}
$$
the dg-category of $\ccO_Y[u]$-modules. Recall the standard adjunction
$$
  \ccO_Y[u] \otimes_{\ccO_Y}-: 
  \QCoh(Y)
    \leftrightarrows
  \Mod_{\ccO_Y[u]}
  :\on{forget}.
$$

\sssec{}

For a full subcategory $\ccC \subseteq \QCoh(Y)$, 
we adopt the notation
$$
  \Mod_{\ccO_Y[u]}\bigt{\ccC}
    := 
  \varprojlim
  \Bigt{
        \Mod_{\ccO_Y[u]}\xto{\on{forget}}\QCoh(Y)\hookleftarrow \ccC
        }.
$$

\begin{defn}

Let $\ccC \subseteq \QCoh(Y)$ be a full subcategory and let $E \in \Mod_{\ccO_Y[u]}(\ccC)$.
We say that $E$ is \emph{eventually two-periodic} if
$$
  \Fib (E\xto{u} E[2]) 
    \in
  \Coh(Y)\cap \ccC.
$$
We denote by
$$
  \Mod^{\etp}_{\ccO_Y[u]}(\ccC)
    \subseteq 
  \Mod_{\ccO_Y[u]}(\ccC)
$$
the full subcategory spanned by eventually two-periodic complexes.
\end{defn}

\begin{rmk}\label{rmk: Modetp Karoubi complete}
Assume that $\ccC \subseteq \QCoh(Y)$ is a Karoubi-complete full subcategory.
Then notice that 
$$
  \Mod^{\etp}_{\ccO_Y[u]}(\ccC) 
    \subseteq 
  \Mod_{\ccO_Y[u]}
$$
is a Karoubi-complete dg-subcategory.
Let us verify the closure under retracts.
Let $F \in \Mod^{\etp}_{\ccO_Y[u]}(\ccC)$ and
assume that $E \in \Mod_{\ccO_Y[u]}$ is a retract of $F$.
Since $\ccC$ is Karoubi-complete, we deduce that $E\in \Mod_{\ccO_Y}(\ccC)$.
Moreover, we have decompositions
$$
  \ccH^a(F)
    \simeq 
  \ccH^a(E) \oplus \ccH^a(F/E)
$$
compatible with the morphism
$
  \ccH^a(F) 
    \to
  \ccH^{a+2}(F)
$
induced by $u$.
\end{rmk}

\begin{rmk}

Let $\ccC \subseteq \QCoh(Y)$ be a Karoubi-complete stable full subcategory and suppose that
$
  \ccC 
    \subseteq
  \Coh(Y).
$
Then the obvious inclusion is an equivalence:
$$
  \Mod_{\ccO_Y[u]}^{\etp}(\ccC)
    \simeq
  \Mod_{\ccO_Y[u]}(\ccC).
$$
\end{rmk}

\begin{lem}\label{lem: Coh(A1[-1]) vs Cohminus etp}
Let $\bbA^1_Y[-1]:=Y\times_{0,\bbA^1_Y,0}Y$.
Then there is an equivalence
$$
  \Coh(\bbA^1_Y[-1])
    \simeq
  \Mod_{\ccO_Y[u]}^{\etp}\bigt{\Cohminus(Y)}.
$$
\end{lem}

\begin{proof}

We will construct two mutually inverse functors.

\sssec*{Construction first functor}

The derived scheme $\bbA^1_Y[-1]$ is affine over $Y$ and corresponds to the sheaf of simplicial algebras $\ccO_Y[\eps]$, where $\eps$ is a free generator of degree $1$ mapped to $0$ by the differential. 
Consider the dg-functor
$$
  \ccO_Y\otimes_{\ccO_Y[\eps]}-:
  \Coh(\bbA^1_Y[-1])
    \to
  \Cohminus(Y).
$$
By \cite[Lemma 2.39]{brtv18} we have that
$$
  \Hom_{\bbA^1_S[-1]}(\ccO_S,\ccO_S)
    \simeq
  \ccO_S[u].
$$
Since $\bbA^1_Y[-1]\simeq \bbA^1_S[-1]\times_S Y$, we see that the functor above factors as
$$
  \ccO_Y\otimes_{\ccO_Y[\eps]}-:
  \Coh(\bbA^1_Y[-1])
    \to
  \Mod_{\ccO_Y[u]}\bigt{\Cohminus(Y)}.
$$

Moreover, for every $M \in \Coh(\bbA^1_Y[-1])$, we have that $\ccO_Y\otimes_{\ccO_Y[\eps]}M$ is eventually two-periodic.
To see this, we notice that
$$
  \ccO_Y
    \simeq
  \cdots
  \xto{\eps}
  \ccO_Y[\eps][4]
  \xto{\eps}
  \ccO_Y[\eps][2]
  \xto{\eps}
  \ccO_Y[\eps],
$$
so that
$$
  \ccO_Y\otimes_{\ccO_Y[\eps]}M
    \simeq
  \cdots
  \xto{\eps}
  M[4]
  \xto{\eps}
  M[2]
  \xto{\eps}
  M.
$$
Then the action of $u$ is induced by
\begin{equation*}
    \begin{tikzcd}
      \cdots
      \rar["\eps"]
      &
      M[4]
      \dar["\id"]
      \rar["\eps"]
      &
      M[2]
      \dar["\id"]
      \rar["\eps"]
      &
      M
      \\
      \cdots
      \rar["\eps"]
      &
      M[4]
      \rar["\eps"]
      &
      M[2].
    \end{tikzcd}
\end{equation*}
Since $M\in \Coh(\bbA^1_Y[-1])$ we see that $\ccO_Y\otimes_{\ccO_Y[\eps]}M$ is eventually two-periodic.
Hence, we have a dg-functor
$$
  \ccO_Y\otimes_{\ccO_Y[\eps]}-:
  \Coh(\bbA^1_Y[-1])
    \to
  \Mod_{\ccO_Y[u]}^{\etp}\bigt{\Cohminus(Y)}.
$$

\sssec*{Construction of the second functor}

Consider the dg-functor
$$ 
  \ccO_Y\otimes_{\ccO_Y[u]}-:
  \Mod_{\ccO_Y[u]}^{\etp}\bigt{\Cohminus(Y)}
  \to
  \QCoh(Y)
$$
Using once again \cite[Lemma 2.39]{brtv18}, we see that this factors as
$$
  \ccO_Y\otimes_{\ccO_Y[u]}-:
  \Mod_{\ccO_Y[u]}^{\etp}\bigt{\Cohminus(Y)}
    \to
  \QCoh(\bbA^1_Y[-1]).
$$
Moreover, the eventually-two periodic condition guarantees that we land in 
$$
  \Coh(\bbA^1_Y[-1])
    \subseteq 
  \QCoh(\bbA^1_Y[-1]).
$$

\sssec*{Equivalence}

To conclude the proof, notice the chain of canonical equivalences
$$
  \ccO_Y\otimes_{\ccO_Y[u]}(\ccO_Y\otimes_{\ccO_Y[\eps]}M) 
    \simeq
  (\ccO_Y\otimes_{\ccO_Y[u]}\ccO_Y)\otimes_{\ccO_Y[\eps]}M
    \simeq 
  \ccO_Y[\eps]\otimes_{\ccO_Y[\eps]}M
    \simeq
  M.
$$
In particular, the composition
$$
  \Coh(\bbA^1_Y[-1])
    \xto{\ccO_Y\otimes_{\ccO_Y[\eps]}-}
  \Mod_{\ccO_Y[u]}^{\etp}\bigt{\Cohminus(Y)}
    \xto{\ccO_Y\otimes_{\ccO_Y[u]}-}
  \Coh(\bbA^1_Y[-1])
$$
is equivalent to the identity functor. On the other hand, 
$$
  \ccO_Y\otimes_{\ccO_Y[\eps]}(\ccO_Y\otimes_{\ccO_Y[u]}M) 
    \simeq
  (\ccO_Y\otimes_{\ccO_Y[\eps]}\ccO_Y)\otimes_{\ccO_Y[u]}M
    \simeq 
  \ccO_Y[u^{-1}]\otimes_{\ccO_Y[u]}M.
$$
Now notice that
$$
  \ccO_Y[u^{-1}]
    \simeq
  \varinjlim_{n\geq 0}\Fib(\ccO_Y[u]\xto{u^n}\ccO_Y[u][2n]),
$$
so that
$$
  \ccO_Y[u^{-1}]\otimes_{\ccO_Y[u]}M
    \simeq
  \varinjlim_{n\geq 0} \Fib (M\xto{u^n}M[2n])
    \simeq
  M.
$$
Therefore, the composition
$$
  \Mod_{\ccO_Y[u]}^{\etp}\bigt{\Cohminus(Y)}
    \xto{\ccO_Y\otimes_{\ccO_Y[u]}-}
  \Coh(\bbA^1_Y[-1])
    \xto{\ccO_Y\otimes_{\ccO_Y[\eps]}-}
  \Mod_{\ccO_Y[u]}^{\etp}\bigt{\Cohminus(Y)}
$$
is equivalent to the identity functor as well.
\end{proof}

\sssec{}\label{sssec: pr^*Coh}

Consider the full subcategory
$$
  \Angles{\pr{}^*\Coh(Y)}
    \subseteq 
  \Coh(\bbA^1_Y[-1])
$$
Karoubi-generated by the essential image of the functor
$$
  \pr{}^*:
  \Coh(Y)
    \to
  \Coh(\bbA^1_Y[-1]),
$$
where $\pr{}:\bbA^1_Y[-1]=Y\times_{0,\bbA^1_Y,0}Y\to Y$ is the canonical projection.

\sssec{}

Our next task is to identify the full subcategory of $\Mod_{\ccO_Y[u]}^{\etp}\bigt{\Cohminus(U)}$ which corresponds to 
$$
  \Angles{\pr{}^*\Coh(Y)}
    \subseteq 
  \Coh(\bbA^1_Y[-1])
$$
under the equivalence of Lemma \ref{lem: Coh(A1[-1]) vs Cohminus etp}.

For this purpose, we consider the full subcategory
$$
  \Mod_{\ccO_Y[u]}^{\etp}\bigt{\Cohminus(U)}^{u-\on{tors}}
    \subseteq 
  \Mod_{\ccO_Y[u]}^{\etp}\bigt{\Cohminus(U)}
$$
of $u$-torsion modules, i.e. objects $M \in \Mod_{\ccO_Y[u]}^{\etp}\bigt{\Cohminus(U)}$ such that there exists some $n\geq 0$ with
$$
  u^n \sim 0:
  M
    \to
  M[2n].
$$

Notice that this condition joint with the etp one immediately implies that the underlying complex of $\ccO_Y$-modules of $M$ lies in $\Coh(Y)$:
$$
  \Mod_{\ccO_Y[u]}^{\etp}\bigt{\Cohminus(Y)}^{u-\on{tors}}
    \simeq
  \Mod_{\ccO_Y[u]}\bigt{\Coh(Y)}^{u-\on{tors}}.
$$

\begin{lem}\label{lem: pr^* vs etp u-tors}
The full subcategory 
$
  \Angles{\pr{}^*\Coh(Y)}
    \subseteq 
  \Coh(\bbA^1_Y[-1])
$
corresponds to the full subcategory of $u$-torsion $\ccO_Y[u]$-modules
$$
  \Mod_{\ccO_Y[u]}^{\etp}\bigt{\Cohminus(Y)}^{u-\on{tors}}
    \subseteq 
  \Mod_{\ccO_Y[u]}^{\etp}\bigt{\Cohminus(Y)}
$$
under the equivalence of Lemma \ref{lem: Coh(A1[-1]) vs Cohminus etp}.
\end{lem}

\begin{proof}
It is clear that the dg-functor
$$
  \Coh(\bbA^1_Y[-1])
    \xto{\ccO_Y\otimes_{\ccO_Y[\eps]}-}
  \Mod_{\ccO_Y[u]}^{\etp}\bigt{\Cohminus(Y)}
$$
restricts to
$$
  \Angles{\pr{}^*\Coh(Y)}
    \to
  \Mod_{\ccO_Y[u]}^{\etp}\bigt{\Cohminus(Y)}^{u-\on{tors}}.
$$
Conversely, assume that $M\in \Mod_{\ccO_Y[u]}^{\etp}\bigt{\Cohminus(U)}^{u-\on{tors}}$.
Choose some $n\geq 0$ such that
$$
  u^n \sim 0:
  M
    \to
  M[2n].
$$
Then we see that the action of $\ccO_Y[u]$ on $\coFib(M\xto{u^n}M[2n])$ factors through the canonical augmentation $\ccO_Y[u]\to \ccO_Y$, hence
\begin{align*}
  \coFib(M\xto{u^n}M[2n]) \otimes_{\ccO_Y[u]}\ccO_Y
  &
    \simeq
  \coFib(M\xto{u^n}M[2n])\otimes_{\ccO_Y}\ccO_Y \otimes_{\ccO_Y[u]}\ccO_Y 
  \\
  &
    \simeq
  \coFib(M\xto{u^n}M[2n]) \otimes_{\ccO_Y} \ccO_Y[\eps]
    \in
  \Angles{\pr{}^*\Coh(Y)}.
\end{align*}
Since $\coFib(M\xto{u^n}M[2n])\simeq M[1]\oplus M[2n]$ and $\Angles{\pr{}^*\Coh(Y)}$ is by definition Karoubi-complete, we conclude that $M\in \Angles{\pr{}^*\Coh(Y)}$.
\end{proof}

\begin{cor}\label{cor: MFcoh vs Cohminus etp/Cohb}
  With the same notation as above, there is an equivalence
  $$
    \frac{\Coh(\bbA^1_Y[-1])}{\Angles{\pr{}^*\Coh(Y)}}
      \simeq
    \frac{\Mod_{\ccO_Y[u]}^{\etp}\bigt{\Cohminus(Y)}}{\Mod_{\ccO_Y[u]}^{\etp}\bigt{\Cohminus(Y)}^{u-\on{tors}}}.
  $$
\end{cor}

\sssec{}\label{sssec: defn MFcoh}

The dg-category 
$
  \frac{\Coh(\bbA^1_Y[-1])}{\Angles{\pr{}^*\Coh(Y)}}
$
is known to be equivalent to the dg-category of coherent matrix factorizations, see \cite[Remark 3.13]{pippi22b}.
Therefore, we will adopt the following notation:
$$
  \MFcoh(Y,0)
    :=
  \frac{\Coh(\bbA^1_Y[-1])}{\Angles{\pr{}^*\Coh(Y)}}.
$$

\sssec{}

The dg-category introduced above admit variants with (set-theoretic) support conditions. 
The equivalences of Lemmas \ref{lem: Coh(A1[-1]) vs Cohminus etp} and \ref{lem: pr^* vs etp u-tors} respect this condition.
Thus, in the sequel, for $W\subseteq Y$ a closed subset, we will freely use the equivalences
$$
  \Coh(\bbA^1_Y[-1])_W
    \simeq
  \Mod_{\ccO_Y[u]}^{\etp}\bigt{\Cohminus(Y)_W}
$$
$$
  \MFcoh(Y,0)_W
    \simeq
  \frac{\Mod_{\ccO_Y[u]}^{\etp}\bigt{\Cohminus(Y)_W}}{\Mod_{\ccO_Y[u]}^{\etp}\bigt{\Cohminus(Y)_W}^{u-\on{tors}}}.
$$

\subsection{Realizations of dg-categories}\label{ssec: cohomology of dg-categories}

In the main body of the paper, we will need certain cohomology theories defined for dg-categories: the \emph{motivic realization} and the \emph{$\ell$-adic realization}. 
These two realizations are related by a \emph{non-commutative Chern character}. 
We recall the main definitions here, referring to the original sources for the constructions and proofs.

\sssec{}

Let $\SH_S$ denote the stable homotopy $\oo$-category of schemes (see \cite{morelvoevodsky99,robalo15}). 
We first recall the \emph{motivic realization of dg-categories}, defined in \cite{brtv18}. 
This is a lax monoidal $\oo$-functor
$$
  \Mv_S:
  \dgCat_A 
    \to 
  \SH_S
$$
that sends localization sequences of dg-categories to fiber sequences of motivic spectra. 
In other words, $\Mv_S$ is a localizing invariant.

\sssec{}

In \cite{robalo15}, it is proven that
$$
  \Mv_S\bigt{\Perf(S)}
    \simeq 
  \BU_S,
$$
where $\BU_S$ denotes the motivic spectrum of non-connective homotopy-invariant algebraic K-theory.
In particular, for every qcqs scheme $X$, we have that 
$$
  \pi_0 
  \Hom_{\SH_S}\bigt{
                    \Mv_S(\Perf(S)),\Mv_S(\Perf(X))
                    }
    \simeq 
  \pi_0 \BU_S(X)
    \simeq 
  \HK_0(X).
$$
By abstract nonsense, the $\oo$-functor can be upgraded to
$$
  \Mv_S: 
  \dgCat_A 
    \to 
  \Mod_{\BU_S}\bigt{\SH_S}.
$$

\sssec{}

For every noetherian bounded derived scheme $X$, we denote $\G(X)\in \SH_S$ the associated G-theory spectrum. 
Then we have that
$$
  \G(X)
    \simeq 
  \Mv_S\bigt{\Coh(X)}(S).
$$
The canonical inclusion $\Perf(X)\hto \Coh(X)$ recovers the canonical morphism
$$
  \HK(X) 
    \to 
  \G(X).
$$
We will use the notation
$$
  \HK^{\on{sg}}(X)
    := 
  \coFib \bigt{\HK(X)\to \G(X)}.
$$
Since $\Mv_S$ is a localizing invariant and $\Sing(X):=\Coh(X)/\Perf(X)$, we see that
$$
  \HK^{\on{sg}}(X)
    \simeq 
  \Mv_S\bigt{\Sing(X)}(S).
$$

\sssec{} 

For $X$ as above, let $Z\hto X$ be a closed subscheme.
We write $\HK(X)_Z$ (respectively, $\G(X)_Z$)
for the homotopy invariant non-connective algebraic K-theory (respectively, G-theory) with support on $Z$.
Similarly, we write $\HK^{\on{sg}}(X)_Z$ for the cofiber of $\HK(X)_Z\to \G(X)_Z$ in the $\oo$-category of spectra.

\sssec{} 

Let $\Shv(S)$ denote the ind-completion of the stable $\oo$-category of constructible $\Qell$-adic sheaves on $S_{\et}$.
Following \cite{brtv18}, we define the $\ell$-adic realization of dg-categories as the composition
$$
  \rl_S: 
  \dgCat_A 
    \to 
  \Mod_{\BU_S}\bigt{\SH_S}
    \xto{\rho^{\ell}} 
  \Mod_{\Ql{,S}(\beta)}\bigt{\Shv(S)},
$$
where $\rho^{\ell}$ is the $\ell$-adic realization functor of \cite{ayoub14,cisinskideglise16} and
$$
  \Ql{,S}(\beta)
    =
  \bigoplus_{i\in \bbZ}\Ql{,S}(i)[2i].
$$

\sssec{} 

For an object $\sT \in \dgCat_A$, we will use the following notation:

$$
  \uH^*_{\et}\bigt{S,\rl_S(\sT)}
    := 
  \Hom_{\Shv(S)} \bigt{\Ql{,S}, \rl_S(\sT)}
    \simeq 
  \Hom_{\Mod_{\Ql{,S}(\beta)}\bigt{\Shv(S)}} \bigt{\Ql{,S}(\beta), \rl_S(\sT)}.
$$
In particular, $\uH^0_{\et}\bigt{S,\rl_S(\sT)}:= \pi_0 \Hom_{\Shv(S)} \bigt{\Ql{,S}, \rl_S(\sT)}$.

\sssec{}\label{sssec: nc Chern character}

As explained in \cite{toenvezzosi22}, there is a canonical morphism
$$
  \chern : 
  \HK(\sT)
    \to 
    \uH^*_{\et} \bigt{S,\rl_S(\sT)}
$$
to be referred to as the \emph{non-commutative $\ell$-adic Chern character}. 
For $\sT=\Perf(Y)$ the dg-category of perfect complexes on a regular $S$-scheme, this induces the usual Chern character
$$
  \up{K}_0(Y)
    \simeq 
  \HK_0(\sT)
    \to 
  \uH^0_{\et} \bigt{S,\rl_S(\sT)}
    \simeq 
  \bigoplus_{i\in \bbZ}\uH^{2i}_{\et} \bigt{Y,\Qell(i)}.
$$
Using this fact and \cite[Proposition 3.30]{brtv18}, it is easy to see that, for $\sT=\Sing(\bbA^1_S[-1])_s$, the non-commutative $\ell$-adic Chern character induces the inclusion
$$
  \up{K}_0(S)_s
    \simeq 
  \bbZ 
    \hto 
  \Qell 
    \simeq 
  \uH^0_{\et} \bigt{S,\rl_S(\sT)}.
$$


\section{The convolution monoidal category \texorpdfstring{$\sB$}{B} and its Drinfeld cocenter}\label{sec: differential forms on two-periodic complexes}

In this section, we begin our study of the derived groupoid $
  G 
  := 
  s\times_S s
$
and of its associated convolution monoidal dg-categories $\sB^+ := \Coh(G)$ and 
$ \sB := \Sing(G)$.
As a plain $A$-linear dg-category, $\Sing(G)$ is equivalent to the dg-category of two-periodic complexes of $k$-vector spaces, 
but the monoidal structure is quite subtle in mixed-characteristic: for instance, $\sB$ is not symmetric monoidal in that case. 
In a brief digression, we explain the simplification that occurs in pure characteristic.

Afterwards, we introduce the Drinfeld cocenter $\HH(\sB/A)$, which can be calculated using the cyclic bar complex of $\sB$ relative to $A$. 
We give geometric descriptions of the nodes of this cyclic bar complex in terms of the dg-categories introduced in Section \ref{ssec: categories external products}.

\subsection{The derived groupoid \texorpdfstring{$G$}{G}}

Following \cite{beraldopippi22, beraldopippi23, beraldopippi24, toenvezzosi22}, we recall here the explicit algebraic model for the derived $s$-groupoid
$$
  G
  :=
  s\times_S s.
$$

\sssec{}

Let $A[\eps] \in \sCAlg_{A/}$ denote the free simplicial $A$-algebra with a generator $\eps$ in degree $1$ that is sent to $\pi$ by the differential.
We see that $A[\eps]$ is a cofibrant replacement for the $A$-algebra $k$ and that
$$
  A[\eps_1,\eps_2]
    \simeq 
  A[\eps]\otimes_A A[\eps] 
    \in 
  \sCAlg_{A/}.
$$
In particular, we have
$$
  G 
  \simeq 
  \Spec {A[\eps_1, \eps_2]}.
$$

\sssec{} \label{sssec:Hopf-algebroid}

We now recall that the derived scheme $G$ admits a groupoid structure over $s$ in the $\oo$-category $\dSch_{/S}$. 
Working with strict models, we have to endow $A[\eps_1,\eps_2]$ with the structure of an Hopf algebroid over $A[\eps]$ in the category $\sCAlg_{A/}$. 
According to \cite[Section 4.1.1]{toenvezzosi22} and \cite[Section 3.1]{beraldopippi24}, the latter structure arises naturally as follows:
\begin{itemize}

\item 
the source and target morphisms
$$
  \sigma: 
  A[\eps] 
    \to 
  A[\eps_1,\eps_2]
$$
$$
  \tau: 
  A[\eps] 
    \to 
  A[\eps_1,\eps_2]
$$
are uniquely determined by $\sigma(\eps)=\eps_1$ and $\tau(\eps)=\eps_2$;

\item
the co-multiplication
$$
  \nabla: 
  A[\eps_1,\eps_2] 
    \to 
  A[\eps_1,\eps_2]
  \usotimes_{A[\eps]}
  A[\eps_1,\eps_2]
$$
is uniquely determined by setting
$$
  \nabla(\eps_1) 
    = 
  \eps_1 
  \otimes 1, 
  \;\;  
  \text{and} 
  \;\; 
  \nabla(\eps_2) 
    = 
  1 \otimes \eps_2;
$$

\item 
the co-unit
$$
  \xi : 
  A[\eps_1,\eps_2]
    \to 
  A[\eps]
$$
is uniquely determined by requiring that 
$$
  \xi(\eps_1)
    =
  \xi(\eps_2)
    = 
  \eps;
$$

\item 
the antipode
$$
  \alpha: 
  A[\eps_1,\eps_2]
    \to 
  A[\eps_1,\eps_2]
$$
exchanges the free variables in degree $-1$, namely
$$
  \alpha(\eps_1) 
    = 
  \eps_2 
  \;\; 
  \text{and} 
  \;\; 
  \alpha(\eps_2) 
    = 
  \eps_1.
$$
\end{itemize}

\subsection{The monoidal dg-category \texorpdfstring{$\sB$}{B}}

We equip $\Sing(G)$ with a monoidal structure induced by the groupoid structure of $G$.

\sssec{}

The groupoid structure on $G$ induces a monoidal structure on $\Coh(G)$. The tensor product is given by the functor
$$
  -\odot -: 
  \Coh(G)\otimes_A \Coh(G) 
    \to 
  \Coh(G)
$$
$$
  (E,F)
    \mapsto 
  (\pr{13})_*\bigt{\pr{12}^*E\otimes \pr{23}^*F}.
$$
Here $\pr{ij}: s\times_Ss\times_Ss \to s\times_Ss=G$ denotes the projection onto the $i^{th}$ and $j^{th}$ factor.
Moreover, the object $(\delta_{s/S})_*\ccO_s$ is a unit for this tensor product, where $\delta_{s/S} : s \to G=s\times_Ss$ is the diagonal embedding.
This monoidal structure induces a monoidal structure on the singularity category $\Sing(G)$, with parallel formulas.

\sssec{}

To make the above statements precise, we proceed as in \cite{toenvezzosi22}.
Consider the dg-algebra corresponding to $A[\eps_1,\eps_2]$ via the Dold-Kan equivalence. 
We shall denote such dg-algebra $A[\eps_1,\eps_2]$ as well
Then consider the (ordinary) category $\dgMod(A[\eps_1,\eps_2])$ of $A[\eps_1,\eps_2]$ dg-modules.
This is a monoidal $A$-linear dg-category,
the convolution tensor product being induced by the Hopf algebroid structure on $A[\eps_1,\eps_2]$.
Explicitly, we have that the convolution product is given by
$$
  -\odot-:
  \dgMod(A[\eps_1,\eps_2])\otimes_A \dgMod(A[\eps_1,\eps_2])
    \to
  \dgMod(A[\eps_1,\eps_2])
$$

$$
  (M,N) 
    \mapsto
  \tau_*M\otimes_{A[\eps]}\sigma_*N.
$$
As it is easy to see, the unit of $-\odot-$ is the diagonal $A[\eps_1,\eps_2]$ dg-module $A[\eps]$.

We endow $\dgMod(A[\eps_1,\eps_2])$ with the projective model structure and notice that this is compatible with the monoidal structure above (\cite[Section 4.1.1]{toenvezzosi22}).

Finally, we restrict our attention to $\cofdgMod(A[\eps_1,\eps_2])\subseteq \dgMod(A[\eps_1,\eps_2])$, the full subcategory spanned by those $A[\eps_1,\eps_2]$ dg-modules which are cofibrant and strictly perfect over $A$ and by the diagonal dg-module $A[\eps]$.

\sssec{}

We consider two classes of morphisms in $\cofdgMod(A[\eps_1,\eps_2])$ to be denoted $W_{\on{qi}}$ and $W_{\on{pe}}$ respectively.
The former is the class of quasi-isomorphisms while the latter is the class of morphisms in $\cofdgMod(A[\eps_1,\eps_2])$ whose homotopy cone is a perfect $A[\eps_1,\eps_2]$ dg-module.

\sssec{}

Using the theory of dg-localizations due to Toen (\cite{toen07}), we find that 
$$
  L_{W_{\on{qi}}}\bigt{\cofdgMod(A[\eps_1,\eps_2])}
    \simeq 
  \Coh(G) 
    \in 
   \dgCat_A,
$$
$$
  L_{W_{\on{pe}}}\bigt{\cofdgMod(A[\eps_1,\eps_2])}
    \simeq 
  \Sing(G) 
    \in 
  \dgCat_A.
$$
See \cite{toenvezzosi22, brtv18} for details.

\sssec{}

Arguing as in \cite[Appendix A]{toenvezzosi22}, one sees that both $\Coh(G)$ and $\Sing(G)$ acquire monoidal structures from that of $\cofdgMod(A[\eps_1,\eps_2])$. In the sequel, we denote by 
$$
  \sB^+ 
    \in 
  \Alg_{\bbE_1}(\dgCat_A)
$$
the monoidal dg-category corresponding to $\Coh(G)$ and by
$$
  \sB 
    \in 
  \Alg_{\bbE_1}(\dgCat_A)
$$
the one corresponding to $\Sing(G)$. Moreover, we let $\sB^+_{\on{pe}}$ denote the non unital $\bbE_1$-algebra in $\dgCat_A$ corresponding to $\Perf(G)$.

\subsection{Digression: the pure characteristic case}

In this subsection we explain why the situation simplifies in a substantial way when $A$ is of pure characteristic.

\sssec{}

Assume that $A$ is a $k$-algebra. 
Let $k[\eps]$ denote the simplicial Koszul algebra with $\eps$ a generator in degree $1$ sent to zero by the differential.
Note that $G \simeq \Spec{k[\eps]}$.
The simplicial ring $k[\eps]$ supports a Hopf algebra structure over $k$, where the co-multiplication is given by
$$
  \ol{\nabla}:
  k[\eps]
    \to 
  k[\eps]\otimes_k k[\eps]
$$
$$
  \eps 
    \mapsto 
  \eps \otimes 1 + 1\otimes \eps,
$$
the co-unit is given by
$$
  \ol{\xi}: 
  k[\eps]
    \to 
  k
$$
$$
  \eps \mapsto 0
$$
and the antipode is given by
$$
  \ol{\alpha}: 
  k[\eps]
    \to 
  k[\eps]
$$
$$
\eps \mapsto -\eps.
$$

\sssec{}

Moreover, it is easily seen that this Hopf algebroid structure is co-commutative.
Indeed, the following diagram commutes:
\begin{equation*}
    \begin{tikzcd}
    & k[\eps] \arrow[ld,swap,"\ol{\nabla}"] \arrow[rd,"\ol{\nabla}"] & \\
    k[\eps]\otimes_kk[\eps] \arrow[rr,"\on{swap}"]& & k[\eps]\otimes_kk[\eps].
    \end{tikzcd}
\end{equation*}
Here $\on{swap} : k[\eps]\otimes_kk[\eps]\to k[\eps]\otimes_kk[\eps]$ is the map uniquely determined by
$$
  \eps \otimes 1 
    \mapsto 
  1 \otimes \eps, 
  \;\;\;\; 
  1 \otimes \eps 
    \mapsto 
  \eps \otimes 1.
$$

\sssec{}

Let $i:k\hto A$ denote the canonical embedding of $k$ into $A$.
The following fact is the key observation in pure characteristic.

\begin{lem}
The morphism $i:k\hto A$ induces an equivalence of Hopf algebroids
$$
  \iota:
  k[\eps]
    \hto 
  A[\eps_1,\eps_2] 
$$
$$ 
  \eps 
  \mapsto 
  \eps_1 - \eps_2.
$$
\end{lem}

\begin{proof}

We start by proving the compatibility of the Hopf structures.

\sssec*{Co-multiplication} 
We need to verify that the following square commutes
\begin{equation*}
    \begin{tikzcd}
    k[\eps] \arrow[r,"\ol{\nabla}"] \arrow[d,"\iota"] & k[\eps]\otimes_k k[\eps]\arrow[d,"\iota \otimes \iota"] \\
    A[\eps_1,\eps_2] \arrow[r,"\nabla"] & A[\eps_1,\eps_2]\otimes_{\tau,A[\eps],\sigma} A[\eps_1,\eps_2].
    \end{tikzcd}
\end{equation*}
We see that
\begin{align*}
    (\iota \otimes \iota) \circ \ol{\nabla} (\eps) & = 
    (\iota \otimes \iota) (\eps \otimes 1 +1 \otimes \eps )\\
                     & = (\eps_1-\eps_2)\otimes 1 + 1\otimes (\eps_1-\eps_2)\\
                                      & = \eps_1 \otimes 1 - 1 \otimes \eps_2\\
                                      & = \nabla (\eps_1-\eps_2)\\
                                      & = \nabla \circ \iota (\eps),
\end{align*}
which immediately implies that the square commutes.

\sssec*{Co-unit}
We need to verify that the following square commutes
\begin{equation*}
    \begin{tikzcd}[column sep = 3cm]
    k[\eps] 
    \arrow[r,"\ol{\xi}"] 
    \arrow[d,"\iota"] 
    & k 
    \arrow[d,"i"] 
    \\
    A[\eps_1,\eps_2] 
    \arrow[r,"\xi"] 
    & 
    A[\eps].
    \end{tikzcd}
\end{equation*}
This is immediate as the top circuit is characterized by the requirement that $\eps \mapsto 0$, a property that the lower one satisfies as well.

\sssec*{Antipode} 

It remains to show that the morphism is compatible with the antipodes as well, i.e. that the square

\begin{equation*}
    \begin{tikzcd}[column sep = 3cm]
    k[\eps] 
    \arrow[r,"\ol{\alpha}"] 
    \arrow[d,"\iota"] 
    & k[\eps] 
    \arrow[d,"\iota"] 
    \\
    A[\eps_1,\eps_2] 
    \arrow[r,"\alpha"] 
    & 
    A[\eps_1,\eps_2],
    \end{tikzcd}
\end{equation*}
which is immediate.

\sssec*{Equivalence} 
The last step is to prove that the map $k[\eps]\to A[\eps_1,\eps_2]$ is a weak equivalence, which follows from an elementary calculation.
\end{proof}

\begin{cor}\label{cor: retraction B to HH(B/A) in pure char}

In the pure characteristic case, the derived groupoid $G$ is a commutative group object. Accordingly, the monoidal dg-category $\sB$ is symmetric monoidal and in fact equivalent to
$$
  \sB 
    \simeq 
  \Perf \bigt{k[u,u^{-1}]} 
    \in 
  \Alg_{\bbE_{\oo}}\bigt{\dgCat_A}.
$$
\end{cor}

\sssec{}\label{sssec: retraction B to HH(B/A) in pure char}

The above corollary implies that the canonical dg-functor 
$$
  \sB 
    \to 
  \HH(\sB/A)
    =
  \sB\otimes_{\sB \otimes_A \sB}\sB
$$
admits a retraction
$$
  \on{m}: 
  \HH(\sB/A)
    \to 
  \sB
$$
induced by the (commutative) multiplication of $\sB$. The existence of such a retraction is an important feature of the pure characteristic situation: see the discussion of Section \ref{sssec:pure-char}.

\subsection{The canonical filtration on the Drinfeld cocenter}

\sssec{} 

Let $\sB^{\rev}$ denote the associative algebra in $\dgCat_A$ obtained from $\sB$ by reversing the monoidal structure (\cite[Remark 4.1.1.7]{lurieha}).
Moreover, we let 
$
  \sB^{\env}
    :=
  \sB \otimes_A \sB^{\rev}
$
denote the enveloping algebra of $\sB$.

\sssec{}

Tautologically, $\sB$ admits both a left and a right $\sB^{\env}$-module structure.
We will denote $\sB^{\uL}$ (resp. $\sB^{\uR}$) this canonical left (resp. right) $\sB^{\env}$-module.

\sssec{}

The main object of study in this paper is the \emph{Drinfeld cocenter of $\sB$ relative to $\Perf(A)$}:
$$
  \HH(\sB/A)
    := 
  \sB^{\uR}\otimes_{\sB^{\env}} \sB^{\uL} 
    \in 
  \dgCat_A.
$$

\begin{rmk}

Since $\sB$ is only an $\bbE_1$-algebra in general, we have that
$$
  \HH(\sB/A)
    \in 
  \Alg_{\bbE_0}(\dgCat_A).
$$
This means that $\HH(\sB/A)$ is an $A$-linear dg-category with a marked object (the image of $(1_{\sB},1_{\sB})$ along the canonical functor $\sB\otimes_A \sB\to \HH(\sB/A)$) but no other structure.
\end{rmk}

\sssec{}

It is well-known that $\HH(\sB/A)$ can be computed as the colimit of the diagram

\begin{equation*}
\begin{tikzcd}
    \cdots
    \;\;
    \sB\otimes_A\sB\otimes_A\sB\otimes_A\sB 
    \arrow[r,shift right=3]
    \arrow[r,shift right=1]
    \arrow[r,shift left=1]
    \arrow[r,shift left=3]
    &
    \sB\otimes_A\sB\otimes_A\sB
    \arrow[r,shift right=2]
    \arrow[r]
    \arrow[r,shift left=2]
    &
    \sB\otimes_A\sB
    \arrow[r,shift right=1]
    \arrow[r,shift left=1]
    &
    \sB,
\end{tikzcd}
\end{equation*}
which is usually referred to as the \emph{cyclic bar complex}. More precisely, the cyclic bar complex is a functor
$$
  \Bar(\sB/A):
  (\Delta_{\inj})^\op 
    \to 
  \dgCat_A,
$$
with $\Delta_{\inj}$ the category whose objects are directed finite sets $[n]$ for all $n \geq 0$ and morphisms 
$$
\Hom_{\Delta_{\inj}}\bigt{[i],[j]}=\Bigl \{\alpha :[i]\to [j]: \alpha \text{ is monotone and injective} \Bigr \}.
$$

\sssec{}

The arrows in the cyclic bar complex are defined using the associative algebra structure on $\sB$. 
In particular, the arrows
$
  \sB^{\otimes_A n}\to \sB
$
are obtained by pre-composition of 
$$
  -\odot- \dots -\odot - :\sB^{\otimes_A n}\to \sB
$$
with the $n$ cyclic permutations on $\sB^{\otimes_A n}$.

\begin{rmk}

Strictly speaking, the cyclic bar complex of $\sB \in \Alg_{\bbE_1}(\dgCat_A)$ is a \emph{simplicial} diagram $\Delta^{\op}\to \dgCat_A$.
However, since we will be interested in the colimit (a.k.a. geometric realization) of this diagram, we can harmlessly consider its underlying \emph{semi-simplicial} diagram by \cite[Lemma 6.5.3.7]{luriehtt}.
\end{rmk}

\sssec{}

We will be interested in approximating the cyclic bar complex via finite sub-complexes. 
To this end, let $\Delta^{\leq m}_{\inj}$ denote full subcategory of $\Delta_{\inj}$ whose objects are the directed finite sets $[0],[1],\dots,[m]$.
If $m_1 \leq m_2$, there are obvious fully faithful embeddings
$$
  \Delta^{\leq m_1}_{\inj} \subseteq \Delta^{\leq m_2}_{\inj}
    \subseteq
  \Delta_{\inj}.
$$

\begin{lem}\label{lem : Deltainj = colim Deltainjm}

The natural functor
$$
  \varinjlim_{m \in \bbN} \Delta_{\inj}^{\leq m}
    \to 
  \Delta_{\inj}
$$
is an equivalence of categories.
\end{lem}

\begin{proof}

This functor is clearly fully faithful and essentially surjective.
\end{proof}

\begin{defn}
 
Let $n\geq 0$. 
The \emph{$n^{th}$ truncated cyclic bar complex} is the functor 
$$
  \Bar_n(\sB/A): 
  (\Delta^{\leq n}_{\inj})^\op 
    \hto 
  (\Delta_{\inj})^\op 
    \xto{\Bar(\sB/A)} 
  \dgCat_A,
$$
depicted as
\begin{equation*}
\Bar_n(\sB/A)
=
\begin{tikzcd}
    \Bigl (
    \sB^{\otimes_A n+1}
    \arrow[r,shift right=3]
    \arrow[r, draw=none, "\raisebox{+0.5ex}{\dots}" description]
    \arrow[r,shift left=3]
    &
    \cdots
    \arrow[r,shift right=2]
    \arrow[r]
    \arrow[r,shift left=2]
    &
    \sB\otimes_A\sB
    \arrow[r,shift right=1]
    \arrow[r,shift left=1]
    &
    \sB
    \Bigr ).
\end{tikzcd}
\end{equation*}
We denote by
$$
  F_n\HH(\sB/A):=
  \varinjlim \on{Bar}_n(\sB/A) 
    \in
  \dgCat_A
$$
its colimit.
\end{defn}

\sssec{}

For $n \leq m$ in $\bbN$, there is a canonical equivalence
$$
  \Bar_n(\sB/A)
    \simeq 
  \Bar_{m}(\sB/A)_{|(\Delta^{\leq n}_{\inj})^\op}
$$
induced by the embedding $\Delta^{\leq n}_{\inj}\hto \Delta^{\leq m}_{\inj}$.
This yields a compatible system of dg-functors
$$
  F_n\HH(\sB/A)
    \to 
  F_{n+m}\HH(\sB/A),
$$
that is, a diagram
$$
  F_\bullet \HH(\sB/A):
  \bbN 
    \to 
  \dgCat_A,
$$
where $\bbN$ is the category associated to the directed set $(\bbN=\bbZ_{\geq 0}, \leq)$.

\sssec{}

We will need analogous constructions for the monoidal dg-category $\sB^+$. 
Explicitly, let
$$
  \on{Bar}(\sB^+/A):
  (\Delta_{\on{inj}})^\op 
    \to 
  \dgCat_A
$$
be the cyclic bar complex of $\sB^+$ relative to $A$, its $n^{th}$ truncation
$$
  \on{Bar}_n(\sB^+/A):
  (\Delta_{\on{inj}}^{\leq n})^\op 
    \to 
  \dgCat_A
$$
and the colimit
$$
  F_n\HH(\sB^+/A)
    :=
  \varinjlim \on{Bar}_n(\sB^+/A)
    \in
  \dgCat_A.
$$

\begin{lem}

With the same notation as above, we have canonical equivalences
$$
  \HH(\sB^+/A)
    \simeq 
  \varinjlim_{\bbN} 
  \bigt{ F_n\HH(\sB^+/A)}
$$
$$
  \HH(\sB/A)
    \simeq 
  \varinjlim_{\bbN} \bigt{ F_n\HH(\sB/A)}
$$
in $\dgCat_A$.
\end{lem}

\begin{proof}

The two proofs are analogous, so we only focus on the second one.
By \cite[Lemma 5.5.2.6]{lurieha} and \cite[Lemma 6.5.3.7]{luriehtt} we have
\begin{align*}
    \HH(\sB/A) 
      & 
      \simeq 
    \varinjlim \Bar(\sB/A) 
      \\
      & 
      \simeq 
    \varinjlim \bigt{\varinjlim_{\bbN}\Bar_n(\sB/A)},
\end{align*}
where the last equivalence follows from Lemma \ref{lem : Deltainj = colim Deltainjm}. 
Then the statement is an immediate consequence of the fact that colimits commute with colimits:
\begin{align*}
  \varinjlim \bigt{\varinjlim_{\bbN}\Bar_n(\sB/A)} 
    & 
    \simeq 
  \varinjlim_{\bbN} \bigt{\varinjlim \Bar_n(\sB/A)}
    \\
    &
    =
  \varinjlim_{\bbN} \bigt{ F_n\HH(\sB/A)}.
\end{align*}
\end{proof}

\subsection{The nodes of the cyclic bar complex}

We now turn our attention to the individual pieces of the cyclic bar complex, that is, the dg-categories $(\sB^{\otimes_An})$'s.

\sssec{}

As a preliminary step, we give geometric descriptions of the tensor products $\sB^{+,\otimes_An}$.
Recall the definition of $\Cohext(G^{\times_Sn})$ from Section \ref{ssec: categories external products}.

\begin{lem}\label{lem: nodes Bar(B^+/A)}

The exterior product $  \sB^{+,\otimes_An} \to 
  \Coh(G^{\times_Sn})$ restricts to an equivalence 
$$
  \sB^{+,\otimes_An}
    \xto{\sim}  
  \Cohext(G^{\times_Sn}) 
$$
in $\dgCat_A$.
\end{lem}

\begin{proof}

We proceed in steps.

\sssec*{Step 1:  image of the generator} 

Denote by 
$$
  \ccE_S^n:
  \sB^{+,\otimes_n}
    \to 
  \Coh(G^{\times_Sn})
$$
$$
  (M_1,\dots,M_n)
    \mapsto 
  M_1\boxtimes_S\dots \boxtimes_SM_n
$$
the dg-functor of external tensor product.
Since $\sB^+$ is Karoubi-generated by $1_{\sB^+}=(\delta_{s/S})_*\ccO_s$, it is important to compute the object $\ccE_S^n(1_{\sB^+},\dots,1_{\sB^+})$.
To this end, consider the morphism 
$$
\delta_{s/S}^{\times_Sn}:s^{\times_Sn}\to G^{\times_Sn}
$$
obtained as the product of $n$ diagonals $\delta_{s/S}$ over $S$.
It follows from \eqref{eqn:push of boxtimes} that
$$
  \ccE_S^n(1_{\sB^+},\dots,1_{\sB^+})
    \simeq 
  (\delta_{s/S}^{\times_Sn})_*\ccO_{s^{\times_Sn}}.
$$

\sssec*{Step 2: fully-faithfulness}

The morphism $\delta_{s/S}^{\times_Sn}:s^{\times_Sn}\to G^{\times_Sn}$ corresponds to the morphism of simplicial rings
$$
  A[\eps_1^{(1)},\eps_2^{(1)},\dots,\eps_1^{(n)},\eps_2^{(n)}] 
    \to 
  A[\eps^{(1)},\dots,\eps^{(n)}]
$$
$$
  \eps_i^{(k)}
    \mapsto 
  \eps^{(k)}, 
    \;\;\;\; 
  i=1,2, \;\; k=1,\dots,n.
$$
Observe that
\begin{align*}
  \Hom_{\sB^{+,\otimes_An}}\bigt{(1_{\sB^+},\dots,1_{\sB^+}),(1_{\sB^+},\dots,1_{\sB^+})} 
    & 
    \simeq 
  \Hom_{\sB^+}(1_{\sB^+},1_{\sB^+})\otimes_A\dots \otimes_A 
  \Hom_{\sB^+}(1_{\sB^+},1_{\sB^+})
    \\
    & 
    \simeq 
  k[u_1]\otimes_A \dots \otimes_A k[u_n],
\end{align*}
where each $u_i$ is a free variable in cohomological degree $2$.
We need to verify that the morphism
\begin{equation}\label{eqn : hom B^+otimesn to Coh(G^n)}
  k[u_1]\otimes_A \dots \otimes_A k[u_n] 
    \to 
  \Hom_{A[\eps_1^{(1)},\eps_2^{(1)},\dots,\eps_1^{(n)},\eps_2^{(n)}]}\bigt{A[\eps^{(1)},\dots,\eps^{(n)}],A[\eps^{(1)},\dots,\eps^{(n)}]}
\end{equation}
corresponding to 
$$
  \Hom_{\sB^{+,\otimes_An}}\bigt{(1_{\sB^+},\dots,1_{\sB^+}),(1_{\sB^+},\dots,1_{\sB^+})} 
    \to 
  \Hom_{\Coh(G^{\times_Sn})}\bigt{\ccE_S^n(1_{\sB^+},\dots,1_{\sB^+}),(1_{\sB^+},\dots,1_{\sB^+})}
$$
is a quasi-isomorphism.
Applying \cite[Theorem 1.2]{ben-zvifrancisnadler10} multiple times, we obtain that
$$
  \Hom_{A[\eps_1^{(1)},\eps_2^{(1)},\dots,\eps_1^{(n)},\eps_2^{(n)}]}\bigt{A[\eps^{(1)},\dots,\eps^{(n)}],A[\eps^{(1)},\dots,\eps^{(n)}]}
    \simeq
    \bigotimes_{j=1}^n
    \Hom_{A[\eps_1^{(j)},\eps_2^{(j)}]}\bigt{A[\eps^{(j)}],A[\eps^{(j)}]}.
$$
It follows that the map \eqref{eqn : hom B^+otimesn to Coh(G^n)} identifies with the tensor product of the map
$$
  k[u_j]\to\Hom_{A[\eps_1^{(j)},\eps_2^{(j)}]}\bigt{A[\eps^{(j)}],A[\eps^{(j)}]},
$$
which is a quasi-isomorphism, see for instance \cite[Lemma 2.39]{brtv18}.

\sssec*{Step 3: Karoubi generation}

By the previous step, we know that 
$$
  \ccE_S^n:
  \sB^{+,\otimes_An}
    \hto 
  \Coh(G^{\times_Sn}),
$$
is fully faithful.
Since $\sB^{+,\otimes_An}$ is Karoubi generated by the object $(1_{\sB^+},\dots,1_{\sB^+})$, it is then clear that the subcategory of $\Coh(G^{\times_Sn})$ corresponding to $\sB^{+,\otimes_An}$ is Karoubi-generated by the object 
$$
  \ccE_S^n(1_{\sB^+},\dots,1_{\sB^+})
    \simeq 
  (\delta_{s/S}^{\times_Sn})_*\ccO_{s^{\times_Sn}}.
$$
Such a dg-category is obviously equivalent to $\Cohext(G^{\times_Sn})$.
\end{proof}

\sssec{}

Fix $n\geq 1$ and, for any $1\leq i \leq n$, denote by 
$$
  \pr{\hat{i}}: 
  G^{\times_Sn}
    \to 
  G^{\times_Sn-1}
$$
the projection that omits the $i^{\textit{th}}$ factor. By definition, $\Perfext(G^{\times_Sn})$ is the full subcategory of $\QCoh(G^{\times_Sn })$ Karoubi-generated by the objects
$$
  \pr{\hat{i}}^*(\delta_{s/S}^{\times_Sn-1})_*\ccO_{s^{\times_Sn-1}}:i=1,\dots n.
$$
By Lemma \ref{lem:Cohext contained in Coh} we see that $\Perfext(G^{\times_Sn})\subseteq \Cohext(G^{\times_Sn})\subseteq  \Coh(G^{\times_Sn})$.  
We then set
$$
  \Singext(G^{\times_Sn})
    :=
  \frac{\Cohext(G^{\times_Sn})}{\Perfext(G^{\times_Sn})} 
    \in 
  \dgCat_A.
$$

\begin{cor}\label{cor: nodes Bar(B/A)}

The exterior product induces an equivalence 
$$
  \sB^{\otimes_An} 
    \xto{\simeq}
  \Singext(G^{\times_Sn})
$$
in $\dgCat_A$.
\end{cor}

\begin{proof}

Since $\sB\simeq \sB^+/\sB^+_{\on{pe}}$, it follows from \cite[Appendix A]{toenvezzosi22} that $\sB^{\otimes_An}$ is equivalent to the dg-quotient of $\sB^{+,\otimes_An}$ by the full subcategory Karoubi-generated by the objects in
$$
  \bigcup_{j=1}^n \sB^+\otimes_A \dots \sB^+\otimes_A \sB^+_{\on{pe}}\otimes_A\sB^+\otimes_A\dots \otimes_A\sB^+,
$$
with $\sB^+_{\on{pe}}$ appearing in the $j^{th}$ factor. 

Since $\sB^+_{\on{pe}}$ is Karoubi-generated by $\ccO_G$ and $\sB^+$ by $(\delta_{s/S})_*\ccO_s$, we see that 
$$
  \sB^+\otimes_A \dots \sB^+\otimes_A \underbrace{\sB^+_{\on{pe}}}_{j^{th} \text{ place}}\otimes_A\sB^+\otimes_A\dots \otimes_A\sB^+
    \simeq 
  \bigl \langle \pr{\hat{j}}^*(\delta_{s/S}^{\times_Sn-1})_*\ccO_{s^{\times_Sn-1}}\bigr \rangle,
$$
so that 
$$
  \Angles{\bigcup_{j=1}^n \sB^+\otimes_A \dots \sB^+\otimes_A \sB^+_{\on{pe}}\otimes_A\sB^+\otimes_A\dots \otimes_A\sB^+}
    \simeq 
  \Perfext(G^{\times_Sn}).
$$
Then, by the previous lemma we get
\begin{align*}
  \sB^{\otimes_An} 
    & 
    \simeq 
  \frac{\sB^{+,\otimes_An}}{\Angles{\bigcup_{j=1}^n \sB^+\otimes_A \dots \sB^+\otimes_A \sB^+_{\on{pe}}\otimes_A\sB^+\otimes_A\dots \otimes_A\sB^+}}
    \\
    &
    \simeq
  \frac{\Cohext(G^{\times_Sn})}{\Perfext(G^{\times_Sn})}
    =: 
  \Singext(G^{\times_Sn})
    \in 
  \dgCat_A.
\end{align*}
\end{proof}


\section{Combinatorial preliminaries} \label{sec: combinatorial preliminaries}

In the first two subsections, we introduce the derived schemes $s(n)$ and some natural morphisms involving them.
Then, in Section \ref{ssec: cofibrant models}, we carefully describe cofibrant simplicial $A$-algebras that model such derived schemes.
This is done because the results of Section \ref{ssec: the coherent homotopies} heavily rely on such explicit models.
In Section \ref{ssec: some combinatorics}, we collect some useful elementary results of combinatorial nature.

\subsection{Infinitesimal neighbourhoods of \texorpdfstring{$s\hto S$}{s hto S}}

\sssec{}

For any $n\geq 1$, we set
$$
  s(n)
    := 
  \Spec{k(n)}.
$$
Notice that
$$
  s(n) 
    =
  \Spec{A^{(n)}[\eps]},
$$
where $A^{(n)}[\eps]$ denotes the simplicial algebra freely generated by the degree $1$ element $\eps$, mapped to $\pi^n$ by the differential.

\sssec{}

For $n\geq m \geq 1$, there is an obvious closed embedding $j(m,n):s(m) \to s(n)$. At the level of simplicial algebras, $j(m,n)$ corresponds to the morphism
$$
  A^{(n)}[\eps]
    \to 
  A^{(m)}[\eps]
$$
determined by 
$$
  \eps 
    \mapsto 
  \pi^{n-m}\eps.
$$

\sssec{}

Let $p\geq 1$. Taking products over $S$, we have 
$$
  s(n)^{\times_Sp} 
    \simeq
  \on{Spec}
  \Bigt{A^{(n)}[\eps^{(i)}:i=1\dots,p]},
$$
where 
$$
  A^{(n)}[\eps^{(i)}:i=1\dots,p]
$$
is the simplicial algebra freely generated by  degree $1$ elements $\eps^{(i)}$ ($i=1,\dots,p$), all mapped to $\pi^n$ by the differential.

\sssec{}

The induced morphism
$$
j(m,n)^{\times_Sp}:s(m)^{\times_Sp}\to s(n)^{\times_Sp}
$$
corresponds to the following morphism of simplicial algebras:
$$
  A^{(n)}[\eps^{(i)}:i=1\dots,p] 
    \to 
  A^{(m)}[\eps^{(i)}:i=1\dots,p]
$$
$$
  \eps^{(i)} 
    \mapsto 
  \pi^{n-m}\eps^{(i)},
  i=1,\dots,p.
$$

\sssec{}

There is a natural action of the cyclic group $C_n = \langle \rho_n \rangle$ on $s^{\times_Sn}$, defined as follows: for $j=0,\dots,n-1$, the action of $\rho_n^j$ on $s^{\times_Sn}$ corresponds to the auto-equivalence 
$$
  A[\eps^{(i)}:i=1\dots,n]
    \to 
  A[\eps^{(i)}:i=1\dots,n]
$$
$$
  \eps^{(i)} 
    \mapsto 
  \eps^{(\rho_n^ji)}, 
  i=1,\dots,n.
$$

\subsection{Infinitesimal neighbourhoods of \texorpdfstring{$G\hto S$}{G hto S}}

\sssec{}

For any $n\geq 1$, we set
$$
  G(n)
    := 
  s(n)\times_Ss(n).
$$
and observe that $G(1) = G$. 
Notice that
$$
  G(n)
    =
  \Spec{A^{(n)}[\eps_1,\eps_2]},
$$
where 
$$
  A^{(n)}[\eps_1,\eps_2]
    =
  A^{(n)}[\eps]\otimes_AA^{(n)}[\eps]
$$
denotes the free simplicial algebra generated by two elements $\eps_1,\eps_2$ in degree $1$, mapped to $\pi^n$ by the differential.

\sssec{}

For any $n\geq m \geq 1$, the closed embedding $j(m,n):s(m) \to s(n)$ induces a closed embedding
$$
  j_G(m,n) :
  G(m)
    \to 
  G(n)
$$
of derived schemes. 
At the level of simplicial algebras, $j_G(m,n)$ corresponds to the morphism
$$
  A^{(n)}[\eps_1,\eps_2]
    \to 
  A^{(m)}[\eps_1,\eps_2]
$$
determined by 
$$
  \eps_1 \mapsto \pi^{n-m}\eps_1, 
    \;\;\;\; 
  \eps_2 \mapsto \pi^{n-m}\eps_2.
$$

\sssec{}

For any $p\geq 1$, we have
$$
  G(n)^{\times_Sp}
    \simeq
  \on{Spec}
           \Bigt{
                 A^{(n)}[\eps_1^{(i)},\eps_2^{(i)}:i=1\dots,p]
                 },
$$
where 
$$
  A^{(n)}[\eps_1^{(i)},\eps_2^{(i)}:i=1\dots,p]
  =
  \bigotimes_{i=1}^pA^{(n)}[\eps_1,\eps_2].
$$

\sssec{}

The natural morphism
$$
  j_G(m,n)^{\times_Sp}:
  G(m)^{\times_Sp}
    \to 
  G(n)^{\times_Sp}
$$
admits the following strict model at the level of simplicial $A$-algebras:
$$
  A^{(n)}[\eps_1^{(i)},\eps_2^{(i)}:i=1\dots,p] 
    \to 
  A^{(m)}[\eps_1^{(i)},\eps_2^{(i)}:i=1\dots,p]
$$
$$
  \eps_1^{(i)} 
    \mapsto 
  \pi^{n-m}\eps_1^{(i)}, 
    \;\;\;\; 
  \eps_2^{(i)} 
    \mapsto 
  \pi^{n-m}\eps_2^{(i)}, 
    \;\;\;\; 
  i=1,\dots,p.
$$

\sssec{}

We will consider the diagonal action of the cyclic group $C_n$ on 
$$
  G^{\times_Sn}
  \simeq
  s^{\times_Sn} \times_S s^{\times_Sn}.
$$
For each $j=0,\dots,n-1$, the action of $\rho_n^j$ on $G^{\times_Sn}$ then corresponds to the auto-equivalence 
$$
  A[\eps_1^{(i)},\eps_2^{(i)}:i=1\dots,n]
    \to 
  A[\eps_1^{(i)},\eps_2^{(i)}:i=1\dots,n]
$$
$$
  \eps_1^{(i)} 
    \mapsto 
  \eps_1^{(\rho_n^ji)}, 
    \;\;\;\; 
  \eps_2^{(i)} 
    \mapsto 
  \eps_2^{(\rho_n^ji)}, 
    \;\;\;\; 
  i=1,\dots,n.
$$

\subsection{Cofibrant models}\label{ssec: cofibrant models}

We now discuss cofibrant models for the simplicial $A$-algebras introduced above. 
These are particularly convenient for the homotopies we are going to define in Theorem
\ref{thm: the homotopy}.

\sssec{Cofibrant model of the \texorpdfstring{$A[t]$}{A[t]}-algebra \texorpdfstring{$A$}{A}}

Consider the $A[t]$-algebra $\phi_0:A[t]\to A$, uniquely determined by the requirement that $t\mapsto 0$. 
We can construct a cofibrant model for $\phi_0:A[t]\to A$ in $\sCAlg_{A[t]/}$ using the monad
$$
  \CAlg_{A/}
    \to 
  \CAlg_{A/}, 
  \;\;\;\; 
  B 
    \mapsto 
  B\otimes_A A[t]. 
$$
We obtain the simplicial $A[t]$-algebra 
$$
  \Bar_{A[t]}(A): 
  \Delta^\op 
    \to 
  \CAlg_{A[t]/}
$$
$$
  [p]
    \mapsto 
  \Bar_{A[t]}(A)_{p}
    :=
  \frac{A\bigq{t^{(j)}:j\in [p+1]}}{\langle t^{(p+1)} \rangle}  
    = 
  \frac{A[t^{(0)},\dots,t^{(p+1)}]}{\langle t^{(p+1)} \rangle},
$$
where the $A[t]$-algebra structure is induced by the embedding 
$$
  A[t]\simeq A[t^{(0)}]
    \subseteq 
  A[t^{(0)},\dots,t^{(p)}]
    \simeq 
  \frac{A[t^{(0)},\dots,t^{(p+1)}]}{\langle t^{(p+1)} \rangle}.
$$

The face maps (for $p\geq 0$ and $0\leq i\leq p+1$) are defined as follows: 
$$
  d_i^p:
  \Bar_{A[t]}(A)_{p+1}
    =
  \frac{A\bigq{t^{(j)}: j \in [p+2]}}{\langle t^{(p+2)} \rangle}
    \to 
  \Bar_{A[t]}(A)_{p}
    =
  \frac{A\bigq{t^{(j)}:j\in [p+1]}}{\langle t^{(p+1)} \rangle} 
$$
$$
  d_i^p(t^{(j)})
    := 
  \begin{cases}
  t^{(j)} & \text{if} \;\; j\leq i;
    \\
  t^{(j-1)} & \text{if} \;\; j\geq i+1.
\end{cases}
$$
    
For $0\leq i \leq p+1$, we have $d^p_i\bigt{t^{(p+2)}}=t^{(p+1)}$, hence the morphism is well-defined.

\medskip

The degeneracy maps (for $p\geq 1$ and $0\leq i\leq p-1$) are defined as follows:
$$
  s_i^p:
  \Bar_{A[t]}(A)_{p-1}
    =
  \frac{A\bigq{t^{(j)}: j \in [p]}}{\langle t^{(p)} \rangle }
    \to 
  \Bar_{A[t]}(A)_{p}
    =
  \frac{A\bigq{t^{(j)}: j \in [p+1]}}{\langle t^{(p+1)} \rangle }
$$
$$
  s^p_i(t^{(j)}) 
    :=
  \begin{cases}
    t^{(j)} & \text{if} \;\; j\leq i;
      \\
    t^{(j+1)} & \text{if} \;\; j\geq i+1.
  \end{cases}
$$

This is well-defined: for every $0\leq i \leq p-1 $, we have $s^p_i\bigt{t^{(p)}}= t^{(p+1)}=0 $.

\begin{rmk}
Notice that 
$$
  d^p_{p+1}(t^{(p+1)})
    =
  t^{(p+1)}
    =
  0,
$$ 
which means that $\on{Bar}_{A[t]}(A)$ is a simplicial commutative $A[t]$-algebra where $t$ is homotopic to zero.
\end{rmk}

By abuse of notation, let $A[t]$ denote the constant simplicial ring associated to $A[t]$.
There is a morphism of simplicial algebras
$$
  A[t]
    \subseteq 
  \Bar_{A[t]}(A),
$$
which is the inclusion 
$$
  A[t]
    \simeq 
  A[t^{(0)}]
    \subseteq 
  \frac{A\bigq{t^{(j)}:j\in [p+1]}}{\langle t^{(p+1)} \rangle }
$$ 
in all degrees.

\sssec{Cofibrant model for the \texorpdfstring{$A$}{A}-algebra \texorpdfstring{$k(m)$}{k(m)}}

Fix an integer $m\geq 1$. 
Let $\phi_{\pi^m}:A[t]\to A$ be the $A[t]$-algebra uniquely determined by the requirement that $t \mapsto \pi^m$.
We have
$$
  k(m)
    \simeq
  A\otimes^{\bbL}_{\phi_0,A[t],\phi_{\pi^m}}A.
$$
The cofibrant $A[t]$-replacement $\Bar_{A[t]}(A)$ yields a cofibrant replacement for the $A$-algebra $k(m)$, namely
$$
  k_A(m)
    :=
  \Bar_{A[t]}(A)\otimes_{A[t],\phi_{\pi^m}}A.
$$
Explicitly,
$$ 
  k_A(m): 
  \Delta^\op 
    \to 
  \CAlg_{A/}
$$
$$
  [p]
    \mapsto 
  \bigt{k_A(m)}_p
    =
  \frac{A\bigq{t^{(j)}:j\in [p+1]}}{I_p(m)}
    \simeq 
  A[t^{(1)},\dots,t^{(p)}],
$$
where 
$$
  I_p(m)
    := 
  \Angles{t^{(0)}-\pi^m, t^{(p+1)}}
    \subseteq 
  A\bigq{t^{(j)}:j\in [p+1]}.
$$
For $p\geq 0$ and $0\leq i\leq p+1$, the face maps 
$$
  d_i^p:\bigt{k_A(m)}_{p+1}
    =
  \frac{A\bigq{t^{(j)}:j\in [p+2]}}{I_{p+1}(m)}
    \to 
  \bigt{k_A(m)}_{p}
    =
  \frac{A\bigq{t^{(j)}:j\in [p+1]}}{I_{p}(m)}
$$
are defined on the generators $t^{(j)}$ ($j\in [p+2]$) by
$$
    d_i^p(t^{(j)})
      := 
    \begin{cases}
      t^{(j)} & \text{if} \;\; j\leq i;
        \\
      t^{(j-1)} & \text{if} \;\; j\geq i+1.
    \end{cases}
$$
These morphisms are well-defined since $d_i^p\bigt{I_{p+1}(m)}\subseteq I_{p}(m)$: indeed, 
$$
  d^p_i(t^{(0)}-\pi^m) 
    = 
  t^{(0)}-\pi^m,
  \;\;\;\;
  d^p_i(t^{(p+2)})
    = 
  t^{(p+1)}
$$
for all $0\leq i \leq p+1$.

Let $p\geq 1$ and $0\leq i\leq p-1$.
The degeneracy maps are defined as follows:
$$
  s_i^p:
  \bigt{k_A(m)}_{p-1}
    =
  \frac{A\bigq{t^{(j)}:j\in [p]}}{I_{p-1}(m)}
    \to 
  \bigt{k_A(m)}_{p}
    =
  \frac{A\bigq{t^{(j)}:j\in [p+1]}}{I_{p}(m)}
$$
$$
  s^p_i(t^{(j)})
    :=
   \begin{cases}
     t^{(j)} & \text{if} \;\; j\leq i;
       \\
     t^{(j+1)} & \text{if} \;\; j\geq i+1.
   \end{cases}
$$
Similarly to the above, one checks that the degeneracy maps are well-defined.

\begin{rmk}
Note that, $d^p_0(t^{(j)})=t^{(j-1)}$. In particular, 
$$
  d^p_0(t^{(1)})
    =
  t^{(0)}
    =
  \pi^m.
$$
On the other hand, notice that $d^p_{p+1}(t^{(j)})=t^{(j)}$. 
In particular, 
$$
  d^p_{p+1}(t^{(p+1)})
    =
  t^{(p+1)}
    =
  0.
$$
This means that, in the simplicial $A$-algebra $k_A(m)$, the element $\pi^m$ is homotopic to zero.
\end{rmk}

\sssec{Notation}

For each $n\geq 0$, we define 
$$
  \angles{n}
    := 
  \{1,\dots,n\} 
    \simeq 
  [n]\setminus \{0\}.
$$

\sssec{The simplicial algebra \texorpdfstring{$k_A^{\otimes_An}(m)$}{K_A^n(m)}}
\label{sssec: simplicial algebra K_A^(otimes_A n)(m)}

For $m,n \in \bbN$, we will need to consider the simplicial $A$-algebras
$$
  k_A^{\otimes_An}(m)
    := 
  \bigotimes_{l=1}^nk_A(m),
$$
where the tensor product is taken over $A$.

Notice that 
$$
  s(m)^{\times_Sn}
    = 
  \Spec{k_A^{\otimes_An}(m)}.
$$
One can verify that this simplicial $A$-algebra is as follows:
$$
  k_A^{\otimes_An}(m):
  \Delta^\op 
    \to 
  \CAlg_{A/}
$$
$$
  [p]
    \mapsto 
  \bigotimes_{l=1}^n\bigt{k_A(m)}_p 
    \simeq 
  \frac{A\bigq{t_a^{(j)}: j\in [p+1], a\in \angles{n}}}{I_p^{\otimes_A n}(m)},
$$
where 
$$
  I_p^{\otimes_A n}(m)
    =
  \Angles{t_a^{(0)}-\pi^m,t_a^{(p+1)}:a\in \angles{n}}
    \subseteq 
  A\bigq{t_a^{(j)}: j\in [p+1], a\in \angles{n}}.
$$
The face maps are as follows: 
for $p \in \bbN$ and $0\leq i \leq p,$ we have
$$
  d_i^p: 
  \frac{A\bigq{t_a^{(j)}: j\in [p+2], a\in \angles{n}}}{I_{p+1}^{\otimes_A n}(m)}
    \to 
  \frac{A\bigq{t_a^{(j)}: j\in [p+1], a\in \angles{n}}}{I_{p}^{\otimes_A n}(m)}
$$
$$
    d_i^p(t^{(j)}_a)
      := 
    \begin{cases}
      t^{(j)}_a & \text{if} \;\; j\leq i;
        \\
      t^{(j-1)}_a & \text{if} \;\; j\geq i+1.
    \end{cases}
$$
The degeneracy maps are as follows. For $0\leq i \leq p-1$, we have 
$$
  s_i^p:
  \frac{A\bigq{t_a^{(j)}: j\in [p], a\in \angles{n}}}{I_{p-1}^{\otimes_A n}(m)}
    \to 
  \frac{A\bigq{t_a^{(j)}: j\in [p+1], a\in \angles{n}}}{I_{p}^{\otimes_A n}(m)}
$$
$$
  s_i^p(t^{(j)}_a)
    := 
  \begin{cases}
    t^{(j)}_a & \text{if} \;\; j\leq i;
      \\
    t^{(j+1)}_a & \text{if} \;\; j\geq i+1.
  \end{cases}
$$
Notice that the face and degeneracy maps of $k_A^{\otimes_An}(m)$ are given by the tensor products of those of $k_A(m)$.
In particular, they are well-defined.

\sssec{Simplicial models for the canonical maps \texorpdfstring{$j(m,m')^{\times_Sn}$}{}}

For every $m\leq m'$ 
we provide a morphism of simplicial $A$-algebras $k_A^{\otimes_An}(m')\to k_A^{\otimes_An}(m)$ which models the morphisms $j(m,m')^{\times_Sn}$.

For each $p \geq 0$, consider the following morphism of $A$-algebras:
\begin{equation}\label{eqn: simplicial model j(m,m')^n}
    \bigt{k_A^{\otimes_An}(m')}_p
    = 
    \frac{A\bigq{t^{(j)}_a:j \in [p+1],a\in \angles{n}}}{I_{p}^{\otimes_A n}(m')}
    \to
    \bigt{k_A^{\otimes_An}(m)}_p
    = 
    \frac{A\bigq{t^{(j)}_a:j \in [p+1],a\in \angles{n}}}{J_{p}^{\otimes_A n}(m)}
\end{equation}
$$
   t^{(j)}_a \mapsto \pi^{m'-m} \; t^{(j)}_a.
$$

\begin{lem}

The above morphisms \eqref{eqn: simplicial model j(m,m')^n} define a morphism
$$
  \can(m',m)^{\otimes_A n}: 
  k_A^{\otimes_An}(m')
    \to 
  k_A^{\otimes_An}(m)
$$
of simplicial $A$-algebras.
\end{lem}

\begin{proof}
We need to verify that the morphisms \eqref{eqn: simplicial model j(m,m')^n} are compatible with the face and degeneracy maps.

\sssec*{Degeneracy maps} 
For $p\geq 0$ and $0\leq i \leq p$, we need to verify that the square
\begin{equation*}
    \begin{tikzcd}
    \bigt{k_A^{\otimes_An}(m')}_p=\frac{A\bigq{t^{(j)}_a:j \in [p+1],a\in \angles{n}}}{I_{p}^{\otimes_A n}(m')}
    \arrow[r,"s^{p+1}_i"]
    \arrow[d,"\eqref{eqn: simplicial model j(m,m')^n}"]
    & 
    \bigt{k_A^{\otimes_An}(m')}_{p+1}=\frac{A\bigq{t^{(j)}_a:j \in [p+2],a\in \angles{n}}}{I_{p+1}^{\otimes_A n}(m')}
    \arrow[d,"\eqref{eqn: simplicial model j(m,m')^n}"]
    \\
    \bigt{k_A^{\otimes_An}(m)}_p=
    \frac{A\bigq{t^{(j)}_a:j \in [p+1],a\in \angles{n}}}{I_{p}^{\otimes_A n}(m)}
    \arrow[r,"s^{p+1}_i"]
    &
    \bigt{k_A^{\otimes_An}(m)}_{p+1}=
    \frac{A\bigq{t^{(j)}_a:j \in [p+2],a\in \angles{n}}}{I_{p+1}^{\otimes_A n}(m)}
    \end{tikzcd}
\end{equation*}
commutes. Indeed, it is easy to check that both compositions send
$$
  t^{(j)}_a 
    \mapsto 
  \begin{cases}
    \pi^{m'-m} \; t^{(j)}_a & \text{if } j\leq i;
      \\
    \pi^{m'-m} \; t^{(j+1)}_a & \text{if } j\geq i+1.
  \end{cases}
$$

\sssec*{Face maps} 

For $p\geq 0$ and $0\leq i \leq p+1$, we need to veify that the square
\begin{equation*}
    \begin{tikzcd}
    \bigt{k_A^{\otimes_An}(m')}_{p+1}=\frac{A\bigq{t^{(j)}_a:j \in [p+2],a\in \angles{n}}}{I_{p+1}^{\otimes_A n}(m')}
    \arrow[r,"d^{p}_i"]
    \arrow[d,"\eqref{eqn: simplicial model j(m,m')^n}"]
    & 
    \bigt{k_A^{\otimes_An}(m')}_{p}=\frac{A\bigq{t^{(j)}_a:j \in [p+1],a\in \angles{n}}}{I_{p}^{\otimes_A n}(m')}
    \arrow[d,"\eqref{eqn: simplicial model j(m,m')^n}"]
    \\
    \bigt{k_A^{\otimes_An}(m)}_{p+1}=\frac{A\bigq{t^{(j)}_a:j \in [p+1],a\in \angles{n}}}{I_{p+1}^{\otimes_A n}(m)}
    \arrow[r,"d^{p}_i"]
    &
    \bigt{k_A^{\otimes_An}(m)}_{p}=\frac{A\bigq{t^{(j)}_a:j \in [p+1],a\in \angles{n}}}{I_{p}^{\otimes_A n}(m)}
    \end{tikzcd}
\end{equation*}
commutes. In this case, both compositions send 
$$
  t^{(j)}_a \mapsto 
  \begin{cases}
    \pi^{m'-m} \; t^{(j)}_a & \text{if } j\leq i;
      \\
    \pi^{m'-m} \; t^{(j-1)}_a & \text{if } j\geq i+1.
  \end{cases}
$$
\end{proof}

\begin{rmk}
Recall that the cyclic group $\angles{\rho_n}\simeq C_n$ acts on $s(m)^{\times_Sn}$.
At the level of simplicial $A$-algebras, the action of $\rho_n^b$ ($b\in [n-1]$) corresponds to the isomorphisms
$$
  \bigt{k_A^{\otimes_An}(m)}_{p}
    =
  \frac{A\bigq{t^{(j)}_a:j \in [p+1],a\in \angles{n}}}{I_{p}^{\otimes_A n}(m)}
    \to
  \bigt{k_A^{\otimes_An}(m)}_{p}
    =
  \frac{A\bigq{t^{(j)}_a:j \in [p+1],a\in \angles{n}}}{I_{p}^{\otimes_A n}(m)}
$$
$$
  t^{(j)}_a 
    \mapsto 
  t^{(j)}_{\rho_n^b(a)}.
$$
It will be clear later that the compositions 
$$
  \can(m',m)^{\otimes_A n}\circ \rho_n^b: 
  k_A^{\otimes_An}(m')
    \to 
  k_A^{\otimes_An}(m)
$$
play a crucial role in the main construction in Section \ref{sec: integrating differential forms on two-periodic complexes}. 
\end{rmk}

\sssec{The simplicial enrichment}
\label{sssec: the simplicial enrichment}

Recall that the category $\sCAlg_{A/}$ of simplicial commutative $A$-algebras is naturally tensored over the category of simplicial sets: for each $K \in \sSet$ and $R \in \sCAlg_{A/}$ one has a simplicial $A$-algebra $K\otimes R$.

We will be particularly interested in the case $K=\Delta^l$ and $R=k_A^{\otimes_An}(n)$, so we write down the explicit formulas in this case. The simplicial $A$-algebra 
$$
  \Delta^l \otimes k_A^{\otimes_An}(n): 
  \Delta^\op 
    \to 
  \CAlg_{A/}
$$
is defined on objects by
$$
  [p]
    \mapsto 
  \bigotimes_{\Delta^l_p}\bigt{k_A^{\otimes_An}(n)}_p 
    \simeq 
  \frac{A\bigq{t_a^{(j)}(\alpha):j \in [p+1], \alpha \in \Delta^l_p, a\in \angles{n}}}{I^{\otimes_An}_{\Delta^l,p}(n)},
$$
where 
\begin{align*}
  I^{\otimes_An}_{\Delta^l,p}(n)
  & 
  := 
  \Angles{t^{(0)}_a(\alpha)-\pi^n,t^{(p+1)}_a(\alpha): \alpha \in \Delta^l_p, a \in \angles{n}} \\
  & 
  \subseteq 
  A\bigq{t^{(j)}(\alpha):j \in [p+1], \alpha \in \Delta^l_p}.
\end{align*}
The face maps
$$
  \frac{A\bigq{t^{(j)}(\alpha):j \in [p+1], \alpha \in \Delta^l_p,a \in \angles{n}}}{I_{\Delta^l,p}^{\otimes_A n}(n)}
    \to
  \frac{A\bigq{t^{(j)}(\alpha):j \in [p], \alpha \in \Delta^l_{p-1},a\in \angles{n}}}{I^{\otimes_An}_{\Delta^l,p-1}(n)}
$$
are determined by the assignments
$$
  t^{(j)}_a(\alpha)
    \mapsto
  d^{p-1}_i(t^{(j)}_a)(\alpha \circ \delta^p_i),
$$
while the degeneracy maps
$$
  \frac{A\bigq{t^{(j)}_a(\alpha):j \in [p+1],\alpha \in \Delta^l_p, a\in \angles{n}}}{I^{\otimes_An}_{\Delta^l,p}(n)}
    \to
  \frac{A\bigq{t^{(j)}(\alpha):j \in [p+2],\alpha \in \Delta^l_{p+1},a\in \angles{n}}}{I^{\otimes_An}_{\Delta^l,p+1}(n)}
$$
are determined by the assignments
$$
  t^{(j)}_a(\alpha)
  \mapsto
  s^{p+1}_i(t^{(j)}_a)(\alpha \circ \sigma^p_i).
$$

\sssec{Cofibrant model for the \texorpdfstring{$A$}{A}-algebra \texorpdfstring{$\bigt{k\otimes_A^{\bbL}k}^{\otimes_An}$}{}}

Since we will use it later on, we also describe a cofibrant model for the simplicial $A$-algebra
$$
  \bigt{k\otimes_A^{\bbL}k}^{\otimes_An}
  \simeq 
  \bigt{
        k_A\otimes_Ak_A
        }^{\otimes_An}.
$$
Denote by $K_{A}$ the simplicial $A$-algebra
$$
  k_A \otimes_A k_A
    \simeq 
  k \otimes_A^{\bbL} k.
$$
Since $k_A$ is a cofibrant simplicial $A$-algebra, $K_A$ is cofibrant as well. Explicitly, we have
$$
  K_A: 
  \Delta^{\op}
    \to 
  \CAlg_{A/}
$$

$$
  [p]
    \mapsto 
  \bigt{K_A}_p
    =
  \frac{A\bigq{t^{(j)}:j\in [p+1]}}{I_p}\otimes_A\frac{A\bigq{t^{(j)}:j\in [p+1]}}{I_p}
    \simeq 
  \frac{A\bigq{t^{(j)},u^{(j)}:j \in [p+1]}}{J_{p}}
,
$$
where 
$
  J_{p}
    :=
  \Angles{t^{(0)}-\pi,t^{(p+1)},u^{(0)}-\pi,u^{(p+1)}}
    \subseteq 
  A\bigq{t^{(j)},u^{(j)}:j \in [p+1]}.
$
The face and degeneracy maps on $K_A$ are then given by the tensor products of the face and degeneracy maps on $k_A$.

\sssec{} 

Notice that 
$$
  G^{\times_Sn}
    = 
  \Spec{K_A^{\otimes_A n}},
$$
where
$$
  K_A^{\otimes_A n}
    := 
  \bigotimes_{l=1}^nK_A,
$$
with the tensor product taken over $A$.
Explicitly, this simplicial $A$-algebra looks as follows:
$$
  K_A^{\otimes_A n}: 
  \Delta^\op 
    \to 
  \CAlg_{A/}
$$
$$
  [p]
    \mapsto 
  \bigotimes_{l=1}^n\bigt{K_A^{\otimes_A n}}_p 
    \simeq 
  \frac{A\bigq{t_a^{(j)}, u_{a}^{(j)}: j\in [p+1], a\in \angles{n}}}{J_p^{\otimes_A n}},
$$
where 
$
  J_p^{\otimes_A n}
    =
  \Angles{t_a^{(0)}-\pi,t_a^{(p+1)},u_a^{(0)}-\pi,u_a^{(p+1)}:a\in \angles{n}}
    \subseteq 
  A\bigq{t_a^{(j)}, u_{a}^{(j)}: j\in [p+1], a\in \angles{n}}.
$
The face and degeneracy maps of $K_A^{\otimes_A n}$ are given by the tensor products of those of $K_A$.

The face maps are as follows: 
for $p \in \bbN$ and $0\leq i \leq p,$ we have
$$
  d_i^p: 
  \frac{A\bigq{t_a^{(j)}, u_{a}^{(j)}: j\in [p+2], a\in \angles{n}}}{J_{p+1}^{\otimes_A n}}
    \to 
  \frac{A\bigq{t_a^{(j)}, u_{a}^{(j)}: j\in [p+1], a\in \angles{n}}}{J_{p}^{\otimes_A n}}
$$
$$
    d_i^p(t^{(j)}_a):= 
    \begin{cases}
    t^{(j)}_a & \text{if} \;\; j\leq i;
    \\
    t^{(j-1)}_a & \text{if} \;\; j\geq i+1.
    \end{cases}
    \;\;\;\;
    d_i^p(u^{(j)}_a):= 
    \begin{cases}
    u^{(j)}_a & \text{if} \;\; j\leq i;
    \\
    u^{(j-1)}_a & \text{if} \;\; j\geq i+1.
    \end{cases}
$$
The degeneracy maps are as follows. For $0\leq i \leq p-1$, we have 
$$
  s_i^p:
  \frac{A\bigq{t_a^{(j)}, u_{a}^{(j)}: j\in [p], a\in \angles{n}}}{J_{p-1}^{\otimes_A n}}
    \to 
  \frac{A\bigq{t_a^{(j)}, u_{a}^{(j)}: j\in [p+1], a\in \angles{n}}}{J_{p}^{\otimes_A n}}
$$
$$
  s_i^p(t^{(j)}_a):= 
  \begin{cases}
    t^{(j)}_a & \text{if} \;\; j\leq i;
    \\
    t^{(j+1)}_a & \text{if} \;\; j\geq i+1.
    \end{cases}
      \;\;\;\;
  s_i^p(u^{(j)}_a):=
  \begin{cases}
    u^{(j)}_a & \text{if} \;\; j\leq i;
    \\
    u^{(j+1)}_a & \text{if} \;\; j\geq i+1.
    \end{cases}
$$

\sssec{Cofibrant model for the \texorpdfstring{$A$}{A}-algebra \texorpdfstring{$K_A^{\otimes_{k_A}m}$}{}}

The cofibrant simplicial $A$-algebra $K_A$ immediately provides us a cofibrant model for
$$
  K_A^{\otimes_{k_A}m}
    :=
  K_A\otimes_{i_u,k_A,i_t}K_A\otimes_{i_u,k_A,i_t} \dots \otimes_{i_u,k_A,i_t}K_A,
$$
where $i_t:k_A \to K_A$ (resp. $i_u:k_A \to K_A$) denotes the morphism induced by the embedding of the variables of type \virg{$t$} (resp. of type \virg{$u$}).

Indeed, we have that
$$
  K_A^{\otimes_{k_A}m}
    =
  \frac{K_A^{\otimes_An}}{\Angles{t^{(\star)}_{a+1}-u^{(\star)}_{a}:a=1,\dots,n-1}},
$$
where 
$$
  \Angles{t^{(\star)}_{a+1}-u^{(\star)}_{a}:a=1,\dots,n-1}
    \subseteq 
  K_A^{\otimes_A n}
$$
denotes the obvious simplicial ideal.

\begin{rmk}\label{rmk: can quot K_A^(otimes_k_A n) --> k_A^(otimes_A n)}

Notice that there is an isomorphims of simplicial $A$-algebra
$$
k_A^{\otimes_An}
\simeq
    \frac{K_A^{\otimes_{k_A}n}}{\Angles{t^{(\star)}_{1}-u^{(\star)}_{n}}}.
  $$
This gives a model for the morphism of $S$-schemes $s^{\times_Sn}\to G^{\times_sn}$ sitting in the following cartesian square:
  \begin{equation*}
      \begin{tikzcd}
        s^{\times_Sn}
        \rar
        \dar["\pr{1}"]
        &
        G^{\times_sn}
        \simeq s^{\times_S n+1}
        \dar["\pr{1,n+1}"]
        \\
        s
        \rar["\delta_{s/S}"]
        &
        s\times_Ss.
      \end{tikzcd}
  \end{equation*}
\end{rmk}

\subsection{A technical result}\label{ssec: some combinatorics}

To goal of this final section is to record the technical Corollary \ref{cor:lemmino}.
To state it, we first need to introduce some more notation and collect some elementary combinatorics.

\begin{defn}
Let $\alpha : [p]\to [n]$ be a morphism in $\Delta$. 
For $k\in [n]$, we set
$$
  m_{\alpha}(k)
    := 
  \# \alpha^{-1}(k)
$$
and
$$
  \ccS_{\alpha}(k)
    := 
  \sum_{i=0}^km_{\alpha}(i)
  = 
  \# \big \{ j\in [p]: \alpha(j)\leq k \big \}.
$$
We use the convention that $m_{\alpha}(-1)=\ccS_{\alpha}(-1)=0$.
\end{defn}

\begin{lem}
Let $\alpha : [p]\to [n]\in \Delta^{n}_p$ and $0\leq i \leq p$. We have:
\begin{enumerate}
    \item 
    $$
      m_{\alpha \circ \sigma^p_i}(k)
        = 
      \begin{cases}
      m_{\alpha}(k)+1 & \text{if }
        \;\;\;\; 
      \ccS_{\alpha}(k-1)\leq i \leq  \ccS_{\alpha}(k)-1,
        \\
      m_{\alpha}(k) & \text{else.}
    \end{cases}
    $$
    \item 
    $$
      m_{\alpha \circ \delta^p_i}(k)= 
      \begin{cases}
      m_{\alpha}(k)-1 & \text{if }
        \;\;\;\; 
      \ccS_{\alpha}(k-1)\leq i \leq \ccS_{\alpha}(k)-1,
        \\
      m_{\alpha}(k) & \text{else.}
    \end{cases}
    $$
\end{enumerate}
\end{lem}

\begin{proof}
By definition, $\alpha: [p]\to [n]$ sends $j \mapsto k$ if
$$
  \ccS_{\alpha}(k-1)\leq j \leq  \ccS_{\alpha}(k)-1.
$$
Since $\sigma^p_i:[p+1]\to [p]$ is the unique surjective morphism in $\Delta$ which hits the element $i\in[p]$ twice, it is clear that 
$$
    m_{\alpha \circ \sigma^p_i}(k)
      = 
    \begin{cases}
      m_{\alpha}(k)+1 & \text{if }\;\;\;\; i\in \alpha^{-1}(k)
        \\
      m_{\alpha}(k) & \text{if }\;\;\;\; i\notin \alpha^{-1}(k).
    \end{cases}
$$
The first claim follows since
$$
  \alpha^{-1}(k)
    =
  \big \{ j \in [n]: \ccS_{\alpha}(k-1)\leq j < \ccS_{\alpha}(k)\big \}.
$$
The second claim is treated similarly, as $\delta^p_i:[p-1]\to [p]$ is the unique injective morphism in $\Delta$ which misses the element $i \in [p]$.
\end{proof}

\begin{cor}\label{cor: S(alpha circ sigma/delta) vs S(alpha)}
Let $\alpha:[p]\to [n] \in \Delta^n_p$. Let $i\in [p]$ and let $k$ be the unique element such that
$$
\ccS_{\alpha}(k-1)\leq i \leq \ccS_{\alpha}(k)-1.
$$
\begin{enumerate}
    \item 
    $$
      \ccS_{\alpha \circ \sigma^p_i}(u)= 
      \begin{cases}
        \ccS_{\alpha}(u) & \text{if }
          \;\;\;\; 
        u\leq k-1,
          \\
        \ccS_{\alpha}(u)+1 & \text{if } \;\;\;\; u\geq k.
    \end{cases}
    $$
    \item 
    $$
      \ccS_{\alpha \circ \delta^p_i}(u)= 
      \begin{cases}
      \ccS_{\alpha}(u) & \text{if }
        \;\;\;\; 
      u\leq k-1,
        \\
      \ccS_{\alpha}(u)-1 & \text{if } 
        \;\;\;\; 
      u\geq k.
    \end{cases}
    $$
\end{enumerate}
\end{cor}

\begin{cor} \label{cor:lemmino}
Let $\alpha:[p]\to [n] \in \Delta^n_p$. 
Let $i\in [p]$ and $r\in [n]$. Then the face and degeneracy maps of the simplicial algebra $\Delta^{n}\otimes k_A^{\otimes_Am'}(m)$ satisfy the equations ($a=1,\dots, m'$)
$$
  d^p_i \bigt{ t^{(\ccS_\alpha(r))}_a}
    = 
  t^{(\ccS_{\alpha \circ \delta_i}(r))}_a,
$$
$$
  s^p_i \bigt{ t^{(\ccS_\alpha(r))}_a}
    = 
  t^{(\ccS_{\alpha \circ \sigma_i}(r))}_a.
$$
\end{cor}


\section{Action of the cyclic groups and higher homotopies}\label{sec: the E_2 structure on G}

This section is the very heart of the paper.
In a nutshell, we study the failure of the multiplicative structure on $G$ to be commutative.
The key observation, upon which several later constructions rely, is that the obstruction to the existence of such additional structure can be overcome by considering higher and higher powers of the uniformizing element.

\subsection{The coherent homotopies}\label{ssec: the coherent homotopies}

\sssec{}

Our goal now is to exhibit, for each $n \geq 1$, a map of simplicial $A$-algebras
$$
  \wt{\ffh}^{(n-1)}:
  \Delta^{n-1} \otimes k_A^{\otimes_An}(n)
    \longto
  k_A^{\otimes_A n}
    =
  \underbrace{k_A \otimes_A \cdots \otimes_A k_A}_{n\textup{ copies of } k_A}
$$
which witnesses the fact that the compositions ($0\leq j \leq n-1$)
$$
  k_A^{\otimes_An}(n)
    \xto{\can(n,1)^{\otimes_An}}
  k_A^{\otimes_A n}
    \xto{\rho_n^{j}}
  k_A^{\otimes_A n}
$$
are \emph{coherently} homotopic.

\sssec{} 

At the level of derived schemes, this means that for every $n\geq 1$ there exists a morphism of derived $S$-schemes
$$
  \wt{\ffh}^{(n-1)}:
  \Delta^{n-1}\otimes s^{\times_S n}
    \to
  s(n)
$$
such that the induced morphisms
$$
  \Delta^{\{k\}}\otimes s^{\times_Sn}
    \hto
  \Delta^{n-1}\otimes s^{\times_S n}
    \to
  s^{\times_Sn}(n)
$$
identify with
$$
  s^{\times_Sn}
    \xto{\rho_n^k}
  s^{\times_Sn}
    \xto{j(n)^{\times_Sn}}
  s^{\times_Sn}(n).
$$

\sssec{}

We are now ready to state the main theorem of this section.

\begin{thm}\label{thm: the homotopy}

For every $n\geq 1$, there is a morphism of simplicial commutative rings
$$
  \wt{\ffh}^{(n-1)}: 
  \Delta^{n-1}\otimes k_A^{\otimes_An}(n)
    \to 
  k_A^{\otimes_An}
$$
such that, for every $0\leq i \leq n-1$, the composition
$$
  \Delta^{\{i\}}\otimes k_A^{\otimes_An}(n) 
    \hto 
  \Delta^{n-1}\otimes k_A^{\otimes_An}(n)\to k_A^{\otimes_An}
$$
identifies with
$$
  k_A^{\otimes_An}(n)
    \xto{\on{can}(n,1)^{\otimes_An}}
  k_A^{\otimes_A n}
    \xto{\rho_n^i}
  k_A^{\otimes_A n}.
$$
\end{thm}

\begin{proof}

We proceed in steps.

\sssec*{Definition of the morphism \texorpdfstring{$\wt{\ffh}^{(n-1)}_p$}{}}

To define the morphism
$$
  \wt{\ffh}^{(n-1)}:
  \Delta^{n-1} \otimes k_A^{\otimes_An}(n)
    \longto
  k_A^{\otimes_A n},
$$
we need to define a morphism of $A$-algebras
$$
  \wt{\ffh}^{(n-1)}_p:
  \Bigt{
        \Delta^{n-1} \otimes k_A^{\otimes_An}(n)
        }_p
    \longto
  \Bigt{
        k_A^{\otimes_A n}
        }_p
$$
for each $p\geq 0$. 
(In the next steps, we will verify the required compatibilities with the face and degeneracy maps.)

Since $\wt{\ffh}^{(n-1)}_p$ must be $A$-linear, it suffices to define its value on the generators (see Section \ref{sssec: the simplicial enrichment}): in view of the definition of the ideal $I^{\otimes_An}_{\Delta^{n-1},p}(n)$, we just need to specify the elements
$$
  \wt{\ffh}^{(n-1)}_p
  \bigt{ 
        t^{(j)}_a(\alpha)
        }
$$
for $j = 1, \ldots, p$, $a=1,\dots, n$ and $\alpha:[p]\to [n-1]$.

Recall that $\rho_n$ denotes the generator
$(1\mapsto 2 \mapsto \cdots \mapsto n \mapsto 1)$
of the cyclic group $C_n$.
We set
\begin{equation*} \label{eqn:mostro}
  \wt{\ffh}^{(n-1)}_p
  \bigt{
        t^{(j)}_a(\alpha)
       }
    =
  \sum_{r=1}^{n}
  \left[
        t^{(j)}_{\rho^{r-1}_n(a)}
        \left( 
              \prod_{i = 0}^{r-2}
              t_{\rho^{i}_n(a)}^{(\ccS_\alpha(i))}
              \right) 
   \Bigt{ 
         \pi - t_{\rho^{r-1}_n(a)}^{(\ccS_\alpha(r-1))}
         }
\pi^{n-r-1}
\right ],
\end{equation*}
The above displayed formulas are well-defined, provided we employ the following conventions: the empty product equals $1$; when $r=n$, the factor $\pi^{-1}$ appears but it is multiplied by
$$
  \Bigt{
        \pi-t^{(\ccS_{\alpha}(n-1))}_{\rho^{n-1}_{n}(a)}
        }
    =
  \pi,
$$
since by definition 
$$
  t^{(\ccS_{\alpha}(n-1))}_{\rho^{n-1}_{n}(a)}
    =
  t^{(p+1)}_{\rho^{n-1}_{n}(a)}
    =
  0.
$$
In particular, the precise definition is 
\begin{equation*} \label{eqn:mostro 2}
  \wt{\ffh}^{(n-1)}_p
  \bigt{
        t^{(j)}_a(\alpha)
        }
    =
  \sum_{r=1}^{n-1}
  \left[
        t^{(j)}_{\rho^{r-1}_n(a)}
        \left( 
              \prod_{i = 0}^{r-2}
              t_{\rho^{i}_n(i)}^{(\ccS_\alpha(i))}
              \right) 
        \Bigt{ 
              \pi - t_{\rho^{r-1}_n(a)}^{(\ccS_\alpha(r-1))}
              }
  \pi^{n-r-1}
  \right ]
    +
  \left[
        t^{(j)}_{\rho^{n}_n(a)}
        \left( 
              \prod_{i = 0}^{n-2}
              t_{\rho^{i-1}_n(a)}^{(\ccS_\alpha(i))}
              \right) 
  \right ]
.
\end{equation*}

\sssec*{Compatibility with the degeneracy maps}

Next, we verify that $\wt{\ffh}^{(n-1)}$ is compatible with the degeneracy maps. 
By the definition and Corollary \ref{cor:lemmino}, the value of $s^p_l \circ \wt{\ffh}^{(n-1)}_p$ on $t^{(j)}_a(\alpha)$ is given by the formula
$$
  \sum_{r=1}^{n}
  \left[
        s^p_l( t^{(j)}_{\rho^{r-1}_n(a)} )
        \left( 
              \prod_{i = 0}^{r-2}
              t_{\rho^{i}_n(a)}^{(\ccS_{\alpha \circ \sigma^p_l}(i))}
        \right) 
        \Bigt{ 
              \pi - t_{\rho^{r-1}_n(a)}^{(\ccS_{\alpha \circ \sigma^p_l}(r-1))}
              }
        \pi^{n-r-1}
  \right ]
    =
$$
$$
  \sum_{r=1}^{n}
  \left[
        \bigt{s^p_l( t^{(j)})}_{\rho^{r-1}_n(a)} 
        \left( 
              \prod_{i = 0}^{r-2}
              t_{\rho^{i}_n(a)}^{(\ccS_{\alpha \circ \sigma^p_l}(i))}
        \right) 
        \Bigt{ 
              \pi - t_{\rho^{r-1}_n(a)}^{(\ccS_{\alpha \circ \sigma^p_l}(r-1))}
              }
        \pi^{n-r-1}
  \right ].
$$
This is written more compactly as
$$
  \wt{\ffh}^{(n-1)}_{p+1}
  \Bigt{
        s^p_l(t^{(j)}_a)
         (\alpha \circ \sigma^p_l)
        }.
$$
It remains to observe that
$$
  \wt{\ffh}^{(n-1)}_{p+1} \circ s^p_l
  \bigt{
         t^{(j)}_a(\alpha)
        }
    = 
  \wt{\ffh}^{(n-1)}_{p+1}
  \Bigt{
        s^p_l(t^{(j)}_a)
        (\alpha \circ \sigma^p_l)
        }
$$
by the definition of the degeneracy maps of $\Delta^{n-1} \otimes k_A^{\otimes_An}(n)$.
In other words, we have proven that
$$
  s^p_l\circ \wt{\ffh}^{(n-1)}_p\bigt{t^{(j)}_a(\alpha)}
    =
  \wt{\ffh}^{(n-1)}_{p+1}\circ s^p_l\bigt{t^{(j)}_a(\alpha)}.
$$

\sssec*{Compatibility with the face maps}

Thanks to Corollary \ref{cor:lemmino} and the definitions, the value of $d^p_l \circ \wt{\ffh}^{(n-1)}_p$ on $t^{(j)}_a(\alpha)$ is given by the formula
$$
  \sum_{r=1}^{n}
  \left[
        d^p_l( t^{(j)}_{\rho^{r-1}_n(a)} )
        \left( 
              \prod_{i = 0}^{r-2}
              t_{\rho^{i}_n(a)}^{(\ccS_{\alpha \circ \delta^p_l}(i))}
        \right) 
        \Bigt{ 
              \pi - t_{\rho^{r-1}_n(a)}^{(\ccS_{\alpha \circ \delta^p_l}(r-1))}
             }
             \pi^{n-r-1}
  \right ]
    =
$$
$$
  \sum_{r=1}^{n}
  \left[
        \bigt{d^p_l( t^{(j)})}_{\rho^{r-1}_n(a)} 
        \left( 
              \prod_{i = 0}^{r-2}
              t_{\rho^{i}_n(a)}^{(\ccS_{\alpha \circ \delta^p_l}(i))}
        \right) 
        \Bigt{ 
              \pi - t_{\rho^{r-1}_n(a)}^{(\ccS_{\alpha \circ \delta^p_l}(r-1))}
              }
        \pi^{n-r-1}
  \right ],
$$
which equals
$$
  \wt{\ffh}^{(n-1)}_{p+1}
  \Bigt{
        d^p_l(t^{(j)}_a)
        (\alpha \circ \delta^p_l)
        }
    =
  \wt{\ffh}^{(n-1)}_{p+1} \circ d^p_l
  \bigt{
        t^{(j)}_a(\alpha)
        }.
$$

\sssec*{Values of \texorpdfstring{$\wt{\ffh}^{(n-1)}$}{} on the vertices}

Let $k \in [n-1]$ and $\alpha:[p]\to [n-1]$ be the constant map with value $k$. 
In this case, the above formula simplifies to give 
\begin{equation*}
\wt{\ffh}^{(n-1)}_p
     \bigt{
           t^{(j)}_a(\alpha)
          }
= 
\pi^{n-1}  t_{\rho^{k}_n(a)}^{(j)}.
\end{equation*}
Indeed, for such an $\alpha$, we have
$$
  \ccS_{\alpha}(i)
    =
  \begin{cases}
    0 & \text{if } i \leq k-1;
      \\
    p+1 & \text{if } i \geq k
\end{cases}
$$
and thus
$$
  t^{(\ccS_{\alpha}(i))}_a
    =
  \begin{cases}
    t^{(0)}=\pi & \text{if } i \leq k-1;
      \\
    t^{(p+1)}=0 & \text{if } i \geq k.
  \end{cases}
$$
Then one immediately sees that the addendum 
\begin{equation}\label{eqn: addendum 2}
  t^{(j)}_{\rho^{r-1}_n(a)}
  \left( 
        \prod_{i = 0}^{r-2}
        t_{\rho^{i}_n(a)}^{(\ccS_\alpha(i))}
  \right) 
  \Bigt{ 
        \pi - t_{\rho^{r-1}_n(a)}^{(\ccS_\alpha(r-1))}
        }
   \pi^{n-r-1}
\end{equation}
vanishes unless $r=k+1$, where it gives the required value.
Indeed:
\begin{itemize}
\item if $r\leq k$, then 
$$
  \pi - t_{\rho^{r-1}_n(a)}^{(\ccS_\alpha(r-1))}
    =
  \pi - t_{\rho^{r-1}_n(a)}^{(0)}=0
$$ 
and \eqref{eqn: addendum 2} vanishes.

\item if $r=k+1$, then
\begin{align*}
  \eqref{eqn: addendum 2}
  &
    =
  t^{(j)}_{\rho^{k+1}_n(a)}
    \left( 
      \prod_{i = 0}^{k-1}
  t_{\rho^{i}_n(a)}^{(\ccS_\alpha(i))}
    \right) 
  \Bigt{ 
        \pi - t_{\rho^{k+1}_n(a)}^{(\ccS_\alpha(k))}
        }
  \pi^{n-k-2}
  \\
  &
    = 
  t^{(j)}_{\rho^{k+1}_n(a)}
  \pi^{k}
  \Bigt{ 
        \pi - t_{\rho^{k+1}_n(a)}^{(p+1)}
        }
  \pi^{n-k-2}
  \\
  &
    =
  \pi^{n-1}  t_{\rho^{k+1}_n(a)}^{(j)}.
\end{align*}
\item if $r \geq k+2$, then the element 
$$
  t_{\rho^{k}_n(a)}^{(\ccS_{\alpha}(k))}
    =
  t_{\rho^{k}_n(a)}^{(p+1)}
    =
  0
$$
appears in the product and \eqref{eqn: addendum 2} vanishes.
\end{itemize}
It follows that the composition
$$
  \Delta^{\{k\}}\otimes k_A^{\otimes_An}(n)
    \hto 
  \Delta^{n-1}\otimes k_A^{\otimes_An}(n)
    \xto{\wt{\ffh}^{(n-1)}}
  k_A^{\otimes_A n}
$$
identifies with
the composition
$$
  k_A^{\otimes_An}(n)
   \xto{\on{can}(n,1)^{\otimes_An}}
  k_A^{\otimes_A n}
   \xto{\rho^k_n}
  k_A^{\otimes_A n}
$$
$$
t^{(j)}_a
\mapsto 
\pi^{n-1} t^{(j)}_{\rho^k_n(a)}.
$$
\end{proof}

\sssec{} 

Next, we are going to deduce that the morphisms
$$
  s^{\times_Sn}
  \xto{\pr{k}}
  s
  \xto{j(n)}
  s(n)
$$
are coherently homotopic.

\begin{cor}
For each $n\geq 1$ there is a morphism of simplicial $A$-algebras
$$
  \ffh^{(n-1)}:
  \Delta^{n-1}\otimes k_A(n)
    \to
  k_A^{\otimes_An}
$$
such that, for each $k\in [n-1]$, the composition
$$
  \Delta^{\{k\}}\otimes k_A(n)
  \hto
  \Delta^{n-1}\otimes k_A(n)
    \xto{\ffh^{(n-1)}}
  k_A^{\otimes_An}
$$
is given by
$$
  k_A(n)
  \xto{\on{can}(n,1)}
  k_A
  \xto{k^{th}-\on{inclusion}}
  k_A^{\otimes_An}.
$$
\end{cor}

\subsection{The master homotopies}

The goal of this section is to provide the coherent homotopies necessary to construct a dg-functor out of the $n^{th}$-truncated cyclic bar complex of $\sB^+$.

\sssec{}

Let $m \geq 0$ and $k\geq 1$.
We fix a vector
$$
\ul d =(d_{m+1},\dots,d_{m+k}) \in \prod_{s=m+1}^{m+k}[s].
$$
For $c \geq 1$ and $r \in [c]$, we denote by 
$$
  \iota^{c-1}_{r}:
  \Delta^{c-1}
    \hto
  \Delta^{c}
$$
the morphism corresponding to $\delta^{c}_{r}:[c-1]\to [c]$ via the Yoneda embedding.

\sssec{Definition of the morphisms \texorpdfstring{$\ffi^{[a,b]}_{\ul d}$}{}}
\label{sssec: def of i[a,b}

Let $m\leq a \leq b \leq m+k$. We will write $\ffi^{[a,b]}_{\ul d}: \Delta^a \to \Delta^b$ to mean the
composition
$$
  \Delta^a
    \xto{\iota^a_{d_{a+1}}}
  \Delta^{a+1}
    \xto{\iota^{a+1}_{d_{a+2}}}
  \cdots
    \xto{\iota^{b-1}_{d_{b}}}
  \Delta^b.
$$
We use the convention that $\ffi^{[a,a]}_{\ul d}=\id$.
For $p\geq 0$ and $\alpha \in \Delta^a_p$, we will write 
$$
  \ffi^{[a,b]}_{\ul d,p}(\alpha)\in \Delta^{b}_p
$$
for the image of $\alpha$ along this map. 
Concretely we have that 
$$
  \ffi^{[a,b]}_{\ul d,p}(\alpha)
    =
  \delta^{b-1}_{d_{b}}
    \circ
  \cdots
    \circ
  \delta^{a}_{d_{a+1}}
    \circ
  \alpha:
  [p]
    \to
  [b].
$$

\sssec{Definition of the morphisms \texorpdfstring{$\wt{f}^{(c)}_r$}{} and \texorpdfstring{$f^{(c)}_r$}{}}
\label{sssec: def of f^(c)_r}

Let $c \geq 1$ and $r\in [c]$. Denote by $\varrho_c$ the map
$$
  [c]
    \to 
  \angles{c}
    =
  \{1,\dots,c\},
    \;\;\;\;
  i
    \mapsto
  \begin{cases}
    i+1 & \text{ if } i\leq c-1,
      \\
    1 & \text{ if } i=c. 
  \end{cases}
$$
We define morphisms 
$$
  \wt{f}^{(c)}_{r}:
  K_A^{\otimes_{k_A}c}
    \to
  K_A^{\otimes_{k_A}c+1}
$$
of simplicial $A$-algebras as follows.
For $q\geq 0$, we define
$$
  \bigt{\wt{f}^{(c)}_{r}}_q:
  (K_A^{\otimes_{k_A}c})_q
    \to
  (K_A^{\otimes_{k_A}c+1})_q
$$
on generators by
$$
  t^{(j)}_a 
    \mapsto
  \begin{cases}
  t^{(j)}_a & \text{ if } a\leq \varrho_c(r),
    \\
  t^{(j)}_{a+1} & \text{ if } a\geq \varrho_c(r)+1,
  \end{cases}
    \;\;\;\;
  u^{(j)}_a 
    \mapsto
  \begin{cases}
    u^{(j)}_a & \text{ if } a\leq \varrho_c(r)-1,
      \\
    u^{(j)}_{a+1} & \text{ if } a\geq \varrho_c(r).
  \end{cases}
$$
Recall that $a$ and $j$ range in $\{1, \ldots, c\}$.
It is easy to check that these morphisms of $A$-algebras $\bigl \{ \bigt{\wt{f}^{(c)}_{r}}_q \bigr \}_{q\geq 0}$ give rise to a morphism of simplicial $A$-algebras.
For instance, let us check they are compatible with the face and degeneracy maps: the face and degeneracy maps of $K_A^{\otimes_{k_A}c}$ and $K_A^{\otimes_{k_A}c+1}$ modify only the index $j$ of the generators $t^{(j)}_a$ and $u^{(j)}_a$, see Section \ref{sssec: simplicial algebra K_A^(otimes_A n)(m)}, while the morphisms $\bigt{\wt{f}^{(c)}_{r}}_q$ modify only the index $a$. Visibly, these operations commute.

Moreover, denote by
$$
  f^{(c)}_{r}:
  k_A^{\otimes_Ac}
    \to
  k_A^{\otimes_Ac+1}
$$
the morphism induced by $\wt{f}^{(c)}_r$ on the quotients
$$
  k_A^{\otimes_Ac}
    \simeq
  \frac{K_A^{\otimes_{k_A}c}}{\Angles{t_1^{(\star)}-u_c^{(\star)}}},
    \;\;\;\;
  k_A^{\otimes_Ac+1}
    \simeq
  \frac{K_A^{\otimes_{k_A}c}}{\Angles{t_1^{(\star)}-u_{c+1}^{(\star)}}}.
$$

\sssec{Explanation of the morphisms \texorpdfstring{$f^{(c)}_r$}{} for \texorpdfstring{$0\leq r \leq c-1$}{}
}
\label{rmk: morphims f^(c)_r}

For $0\leq r \leq c-1$, the morphism 
$$
\wt{f}^{(c)}_r:
K_A^{\otimes_{k_A}c}
\to
K_A^{\otimes_{k_A} c+1}
$$
\virg{skips} the generators $t^{(j)}_{r+2}=u^{(j)}_{r+1}$.
Therefore, it is easy to see that it sits in the commutative diagram of simplicial $A$-algebras

\begin{equation*}
    \begin{tikzcd}
    K_A^{\otimes_A c+1}
    \dar
    \rar
    &
    \frac{K_A^{\otimes_A c+1}}{\Angles{t^{(\star)}_{r+2}-u^{(\star)}_{r+1}}}
    \dar
    &
    \lar
    K_A^{\otimes_A c}
    \dar
    \\
    K_A^{\otimes_{k_A} c+1}
    \rar[swap,"\id"]
    \dar
    &
    K_A^{\otimes_{k_A} c+1}
    \dar
    &
    \lar["\wt{f}^{(c)}_r"]
    K_A^{\otimes_{k_A} c}
    \dar
    \\
    k_A^{\otimes_A c+1}
    \rar[swap,"\id"]
    &
    k_A^{\otimes_A c+1}
    &
    \lar["f^{(c)}_r"]
    k_A^{\otimes_A c}
    \end{tikzcd}
\end{equation*}
where:

\begin{itemize}
    \item $\Angles{t^{(\star)}_{r+2}-u^{(\star)}_{r+1}}\subseteq K_A^{\otimes_A c+1}$ denotes the simplicial ideal
    $$
    \Delta^{\op} 
    \to
    \Mod_A
    $$
    $$
    q
    \mapsto
    \Angles{t^{(j)}_{r+2}-u^{(j)}_{r+1}: j\in [q+1]};
    $$
    \item the vertical maps are the obvious quotients (see Remark \ref{rmk: can quot K_A^(otimes_k_A n) --> k_A^(otimes_A n)});
    \item the upper-left horizontal map is the obvious quotient;
    \item the upper-right horizontal map is defined by the same formulas as $\wt{f}^{(c)}_r$.
\end{itemize}
Notice that the right squares are cocartesian. Said geometrically, $\wt{f}^{(c)}_r$ and $f^{(c)}_r$ provide strict models for the lower right morphisms of derived schemes in the commutative diagram

\begin{equation*}
    \begin{tikzcd}[column sep = 1.5cm]
    G^{\times_S c+1}
    &
    \lar
    G^{\times_S r}\times_S G\times_s G \times_S G^{\times_S c-r-1}
    \rar["\id \times \pr{13} \times \id"]
    &
    G^{\times_S r}\times_S G \times_S G^{\times_S c-r-1}
    \simeq
    G^{\times_S c}
    \\
    \uar
    G^{\times_s c+1}
    &
    \uar
    \lar["\id"]
    G^{\times_s c+1}
    \rar["\wt{f}^{(c)}_r"]
    &
    \uar
    G^{\times_s c}
    \\
    \uar
    s^{\times_S c+1}
    &
    \uar
    \lar["\id"]
    s^{\times_S c+1}
    \rar["f^{(c)}_r"]
    &
    \uar
    s^{\times_S c},
    \end{tikzcd}
\end{equation*}
where the right squares are derived cartesian.

\sssec{Explanation of the morphisms \texorpdfstring{$f^{(c)}_c$}{}}

Since $\varrho_c(c)=\varrho_c(0)$,
we have
$$
\wt{f}^{(c)}_0
=
\wt{f}^{(c)}_c.
$$
For $r=c$ we want to consider the diagram of simplicial $A$-algebras
\begin{equation*}
    \begin{tikzcd}[column sep = 2cm]
    K_A^{\otimes_A c+1}
    \dar
    \rar
    &
    \frac{K_A^{\otimes_A c+1}}{\Angles{t^{(\star)}_{1}-u^{(\star)}_{c+1}}}
    \dar
    &
    \lar[swap,"\vartheta^{(c)}"]
    K_A^{\otimes_A c}
    \dar
    \\
    K_A^{\otimes_{k_A} c+1}
    \rar["\bigt{\varphi^{(c)}}^{-1}"]
    \dar[swap,"\id"]
    &
    \frac{K_A^{\otimes_{A} c+1}}{\Angles{t^{(\star)}_{1}-u^{(\star)}_{c+1}, u^{(\star)}_a-t^{(\star)}_a:a=2,\dots,c}}
    \dar["\varphi^{(c)}"]
    &
    \lar[swap,"\psi^{(c)}"]
    K_A^{\otimes_{k_A} c}
    \dar[swap,"\id"]
    \\
    K_A^{\otimes_{k_A} c+1}
    \rar["\id"]
    \dar
    &
    K_A^{\otimes_{k_A} c+1}
    \dar
    &
    \lar[swap,"\wt{f}^{(c)}_c"]
    K_A^{\otimes_{k_A} c}
    \dar
    \\
    k_A^{\otimes_A c+1}
    \rar["\id"]
    &
    k_A^{\otimes_A c+1}
    &
    \lar[swap,"f^{(c)}_c"]
    k_A^{\otimes_A c}
    \end{tikzcd}
\end{equation*}
where 

\begin{itemize}
    \item $\Angles{t^{(\star)}_{1}-u^{(\star)}_{c+1}}\subseteq K_A^{\otimes_A c+1}$ denotes the simplicial ideal
    $$
    \Delta^{\op} 
    \to
    \Mod_A
    $$
    $$
    q
    \mapsto
    \Angles{t^{(j)}_{1}-u^{(j)}_{c+1}: j\in [q+1]};
    $$
    \item $\Angles{t^{(\star)}_{1}-u^{(\star)}_{c+1}, u^{(\star)}_a-t^{(\star)}_a:a=2,\dots,c}$ denotes the simplicial ideal
    $$
    \Delta^{\op} 
    \to
    \Mod_A
    $$
    $$
    q
    \mapsto
    \Angles{t^{(j)}_{1}-u^{(j)}_{c+1}, u^{(j)}_a-t^{(j)}_a:a=2,\dots,c, j\in [q+1]};
    $$
    \item the vertical maps in the upper squares are the obvious quotients;
    \item the upper-left horizontal map is the obvious quotient;
    \item the upper-right horizontal map is the morphism of simplicial $A$-algebras such that, at level $q\geq 0$
    $$
    \vartheta^{(c)}_q:
    \bigt{K_A^{\otimes_A c}}_q
    \to
    \Biggt{\frac{K_A^{\otimes_A c+1}}{\Angles{t^{(\star)}_{1}-u^{(\star)}_{c+1}}}}_q
    $$
    is defined by ($j\in [q+1], a=1,\dots c$)
    \begin{equation}\label{eqn: definition varthetaq}
        t^{(j)}_a
        \mapsto
        u^{(j)}_a,
        \;\;\;\;
        u^{(j)}_a
        \mapsto 
        t^{(j)}_{a+1};
    \end{equation}
    \item the morphism $\psi^{(c)}$ is defined by the same formulas \eqref{eqn: definition varthetaq} as $\vartheta^{(c)}$;
    \item the map $\varphi^{(c)}$ is the morphism of simplicial $A$-algebras such that, at level $q\geq 0$
    $$
    \varphi^{(c)}_q:
    \Biggt{\frac{K_A^{\otimes_{A} c+1}}{\Angles{t^{(\star)}_{1}-u^{(\star)}_{c+1}, u^{(\star)}_a-t^{(\star)}_a:a=2,\dots,c}}}_q
    \to
    \bigt{K_A^{\otimes_{k_A} c+1}}_q
    $$
    is defined by ($j\in [q+1], a=1,\dots c+1$)
    \begin{equation}\label{eqn: definition varphiq}
        t^{(j)}_a
        \mapsto
        u^{(j)}_a,
        \;\;\;\;
        u^{(j)}_a
        \mapsto 
        \begin{cases}
        t^{(j)}_{1} & \text{ if } a=1,
        \\
        t^{(j)}_{a+1} & \text{ if } 2\leq a \leq c,
        \\
        t^{(j)}_2 & \text{ if } a= c+1.
        \end{cases}
    \end{equation}
    Notice that $\varphi^{(c)}$ is an isomorphism of simplicial $A$-algebras.
\end{itemize}
We also observe that the right squares are both cocartesian.
In particular, so is their composition.

Therefore, $\wt{f}^{(c)}_c$ provides a strict model for the lower-right morphism in the following commutative diagram of derived $S$-schemes
\begin{equation*}
    \begin{tikzcd}
    G^{\times_Sc+1}
    &
    \lar["\pr{G^{\times_Sc+1}}"]
    \overbrace{\bigt{s\times_SG^{\times_Sc}\times_Ss}}^{\simeq \; G^{\times_Sc+1}}\times_{s\times_Ss}s
    \rar["\pr{G^{\times_Sc}}"]
    &
    G^{\times_Sc}
    \\
    \uar
    G^{\times_sc+1}
    &
    \lar["\id"]
    \uar
    G^{\times_sc+1}
    \rar["\wt{f}^{(c)}_c"]
    &
    \uar
    G^{\times_sc}
    \\
    \uar
    s^{\times_sc+1}
    &
    \lar["\id"]
    \uar
    s^{\times_sc+1}
    \rar["f^{(c)}_c"]
    &
    \uar
    s^{\times_sc},
    \end{tikzcd}
\end{equation*}
where the vertical arrow in the middle is given by
$$
\frac{K_A^{\otimes_A c+1}}{\Angles{t^{(\star)}_{1}-u^{(\star)}_{c+1}}}
\xto{\on{can. \; quot.}}
 \frac{K_A^{\otimes_{A} c+1}}{\Angles{t^{(\star)}_{1}-u^{(\star)}_{c+1}, u^{(\star)}_a-t^{(\star)}_a:a=2,\dots,c}}
 \xto{\varphi^{(c)}}
 K_A^{\otimes_{k_A}c+1}
$$
and the right squares are cartesian.

\sssec{Definition of the morphisms \texorpdfstring{$\fff^{[a,b]}_{\ul d}$}{}}
\label{sssec: def of f[a,b]_d}

Let $m+1\leq a \leq b \leq m+k+1$.
We will write $\fff^{[a,b]}_{\ul d}: k_A^{\otimes_Aa} \to k_A^{\otimes_Ab}$ to mean the
composition
$$
  k_A^{\otimes_Aa}
    \xto{f^{(a)}_{d_{a}}}
  k_A^{\otimes_Aa+1}
    \xto{f^{(a+1)}_{d_{a+1}}}
  \cdots
    \xto{f^{(b-1)}_{d_{b-1}}}
  k_A^{\otimes_Ab}.
$$
We use the convention that $\fff^{[a,a]}_{\ul d}=\id$.

\begin{thm}\label{thm: master homotopies}

With the same notation as above, there exists a morphism of simplicial $A$-algebras
$$
  \scrH^{(k)}_{m}:
  \Delta^k \otimes \bigt{\Delta^m \otimes k_A(n+2)}
    \to
  k_A^{\otimes_Am+k+1}
$$
such that, for $l\in [k]$, the restriction to $\Delta^{\{l\}} \otimes \bigt{\Delta^m \otimes k_A(n+2)}$ is equal to the composition
\begin{align*}
  \Delta^m \otimes k_A(n+2)
    &
    \xto{\ffi^{[m,m+l]}_{\ul d}}
  \Delta^{m+l} \otimes k_A(n+2)
    \\
    &
    \xto{(\id \otimes \on{can}(n+2,m+l+1))}
  \Delta^{m+l} \otimes k_A(m+l+1)
    \\
    &
    \xto{\ffh^{(m+l)}}
  k_A^{\otimes_Am+l+1}
    \\
    &
    \xto{\fff^{[m+l+1,m+k+1]}_{\ul d}}
  k_A^{\otimes_Am+k+1}.
\end{align*}

Furthermore, the following diagram commutes:
\begin{equation}\label{eqn: H^(k-1)_m vs H^(k)_m}
    \begin{tikzcd}[column sep = 2cm]
      \Delta^{k-1} \otimes \bigt{\Delta^m \otimes k_A(n+2)}
      \rar["\id \otimes \on{can}(n+2\text{,}n+1)"]
      \dar["\iota^{k-1}_k\otimes \id"]
      &
      \Delta^{k-1} \otimes \bigt{\Delta^m \otimes k_A(n+1)}
      \rar["\scrH^{(k-1)}_m"]
      &
      k_A^{\otimes_Am+k}
      \dar["f^{(m+k)}_{d_{m+k}}"]
      \\
      \Delta^k \otimes \bigt{\Delta^m \otimes k_A(n+2)}
      \arrow[rr,"\scrH^{(k)}_m"]
      &
      &
      k_A^{\otimes_Am+k+1}.
    \end{tikzcd}
\end{equation}

\end{thm}

\begin{proof}

We proceed in steps.

\sssec*{Definition of the morphism $(\scrH^{(k)}_m)_p$}

To define $\scrH^{(k)}_m$, we need to define a morphism of $A$-algebras
$$
  (\scrH^{(k)}_m)_p:
  \Bigt{\Delta^k \otimes \bigt{\Delta^m \otimes k_A(n+2)}}_p
    \to
  \bigt{k_A^{\otimes_Am+k+1}}_p
$$
for each $p\geq 0$.
Later, we will check that these morphisms provide a morphism of simplicial $A$-algebras as wanted.
By definition,
$$
  \Bigt{\Delta^k \otimes \bigt{\Delta^m \otimes k_A(n+2)}}_p
    =
  \bigotimes_{\Delta^k_p}
  \Bigt{
        \bigotimes_{\Delta^m_p}\bigt{k_A(n+2)}_p
        },
$$
so that
$$
\Bigt{\Delta^k \otimes \bigt{\Delta^m \otimes k_A(n+2)}}_p
=
\frac{
      A\bigq{t^{(j)}(\alpha,\beta):j\in[p+1], \alpha \in \Delta^m_p, \beta \in \Delta^k_p}
     }
     { 
     \Angles{t^{(0)}(\alpha,\beta)-\pi,t^{(p+1)}(\alpha,\beta):\alpha \in \Delta^m_p, \beta \in \Delta^k_p}
     }
     .
$$
We define $(\scrH^{(k)}_m)_p$ on the generators by setting
$$
  (\scrH^{(k)}_m)_p\bigt{t^{(j)}(\alpha,\beta)} 
    :=
  \sum_{r=0}^k 
  \biggq{
         \fff^{[m+r+1,m+k+1]}_{\ul d,p}
           \circ 
         \ffh^{(m+r)}_p\Bigt{t^{(j)}\bigt{\ffi^{[m,m+r]}_{\ul d,p}(\alpha)}} 
           \cdot
         \fft(n,m,k,r,\beta)
         },
$$
where
$$
  \fft(n,m,k,r,\beta)
    :=
  t^{(\ccS_{\beta}(r-1))}_{1}
  \Bigt{
        \prod_{i=r}^{k-1}(\pi-t^{(\ccS_{\beta}(i))}_{1})
        }
  \pi^{n-m-k}.
$$

\sssec*{Compatibility with the face maps}

Let $l\in [p]$. 
We need to check that the following square commutes
\begin{equation}\label{eqn: scrH^(k)_m vs face maps}
    \begin{tikzcd}[column sep = 2cm]
      \Bigt{\Delta^k \otimes \bigt{\Delta^m \otimes k_A(n+2)}}_p
      \rar["(\scrH^{(k)}_m)_p"]
      \dar["d^p_l"]
      &
      \bigt{k_A^{\otimes_Am+k+1}}_p
      \dar["d^p_l"]
      \\
      \Bigt{\Delta^k \otimes \bigt{\Delta^m \otimes k_A(n+2)}}_{p-1}
      \rar["(\scrH^{(k)}_m)_{p-1}"]
      &
      \bigt{k_A^{\otimes_Am+k+1}}_{p-1}.
    \end{tikzcd}
\end{equation}
Let $1\leq j \leq p$, $\alpha \in \Delta^m_p$ and $\beta \in \Delta^k_p$.
Since by definition we have
$$
  d^p_l \bigt{t^{(j)}(\alpha,\beta)}
    =
  d^p_l(t^{(j)})(\alpha \circ \delta^p_{l},\beta \circ \delta^p_{l})
$$
we see that
\begin{align*}
    (\scrH^{(k)}_m)_{p-1}\bigt{t^{(j)}(\alpha \circ \delta^p_{l},\beta \circ \delta^p_{l})}
    &
    =
    \sum_{r=0}^k 
    \biggq{
    \fff^{[m+r+1,m+k+1]}_{\ul d,p-1}
    \circ 
    \ffh^{(m+r)}_{p-1}\Bigt{d^p_l(t^{(j)})\bigt{\ffi^{[m,m+r]}_{\ul d,p-1}(\alpha \circ \delta^p_{l})}} 
    \\
    &
    \;\;\;\;\;\;\;\;\;\;\;\;
    \fft(n,m,k,r,\beta \circ \delta^p_{l})
    },
\end{align*}
where
\begin{align*}
    \fft(n,m,k,r,\beta \circ \delta^p_{l})
    &
    =
    t^{(\ccS_{\beta \circ \delta^p_{l}}(r-1))}_{1}
    \Bigt{\prod_{i=r}^{k-1}(\pi-t^{(\ccS_{\beta \circ \delta^p_{l}}(i))}_{1})}
    \pi^{n-m-k}
    \\
    &
    =
    d^p_l \bigt{
                t^{(\ccS_{\beta}(r-1))}_{1}}
    \biggt{\prod_{i=r}^{k-1}\Bigt{\pi-d^p_l\bigt{t^{(\ccS_{\beta}(i))}_{1}}}}
    \pi^{n-m-k}.
\end{align*}
The last equality follows from Corollary \ref{cor:lemmino}.
On the other hand, 
\begin{align*}
    d^p_l \bigt{
               (\scrH^{(k)}_m)_p\bigt{t^{(j)}(\alpha,\beta)}
               }
    &
    =
    d^p_l \Biggt{
      \sum_{r=0}^k 
      \biggq{
      \fff^{[m+r+1,m+k+1]}_{\ul d,p}
      \circ 
      \ffh^{(m+r)}_p\Bigt{t^{(j)}\bigt{\ffi^{[m,m+r]}_{\ul d,p}(\alpha)}} 
      \cdot
      \fft(n,m,k,r,\beta)}
      }
    \\
    &
    =
    \sum_{r=0}^k 
      \biggq{
      d^p_l \biggt{
      \fff^{[m+r+1,m+k+1]}_{\ul d,p}
      \circ 
      \ffh^{(m+r)}_p\Bigt{t^{(j)}\bigt{\ffi^{[m,m+r]}_{\ul d,p}(\alpha)}}
      }
      \cdot
      d^p_l \bigt{
      \fft(n,m,k,r,\beta)
                   }
                   }.
\end{align*}
Since $\fff^{[m+r+1,m+k+1]}_{\ul d}$ and $\ffh^{(m+r)}$ are morphisms of simplicial $A$-algebras, we have that
\begin{align*}
d^p_l \biggt{
      \fff^{[m+r+1,m+k+1]}_{\ul d,p}
      \circ 
      \ffh^{(m+r)}_p\Bigt{t^{(j)}\bigt{\ffi^{[m,m+r]}_{\ul d,p}(\alpha)}}
      }
&
=
\fff^{[m+r+1,m+k+1]}_{\ul d,p-1}
      \circ 
      \ffh^{(m+r)}_{p-1}\Bigt{d^p_l\bigt{t^{(j)}\bigt{\ffi^{[m,m+r]}_{\ul d,p}(\alpha)}}}
\\
&
=
\fff^{[m+r+1,m+k+1]}_{\ul d,p-1}
      \circ 
      \ffh^{(m+r)}_{p-1}\Bigt{d^p_l(t^{(j)})\bigt{\ffi^{[m,m+r]}_{\ul d,p-1}(\alpha \circ \delta^p_l)}}.
\end{align*}
We also notice that
$$
d^p_l \bigt{
      \fft(n,m,k,r,\beta)
                   }
=
\fft(n,m,k,r,\beta \circ \delta^p_{l}).
$$
This shows that
$$
(\scrH^{(k)}_m)_{p-1}\circ d^p_l \bigt{t^{(j)}(\alpha,\beta)}
=
d^p_l \circ (\scrH^{(k)}_m)_{p}\bigt{t^{(j)}(\alpha,\beta)}.
$$

\sssec*{Compatibility with the degeneracy maps}

For $l \in [p]$, we need to check that the following square commutes:
\begin{equation}\label{eqn: scrH^(k)_m vs deegeneracy maps}
    \begin{tikzcd}[column sep = 2cm]
      \Bigt{\Delta^k \otimes \bigt{\Delta^m \otimes k_A(n+2)}}_p
      \rar["(\scrH^{(k)}_m)_p"]
      \dar["s^p_l"]
      &
      \bigt{k_A^{\otimes_Am+k+1}}_p
      \dar["s^p_l"]
      \\
      \Bigt{\Delta^k \otimes \bigt{\Delta^m \otimes k_A(n+2)}}_{p+1}
      \rar["(\scrH^{(k)}_m)_{p+1}"]
      &
      \bigt{k_A^{\otimes_Am+k+1}}_{p+1}.
    \end{tikzcd}
\end{equation}
This verification is similar to the proof of the compatibility with the face maps and is left to the reader.

\sssec*{Restriction of $\scrH^{(k)}_m$ to the vertices of $\Delta^k$}

So far, we have produced a morphism of simplicial $A$-algebras
$$
  \scrH^{(k)}_{m}:
  \Delta^k \otimes \bigt{\Delta^m \otimes k_A(n+2)}
    \to
  k_A^{\otimes_Am+k+1}.
$$
We now compute its restriction along the vertices
of $\Delta^{k}$, i.e. we compute the compositions
$$
  \Delta^{\{l\}} \otimes \bigt{\Delta^m \otimes k_A(n+2)}
    \hto
  \Delta^k \otimes \bigt{\Delta^m \otimes k_A(n+2)}
    \to
  k_A^{\otimes_Am+k+1}
$$
for $l\in [k]$. 
In other words, we compute
$$
  (\scrH^{(k)}_m)_p\bigt{t^{(j)}(\alpha,\beta_l)}
$$

$$
  (\scrH^{(k)}_m)_p\bigt{u^{(j)}(\alpha,\beta_l)}
$$
for $\beta_l:[p]\to [k]$ the constant map with value $l$. This boils down to the computation of $\fft(n,m,k,r,\beta_l)$ and $\ffu(n,m,k,r,\beta_l)$ for $r=0,\dots,k$.
Recall that 
$$
  \ccS_{\beta_l}(i)=
    \begin{cases}
      0 & \text{ if } i\leq l-1;
        \\
      p+1 & \text{ if } i \geq l.
    \end{cases}
$$
We consider all the possible cases:
\begin{itemize}
    \item If $r \leq l-1$, then the factor 
    $$
    (\pi - t^{(\ccS_{\beta_l}(r))}_{1})
      =
    (\pi - t^{(0)}_{1})
    $$
    appears in the product which defines $\fft(n,m,k,r,\beta_l)$, so that
    $$
    \fft(n,m,k,r,\beta_l)
      =
    0.
    $$
    \item If $r = l$, then we see that 
    \begin{align*}
        \fft(n,m,k,r,\beta_l)
        &
        =
        t^{(\ccS_{\beta_l}(l-1))}_{1}\Bigt{\prod_{i=l}^{k-1} \bigt{\pi-t^{(\ccS_{\beta_l}(i))}_{1}}} \pi^{n-m-k}
        \\
        &
        =
        \pi \cdot \pi^{k-l} \cdot \pi^{n-m-k} = \pi^{n-m-l+1}.
    \end{align*}
    \item If $r \geq l+1$, then we see that 
    $$
    \fft(n,m,k,r,\beta_l)
    =
    0
    $$
    since $t^{(\ccS_{\beta_l}(r-1))}_{1}=t^{(p+1)}_{1}=0$.
\end{itemize}

We thus obtain that
$$
  (\scrH^{(k)}_m)_p\bigt{t^{(j)}(\alpha,\beta_l)}
    =
  \fff^{[m+l+1,m+k+1]}_{\ul d,p}
    \circ 
  \ffh^{(m+l)}_p\Bigt{t^{(j)}\bigt{\ffi^{[m,m+l]}_{\ul d,p}(\alpha)}} \cdot \pi^{n-m-l+1},
$$
and consequently
$$
  \restr{\scrH^{(k)}_m}{\Delta^{\{l\}}\otimes \bigt{\Delta^m \otimes K_A(n+2)}}
    =
  \fff^{[m+l+1,m+k+1]}_{\ul d}
    \circ 
  \ffh^{(m+l)}
    \circ
  (\id \otimes \on{can}(n+2,m+l+1))
    \circ
  \ffi^{[m,m+l]}_{\ul d}.
$$

\sssec*{Commutativity of diagram \eqref{eqn: H^(k-1)_m vs H^(k)_m}}

We need to check that, for all $p\geq 0$, the identity
$$
  \Bigt{
        f^{(m+k)}_{d_{m+k}}\circ \scrH^{(k-1)}_m \circ  \id \otimes \on{can}(n+2,n+1)
       }_p
    =
  \Bigt{
        \scrH^{(k)}_m\circ \iota^{k-1}_k\otimes \id
        }_p
$$
holds.
We do so by inspecting the action of both terms on the generators 
$$
  t^{(j)}(\alpha, \beta),u^{(j)}(\alpha, \beta)
    \in 
  \Bigt{
        \Delta^k \otimes \bigt{
                               \Delta^m\otimes K_A(n+2)
                               }
        }_p,
$$
where $j\in \angles{p}$, $\alpha \in \Delta^m_p$ and $\beta \in \Delta^{k-1}_p$.
Unraveling the definitions, we see that 
$$
  \Bigt{
        f^{(m+k)}_{d_{m+k}}
          \circ 
        \scrH^{(k-1)}_m 
          \circ  
        \id 
        \otimes 
        \on{can}(n+2,n+1)
        }_p
  \bigt{
        t^{(j)}(\alpha,\beta)
        }
$$
is equal to 
$$
  f^{(m+k)}_{d_{m+k},p}
    \circ
  \sum_{r=0}^{k-1} 
  \biggt{
         \fff^{[m+r+1,m+k]}_{\ul d,p}
            \circ 
         \ffh^{(m+r)}_p\Bigt{\pi t^{(j)}\bigt{\ffi^{[m,m+r]}_{\ul d,p}(\alpha)}} 
            \cdot
         t^{(\ccS_{\beta}(r-1))}_{1}
        \Bigt{\prod_{i=r}^{k-2}(\pi-t^{(\ccS_{\beta}(i))}_{1})}
        \pi^{n-m-k}
}.
$$
Since
$$
  f^{(m+k)}_{d_{m+k},p}
    \circ 
  \fff^{[m+r+1,m+k]}_{\ul d,p}
    =
  \fff^{[m+r+1,m+k+1]}_{\ul d,p}
$$
this is in turn equal to
$$
  \sum_{r=0}^{k-1} 
  \biggt{
         \fff^{[m+r+1,m+k+1]}_{\ul d,p}
           \circ 
         \ffh^{(m+r)}_p\Bigt{t^{(j)}\bigt{\ffi^{[m,m+r]}_{\ul d,p}(\alpha)}} 
            \cdot
         t^{(\ccS_{\beta}(r-1))}_{1}
         \Bigt{\prod_{i=r}^{k-2}(\pi-t^{(\ccS_{\beta}(i))}_{1})}
         \pi^{n-m-k+1}
}.
$$
Notice that here we have used the fact that 
$$
  f^{(m+k)}_{d_{m+k},p}(t^{(i)}_1)
    =
  t^{(i)}_1,
$$ 
for all $i=1,\dots,p$, regardless of the value $d_{m+k}\in [m+k]$.

On the other hand, since
$$
  (\iota^{k-1}_k\otimes \id)_p \bigt{t^{(j)}(\alpha,\beta)}
    =
  t^{(j)}(\alpha,\delta^k_k\circ \beta),
$$
we see that 
$$
  \bigt{\scrH^{(k)}_m\circ \iota^{k-1}\otimes \id}_p\bigt{t^{(j)}(\alpha, \beta)}
$$
is equal to
$$
  \sum_{r=0}^k 
  \biggt{
         \fff^{[m+r+1,m+k+1]}_{\ul d,p}
            \circ 
            \ffh^{(m+r)}_p\Bigt{t^{(j)}\bigt{\ffi^{[m,m+r]}_{\ul d,p}(\alpha)}} 
          \cdot
            t^{(\ccS_{\delta^k_k\circ \beta}(r-1))}_{1}
             \Bigt{\prod_{i=r}^{k-1}(\pi-t^{(\ccS_{\delta^k_k\circ \beta}(i))}_{1})}
            \pi^{n-m-k}
        }.
$$
Since
$$
  \ccS_{\delta^k_k\circ \beta}(i)
    =
  \begin{cases}
    \ccS_{\beta}(i) & \text{ if } i\leq k-2;
      \\
    p+1 & \text{ if } i=k-1, k
  \end{cases}
$$
this in turn equals
$$
  \sum_{r=0}^{k-1} 
  \biggt{
        \fff^{[m+r+1,m+k+1]}_{\ul d,p}
          \circ 
        \ffh^{(m+r)}_p\Bigt{t^{(j)}\bigt{\ffi^{[m,m+r]}_{\ul d,p}(\alpha)}} 
          \cdot
        t^{(\ccS_{\delta^k_k\circ \beta}(r-1))}_{1}
        \Bigt{\prod_{i=r}^{k-2}(\pi-t^{(\ccS_{\delta^k_k\circ \beta}(i))}_{1})}
        \pi^{n-m-k+1}
        }.
$$
This shows that
$$
  \Bigt{
        f^{(m+k)}_{d_{m+k}}\circ \scrH^{(k-1)}_m \circ  \id \otimes \on{can}(n+2,n+1)
        }_p \bigt{t^{(j)}(\alpha, \beta)}
    =
  \Bigt{
        \scrH^{(k)}_m\circ \iota^{k-1}_k\otimes \id
        }_p \bigt{t^{(j)}(\alpha, \beta)}.
$$

\end{proof}

\subsection{The master homotopies - variant}\label{ssec: master homotopies - variant}

The task of this section is to provide the coherent homotopies which are necessary to construct a dg-functor out of the $n^{th}$-truncated cyclic bar complex of $\sT^+$.

\sssec{}

Let $m \geq 0$ and $k \geq 1$.
We fix a vector
$$
\ul d =(d_{m+1},\dots,d_{m+k}) \in \prod_{s=m+1}^{m+k}[s].
$$

\sssec{Definition of the morphisms \texorpdfstring{$g^{(c)}_d$}{}}

For each $e\geq 0$
let us use the following presentation of the simplicial algebra 
$
K_A\otimes_Ak_A^{\otimes_{A}e}
$
: for $q\geq 0$
$$
  (K_A\otimes_Ak_A^{\otimes_{A}e})_q
    =
  \frac{
        (K_A^{\otimes_Ae+1})_q
        }
        {
         \Angles{t^{(j)}_a-u^{(j)}_a:j\in [q+1],a=2,\dots,e+1}
        }.
$$

Similarly, we use the following presentation for the simplicial algebra $k_A^{\otimes_A e+1}$: for $q\geq 0$
$$
  (k_A^{\otimes_{A}e+1})_q
    =
  \frac{
        (K_A^{\otimes_Ae+1})_q
        }
        {
         \Angles{t^{(j)}_a-u^{(j)}_a:j\in [q+1],a=1,\dots,e+1}
        }.
$$
Let $c \ge 1$ and $r\in [c]$.
We consider the morphisms of simplicial $A$-algebras
$$
  \wt{g}^{(c)}_{r}:
  K_A\otimes_Ak_A^{\otimes_{A}c-1}
    \to
  K_A\otimes_Ak_A^{\otimes_{A}c}
$$
defined as follows.

Recall that $\varrho_c$ denotes the morphism
$$
  [c]
    \to 
  \angles{c}
    =
  \{1,\dots,c\},
    \;\;\;\;
  i
    \mapsto
  \begin{cases}
    i+1 & \text{ if } i\leq c-1,
      \\
    1 & \text{ if } i=c. 
  \end{cases}
$$
Then for each $q\geq 0$ we set $(1\leq j \leq c, 1\leq a \leq c)$
$$
  \bigt{\wt{g}^{(c)}_{r}}_q:
  (K_A\otimes_AK_A^{\otimes_{A}c-1})_q
    \to
  (K_A\otimes_AK_A^{\otimes_{A}c})_q
$$
$$
t^{(j)}_a 
\mapsto
\begin{cases}
t^{(j)}_a & \text{ if } a\leq \varrho_c(r),
\\
t^{(j)}_{a+1} & \text{ if } a\geq \varrho_c(r)+1,
\end{cases}
\;\;\;\;
u^{(j)}_a 
\mapsto
\begin{cases}
u^{(j)}_a & \text{ if } a\leq \varrho_c(r),
\\
u^{(j)}_{a+1} & \text{ if } a\geq \varrho_c(r)+1.
\end{cases}
$$
Since these morphisms of $A$-algebras act on the index $a$ of the generators of $t^{(j)}_a,u^{(j)}_a\in K_A\otimes_A k_A^{\otimes_Ac-1}$, while the face and degeneracy maps  act on the index $j$, the $\bigl \{ \bigt{\wt{g}^{(c)}_{r}}_q \bigr \}_{q\geq 0}$'s define a morphisms of simplicial $A$-algebras.

Moreover, we denote by
$$
  g^{(c)}_{r}:
  k_A^{\otimes_{A}c}
    \to
  k_A^{\otimes_{A}c+1}
$$
the induced morphism.

\sssec{Explanation of the morphisms \texorpdfstring{$g^{(c)}_r$}{} (\texorpdfstring{$0\leq r \leq c-1$}{})}
\label{rmk: morphims g^(c)_r}
Let $0\leq r \leq c-1$.
The morphism 
$$
  \wt{g}^{(c)}_r:
  K_A\otimes_Ak_A^{\otimes_{A}c-1}
    \to
  K_A\otimes_Ak_A^{\otimes_{A}c}
$$
\virg{skips} the generators $t^{(j)}_{r+2}=u^{(j)}_{r+2}$ and
thus sits in the commutative diagram of simplicial $A$-algebras
\begin{equation*}
    \begin{tikzcd}
    K_A\otimes_AK_A^{\otimes_{A}c}
    \dar
    \rar
    &
    \frac{K_A\otimes_AK_A^{\otimes_{A}c}}{\Angles{t^{(\star)}_{r+2}-u^{(\star)}_{r+2}}}
    \dar
    &
    \lar
    K_A\otimes_AK_A^{\otimes_{A}c-1}
    \dar
    \\
    K_A\otimes_Ak_A^{\otimes_{A}c}
    \rar[swap,"\id"]
    \dar
    &
    K_A\otimes_Ak_A^{\otimes_{A}c}
    \dar
    &
    \lar["\wt{g}^{(c)}_r"]
    K_A\otimes_Ak_A^{\otimes_{A}c-1}
    \dar
    \\
    k_A^{\otimes_{A}c+1}
    \rar[swap,"\id"]
    &
    k_A^{\otimes_{A}c+1}
    &
    \lar["g^{(c)}_r"]
    k_A^{\otimes_{A}c}.
    \end{tikzcd}
\end{equation*}
Here:
\begin{itemize}
    \item $\Angles{t^{(\star)}_{r+2}-u^{(\star)}_{r+2}}\subseteq K_A^{\otimes_A c+1}$ denotes the simplicial ideal
    $$
      \Delta^{\op} 
        \to
      \Mod_A
    $$
    $$
      q
        \mapsto
      \Angles{t^{(j)}_{r+2}-u^{(j)}_{r+2}: j\in [q+1]};
    $$
    \item the vertical maps are the canonical quotients;
    \item the upper-left horizontal map is the canonical quotient;
    \item the upper-right horizontal map is defined by the same formulas as $\wt{g}^{(c)}_r$ and $g^{(c)}_r$.
\end{itemize}
Also, note that the right squares in the above diagram are cocartesian.

Geometrically, $\wt{g}^{(c)}_r$ and $g^{(c)}_r$ sits in the following commutative diagram of derived schemes
\begin{equation*}
    \begin{tikzcd}
    G\times_SG^{\times_S c}
    &
    \lar
    G\times_SG^{\times_S r}\times_S s\times_S G^{\times_S c-r-1}
    \rar["\on{proj.}"]
    &
    G\times_SG^{\times_S r}\times_S G^{\times_S c-r-1}
    \simeq
    G\times_SG^{\times_S c-1}
    \\
    \uar
    G\times_Ss^{\times_S c}
    &
    \uar
    \lar["\id"]
    G\times_Ss^{\times_S c}
    \rar["\wt{g}^{(c)}_r"]
    &
    \uar
    G\times_Ss^{\times_S c-1}
    \\
    \uar
    s^{\times_S c+1}
    &
    \uar
    \lar["\id"]
    s^{\times_S c+1}
    \rar["g^{(c)}_r"]
    &
    \uar
    s^{\times_S c},
    \end{tikzcd}
\end{equation*}
where the right squares are derived cartesian.
The vertical maps are induced by the diagonal $\delta_{s/S}:s\to G$.

\sssec{Explanation of the morphisms \texorpdfstring{$g^{(c)}_c$}{}}
Since $\varrho(c)=\varrho(0)$, we have
$$
  \wt{g}^{(c)}_0
    =
  \wt{g}^{(c)}_c,
    \;\;\;\;
  g^{(c)}_0
    =
  g^{(c)}_c.
$$
As in the case of the morphisms $\wt{f}^{(c)}_r$ and $f^{(c)}_r$, this is something we want.
For $r=c$ we wish to consider the diagram of simplicial $A$-algebras
\begin{equation*}
    \begin{tikzcd}[column sep = 2cm]
    K_A^{\otimes_A c+1}
    \dar
    \rar
    &
    \frac{K_A^{\otimes_A c+1}}{\Angles{t^{(\star)}_{c+1}-u^{(\star)}_{c+1}}}
    \dar
    &
    \lar[swap,"\omega^{(c)}"]
    K_A^{\otimes_A c}
    \dar
    \\
    k_A^{\otimes_{A} c}\otimes_AK_A
    \rar["\tau^{(c)}"]
    \dar[swap,"\rho_{c+1}"]
    &
    \frac{K_A^{\otimes_{A} c+1}}{\Angles{t^{(\star)}_{c+1}-u^{(\star)}_{c+1}, u^{(\star)}_a-t^{(\star)}_a:a=1,\dots,c-1}}
    \dar["\upsilon^{(c)}"]
    &
    \lar[swap,"\gamma^{(c)}"]
    k_A^{\otimes_{A} c-1}\otimes_AK_A
    \dar[swap,"\rho_c"]
    \\
    K_A\otimes_Ak_A^{\otimes_{A} c}
    \rar["\id"]
    \dar
    &
    K_A\otimes_Ak_A^{\otimes_{A} c}
    \dar
    &
    \lar[swap,"\wt{g}^{(c)}_c"]
    K_A\otimes_Ak_A^{\otimes_{A} c-1}
    \dar
    \\
    k_A^{\otimes_{A} c+1}
    \rar["\id"]
    &
    k_A^{\otimes_{A} c+1}
    &
    \lar[swap,"g^{(c)}_c"]
    k_A^{\otimes_{A} c}
    \end{tikzcd}
\end{equation*}
where 
\begin{itemize}
    \item $\Angles{t^{(\star)}_{c+1}-u^{(\star)}_{c+1}}\subseteq K_A^{\otimes_A c+1}$ denotes the simplicial ideal
    $$
    \Delta^{\op} 
    \to
    \Mod_A
    $$
    $$
    q
    \mapsto
    \Angles{t^{(j)}_{c+1}-u^{(j)}_{c+1}: j\in [q+1]};
    $$
    \item $\Angles{t^{(\star)}_{c+1}-u^{(\star)}_{c+1}, u^{(\star)}_a-t^{(\star)}_a:a=1,\dots,c-1}$ denotes the simplicial ideal
    $$
    \Delta^{\op} 
    \to
    \Mod_A
    $$
    $$
    q
    \mapsto
    \Angles{t^{(j)}_{c+1}-u^{(j)}_{c+1}, u^{(j)}_a-t^{(j)}_a:a=1,\dots,c-1, j\in [q+1]};
    $$
    \item the vertical maps in the upper and lower squares are the obvious quotients;
    \item the upper-left horizontal map is the obvious quotient;
    \item the upper-right horizontal map is the morphism of simplicial $A$-algebras such that, at level $q\geq 0$
    $$
    \omega^{(c)}_q:
    \bigt{K_A^{\otimes_A c}}_q
    \to
    \Biggt{\frac{K_A^{\otimes_A c+1}}{\Angles{t^{(\star)}_{c+1}-u^{(\star)}_{c+1}}}}_q
    $$
    is defined by ($j\in [q+1], a=1,\dots c$)
    \begin{equation}\label{eqn: definition omegaq}
        t^{(j)}_a
        \mapsto
        t^{(j)}_a,
        \;\;\;\;
        u^{(j)}_a
        \mapsto 
        u^{(j)}_{a};
    \end{equation}
    \item the morphism $\gamma^{(c)}$ is defined by the same formulas \eqref{eqn: definition omegaq} as $\omega^{(c)}$;
    \item the map $\upsilon^{(c)}$ is the morphism of simplicial $A$-algebras such that, at level $q\geq 0$
    $$
    \upsilon^{(c)}_q:
    \Biggt{\frac{K_A^{\otimes_{A} c+1}}{\Angles{t^{(\star)}_{c+1}-u^{(\star)}_{c+1}, u^{(\star)}_a-t^{(\star)}_a:a=1,\dots,c-1}}}_q
    \to
    \bigt{K_A\otimes_Ak_A^{\otimes_{A} c}}_q
    $$
    is defined by ($j\in [q+1], a=1,\dots c+1$)
    \begin{equation}\label{eqn: definition upsilonq}
        t^{(j)}_a
        \mapsto
        t^{(j)}_{\rho^2_{c+1}(a)},
        \;\;\;\;
        u^{(j)}_a
        \mapsto 
        u^{(j)}_{\rho^2_{c+1}(a)}.
    \end{equation}
    Notice that $\upsilon^{(c)}$ is an isomorphism of simplicial $A$-algebras.
    \item the morphism 
    $$
    \rho_{c}:
    k_A^{\otimes_Ac-1}\otimes_AK_A
    \to
    K_A\otimes_Ak_A^{\otimes_Ac-1}
    \;\;\;\;
    \Bigt{\text{ resp. }
    \rho_{c+1}:
    k_A^{\otimes_Ac}\otimes_AK_A
    \to
    K_A\otimes_Ak_A^{\otimes_Ac}
    }
    $$
    is induced by the permutation
    $$
    t^{(j)}_a
    \mapsto
    t^{(j)}_{\rho_c(a)}
    \;\;\;\;
    \Bigt{\text{ resp. }
    t^{(j)}_a
    \mapsto
    t^{(j)}_{\rho_{c+1}(a)}
    },
    $$
    $$
    u^{(j)}_a
    \mapsto
    u^{(j)}_{\rho_c(a)}
    \;\;\;\;
    \Bigt{\text{ resp. }
    u^{(j)}_a
    \mapsto
    u^{(j)}_{\rho_{c+1}(a)}
    }.
    $$
    \item the morphism 
    $$
    \tau^{(c)}:
    k_A^{\otimes_Ac}\otimes_AK_A
    \to
    \frac{K_A^{\otimes_{A} c+1}}{\Angles{t^{(\star)}_{c+1}-u^{(\star)}_{c+1}, u^{(\star)}_a-t^{(\star)}_a:a=1,\dots,c-1}}
    $$
    is defined by
    $$
    \tau^{(c)}\bigt{t^{(j)}_a}
    =
    t^{(j)}_{\rho^{-1}_{c+1}(a)},
    $$
    $$
    \tau^{(c)}\bigt{u^{(j)}_a}
    =
    u^{(j)}_{\rho^{-1}_{c+1}(a)}.
    $$
\end{itemize}
Note that the right squares in the diagram above are cocartesian.

The geometric meaning of the discussion above is that $\wt{g}^{(c)}_c$ and $g^{(c)}_c$ sit in the following commutative diagram of derived $S$-schemes
\begin{equation*}
    \begin{tikzcd}
    G^{\times_Sc}\times_SG
    &
    \lar[]
    G^{\times_Sc-1}\times_Ss\times_SG
    \rar["\on{proj.}"]
    &
    G^{\times_Sc-1}\times_SG
    \\
    \uar
    s^{\times_Sc}\times_SG
    &
    \lar["\id"]
    \uar
    s^{\times_Sc}\times_SG
    \rar["\wt{g}^{(c)}_c"]
    &
    \uar
    s^{\times_Sc-1}\times_SG
    \\
    \uar
    s^{\times_Sc+1}
    &
    \lar["\id"]
    \uar
    s^{\times_sc+1}
    \rar["g^{(c)}_c"]
    &
    \uar
    s^{\times_Sc},
    \end{tikzcd}
\end{equation*}
where the vertical arrow in the middle of the upper half of the diagram is given by
$$
  \frac{
        K_A^{\otimes_A c+1}
        }
        {
        \Angles{t^{(\star)}_{c+1}-u^{(\star)}_{c+1}}
        }
    \xto{\on{can. \; quot.}}
  \frac{
        K_A^{\otimes_{A} c+1}
        }
        {
        \Angles{t^{(\star)}_{c+1}-u^{(\star)}_{c+1}, u^{(\star)}_a-t^{(\star)}_a:a=1,\dots,c-1}
        }
    \xto{\upsilon^{(c)}}
  K_A\otimes_Ak_A{\otimes_{A}c}
$$
and the right squares are cartesian.

\sssec{}

We now establish the following connection between the morphisms $f^{(c)}_r$ defined in Section \ref{sssec: def of f^(c)_r} and the morphisms $g^{(c)}_r$.
Define $\zeta_{n} \in S_{2n}=\on{Aut}(\angles{2n})$ to be the cyclic permutation $(2 \mapsto 4 \dots \mapsto 2n-2 \mapsto 2n \mapsto 2)^{-1}$ of the even elements.
For $n\geq 1$, denote by the same symbol 
$$
  \zeta_n:
  K_A^{\otimes_An}
    \to
  K_A^{\otimes_An}
$$
the induced automorphism of $K_A^{\otimes_An}$. It is easy to see that $\zeta_n$  descends the quotients:
\begin{equation*}
    \begin{tikzcd}
      K_A^{\otimes_{k_A}n}
      \rar["\zeta_n"]
      \dar
       &
      K_A\otimes_Ak_A^{\otimes_An-1}
      \dar
      \\
      \frac{K_A^{\otimes_{k_A}n}}{\Angles{t^{(\star)}_1-u^{(\star)}_n}}
      \rar["\zeta_n"]
      &
      k_A^{\otimes_An}.
    \end{tikzcd}
\end{equation*}
The following lemma is a direct consequence of the definitions.

\begin{lem}\label{lem: f^(c)_r vs g^(c)_r}
With the above notation, the square
\begin{equation}\label{eqn: f^(c)_r vs g^(c)_r}
  \begin{tikzcd}
    k_A^{\otimes_Ac}
    \rar["g^{(c)}_r"]
    &
    k_A^{\otimes_Ac+1}
    \\
    \uar["\zeta_c"]
    \frac{K_A^{\otimes_{k_A}c}}{\Angles{t^{(\star)}_1-u^{(\star)}_c}}
    \rar["f^{(c)}_r"]
    &
    \uar["\zeta_{c+1}"]
    \frac{K_A^{\otimes_{k_A}c+1}}{\Angles{t^{(\star)}_1-u^{(\star)}_{c+1}}}.
  \end{tikzcd}
\end{equation}
commutes in $\sCAlg_{A/}$.
\end{lem}

\sssec{Definition of the morphisms \texorpdfstring{$\ffg^{[a,b]}_{\ul d}$}{}}
\label{sssec: def of g[a,b]_d}

For $m+1\leq a \leq b \leq m+k+1$, we will write 
$$
  \ffg^{[a,b]}_{\ul d}:
  k_A^{\otimes_Aa}
  \to 
  k_A^{\otimes_Ab}
$$
to mean the
composition
$$
  k_A^{\otimes_Aa}
    \xto{g^{(a)}_{d_{a}}}
  k_A^{\otimes_Aa+1}
    \xto{g^{(a+1)}_{d_{a+1}}}
  \cdots
    \xto{g^{(b-1)}_{d_{b-1}}}
  k_A^{\otimes_Ab}.
$$
We use the convention that $\ffg^{[a,a]}_{\ul d}=\id$.

\sssec{}

As an immediate consequence of Lemma \ref{lem: f^(c)_r vs g^(c)_r} above, we have the following
\begin{cor}\label{cor: fff^[a,b] vs ffg^[a,b]}
With the same notation as above, the following square in $\sCAlg_{A/}$ is commutative:
\begin{equation}\label{eqn: fff^(c)_r vs ffg^(c)_r}
  \begin{tikzcd}
    k_A^{\otimes_Aa}
    \rar["\ffg^{[a,b]}_{\ul d}"]
    &
    k_A^{\otimes_Ab}
    \\
    \uar["\zeta_a"]
    \frac{K_A^{\otimes_{k_A}a}}{\Angles{t^{(\star)}_1-u^{(\star)}_a}}
    \rar["\fff^{[a,b]}_{\ul d}"]
    &
    \uar["\zeta_{b}"]
    \frac{K_A^{\otimes_{k_A}b}}{\Angles{t^{(\star)}_1-u^{(\star)}_b}}.
  \end{tikzcd}
\end{equation}
\end{cor}

\sssec{}

For each $n\geq 1$
$$
  \ffk^{(n-1)}:
  \Delta^{n-1}\otimes k_A(n)
    \to
  k_A^{\otimes_An}
$$
denote the
$$
  \Delta^{n-1}\otimes k_A(n)
    \xto{\ffh^{(n-1)}}
  \frac{K_A^{\otimes_{k_A}n}}{\Angles{t^{(\star)}_1-u^{(\star)}_n}}
    \xto{\zeta_n}
  k_A^{\otimes_An},
$$

We are finally ready to state the following
\begin{thm}\label{thm: master homotopies - variant}

With the same notation as above, there exists a morphism of simplicial $A$-algebras
$$
  \scrK^{(k)}_{m}:
  \Delta^k \otimes \bigt{\Delta^m \otimes k_A(n+2)}
    \to
  k_A^{\otimes_Am+k+1}
$$
such that, for $l\in [k]$, the restriction to $\Delta^{\{l\}} \otimes \bigt{\Delta^m \otimes K_A(n+2)}$ is equal to the composition
\begin{align*}
  \Delta^m \otimes k_A(n+2)
    &
    \xto{\ffi^{[m,m+l]}_{\ul d}}
  \Delta^{m+l} \otimes k_A(n+2)
    \\
    &
    \xto{(\id \otimes \on{can}(n+2,m+l+1))}
  \Delta^{m+l} \otimes k_A(m+l+1)
    \\
    &
    \xto{\ffh^{(m+l)}}
    k_A^{\otimes_Am+l+1}
    \\
    &
    \xto{\ffg^{[m+l+1,m+k+1]}_{\ul d}}
    k_A^{\otimes_Am+k+1}.
\end{align*}
\end{thm}

Furthermore, the following diagram commutes:
\begin{equation}\label{eqn: K^(k-1)_m vs K^(k)_m}
    \begin{tikzcd}[column sep = 2cm]
      \Delta^{k-1} \otimes \bigt{\Delta^m \otimes k_A(n+2)}
      \rar["\id \otimes \on{can}(n+2\text{,}n+1)"]
      \dar["\iota^{k-1}_k\otimes \id"]
      &
      \Delta^{k-1} \otimes \bigt{\Delta^m \otimes k_A(n+1)}
      \rar["\scrK^{(k-1)}_m"]
      &
      k_A^{\otimes_Am+k+1}
      \dar["g^{(m+k)}_{d_{m+k}}"]
      \\
      \Delta^k \otimes \bigt{\Delta^m \otimes k_A(n+2)}
      \arrow[rr,"\scrK^{(k)}_m"]
      &
      &
      k_A^{\otimes_Am+k+2}.
    \end{tikzcd}
\end{equation}

\begin{proof}

We define
$$
  \scrK^{(k)}_{m}:
  \Delta^k \otimes \bigt{\Delta^m \otimes k_A(n+2)}
    \to
  k_A^{\otimes_Am+k+1}
$$
as the composition
$$
  \Delta^k \otimes \bigt{\Delta^m \otimes k_A(n+2)}
    \xto{\scrH^{(k)}_m}
  \frac{K_A^{\otimes_{k_A}m+k+1}}{\Angles{t^{(\star)}_1-u^{(\star)}_{m+k+1}}}
    \xto{\zeta_{m+k+1}}
  k_A^{\otimes_Am+k+1}.
$$
Since
$$
  \ffk^{(c)}:
  \Delta^{c-1}\otimes k_A(c)
    \to
  k_A^{\otimes_Ac}
$$
is by definition the composition
$$
  \Delta^{c}\otimes k_A(c+1)
    \xto{\ffh^{(c)}}
  \frac{K_A^{\otimes_{k_A}c+1}}{\Angles{t^{(\star)}_1-u^{(\star)}_{c+1}}}
    \xto{\zeta_c}
  k_A^{\otimes_Ac+1},
$$
the theorem follows immediately from Theorem \ref{thm: master homotopies} and Corollary \ref{cor: fff^[a,b] vs ffg^[a,b]}.
\end{proof}


\section{Integrating categorical differential forms}\label{sec: integrating differential forms on two-periodic complexes}

In this section, we construct a dg-functor
\begin{equation}\label{eqn:oint}
  \oint : \HH(\sB/A)
    \to 
  \Sing(\bbA^1_S[-1])_s
\end{equation}
that \virg{integrates categorical differential forms} on the singularity category of the loop groupoid $G$.
The construction is performed in Section \ref{ssec:oint-construction I} and Section \ref{ssec:oint-construction II}.
As an application, in Section \ref{ssec:realization of HH}, we show that $\uH^0_{\et}\bigt{S,\rl_S\bigt{\HH(\sB/A)}}$, the $\Qell$-vector space where \eqref{eqn: categorical gBCF} takes place, contains a copy of $\Qell$ as a canonical retract.

\subsection{The construction of the functor \texorpdfstring{$\oint^+$}{oint+}} \label{ssec:oint-construction I}

In this section, we define the coherent version 
$$
  \oint^+: \HH(\sB^+/A)
    \to
  \Coh(\bbA^1_S[-1])_s
$$
of the integration dg-functor.

This construction is performed in three steps: we first define an intermediate dg-category $\sB^+(\infty)$, and then we define two dg-functors
\begin{equation} \label{eqn:first-arrow}
\xi^+: \HH(\sB^+/A)
\to \sB^+(\infty)
\end{equation}
\begin{equation} \label{eqn:second-arrow}
r(\infty)_*:
\sB^+(\infty)
\to 
\Coh(\bbA^1_S[-1])_s
\end{equation}
whose composition gives the desired $\oint^+$.

\sssec{} \label{sssec:notation on B+(n) etc}

For $n\geq 1$, we consider the derived $S$-scheme
$$ 
  \bbA^1_{s(n)}[-1]
    :=
  s(n)\times_{0,\bbA^1_{s(n)},0} s(n).
$$
Notice that the structure map $\bbA^1_{s(n)}[-1] \to S$ is a quasi-smooth closed embedding. 
As $n$ varies, we obtain natural functors
$$
  \bbA^1_{s(\bullet)}[-1]:
  \bbN_{\geq 1} 
    \to 
  \dSch_{/S}
$$
\begin{equation} \label{eqn:functor Bplus(n)}
  \Coh(\bbA^1_{s(\bullet)}[-1]):
  \bbN_{\geq 1} 
    \to 
  \dgCat_A,
\end{equation}
the latter induced by the push-forwards along $\bbA^1_{s(n)}[-1] \hto \bbA^1_{s(m)}[-1]$ for $n \leq m$.
We set 
$$
  \sB^+(n):
    = 
  \Coh\bigt{\bbA^1_{s(n)}[-1]},
$$
$$
  \sB^+(\infty)
    :=
  \varinjlim_{\bbN} \; \sB^+(n+1).
$$

\sssec{} 

Recall that 
$$
  \bbA^1_S[-1]
    :=
  S\times_{0,\bbA^1_S,0}S
$$
and notice that $\bbA^1_{s(n)}[-1]$ sits in the fiber product

\begin{equation}
  \label{diag:fiber-prod-for-G(n)}
  \begin{tikzpicture}[scale=1.5]
    \node (00) at (0,0) {$s(n)$};
    \node (10) at (2,0) {$S.$ };
    \node (01) at (0,1) {$\bbA^1_{s(n)}[-1]$};
    \node (11) at (2,1) {$ \bbA^1_S[-1] $}; 
    \path[->,font=\scriptsize,>=angle 90]
    (00.east) edge node[above] {${}$}  (10.west); 
    \path[->,font=\scriptsize,>=angle 90]
    (01.east) edge node[above] {$r(n)$} (11.west); 
    \path[->,font=\scriptsize,>=angle 90]
    (01.south) edge node[right] {${}$} (00.north);
    \path[->,font=\scriptsize,>=angle 90]
    (11.south) edge node[right] {${}$} (10.north);
    \end{tikzpicture}
\end{equation}

Pushing forward along $r(n)$, we obtain a dg-functor
\begin{equation} \label{eqn:push from G(n) to shifted A1}
  \sB^+(n)
    =
  \Coh \bigt{\bbA^1_{s(n)}[-1]}
    \to 
  \Coh (\bbA^1_S[-1]),    
\end{equation}
which lands in the full subcategory $\Coh (\bbA^1_S[-1])_s \subseteq \Coh (\bbA^1_S[-1])$. 
Indeed, $\Coh \bigt{\bbA^1_{s(n)}[-1]}$ is Karoubi-generated by $\bigt{s\to \bbA^1_{s(n)}[-1]}_*\ccO_s$ and this object is sent to $(s\to \bbA^1_S[-1])_*\ccO_s$, which belongs to (and in fact is a Karoubi-generator of) $\Coh(\bbA^1_S[-1])_s$.

\sssec{}

The functoriality of \eqref{eqn:push from G(n) to shifted A1} in $n$ yields the dg-functor
$$
  r(\oo)_*
    :=
  \varinjlim_{\bbN}r(n+1)_*:
  \sB^+(\infty) 
    = 
  \varinjlim_{\bbN} \sB^+(n+1) 
    \to
  \Coh (\bbA^1_S[-1])_s,
$$
which is the second functor used in defining $\oint$.

\sssec{}

We now proceed with the construction of \eqref{eqn:first-arrow}, that is, a dg-functor $\HH(\sB^+/A) \to \sB^+(\infty)$. 
Thanks to the equivalence
$$
  \HH(\sB^+/A)
    \simeq 
  \varinjlim_{\bbN} 
  F_n\HH(\sB^+/A)
$$
and the definition
$$
  \sB^+(\infty)
    =
  \varinjlim_{\bbN} \sB^+(n+1),
$$
it suffices to construct a compatible family of dg-functors
$$
  \xi^+_n : 
  F_n\HH(\sB^+/A)
    \to
  \sB^+(n+2).
$$

\begin{example}
    
To illustrate the main ideas without clutter, we treat the example of $\xi^+_1$ explicitly. This amounts to defining a dg-functor
\begin{equation*}
\begin{tikzcd}
 \colim \Big(
    \sB^+ 
    \otimes_A
    \sB^+
    \arrow[r,swap,shift right=2, "\on m^{\rev}"]
    \arrow[r,shift left=2, "\on m" ]
    &
    \sB^+ \Big)
     \arrow[r, "\xi_1^+"]
    &
    \sB^+(3).
\end{tikzcd}
\end{equation*}
In this special case we do not need Theorem \ref{thm: master homotopies}, hence we are able to land in $\sB^+(2)$. 
Concretely, then, we will define a dg-functor $\xi_{1,0}^+: \sB^+ \to \sB^+(2)$, together with a homotopy between $\xi_{1,0}^+ \circ \on{m}$ and $\xi_{1,0}^+ \circ \on{m}^{\rev}$.

\sssec*{The dg-functor $\xi^+_{1,0}$}

In a nutshell, this is induced by pull-push along the correspondence $G \xleftarrow{\delta_{s/S}} s \xto{j(2)} s(2)$. Let us explain this in detail. Pullback along $\delta_{s/S} : s \to G$ yields a dg-functor
$$
\delta_{s/S}^{*,\enh}:
\sB^+ \to \Mod_E(\Cohminus(s)),
$$
where 
$$
  E:= 
  \End_G(k)
    := 
  \pr{1*}
  \ul\Hom_G(k,k)
    \simeq
  \pr{2*}
  \ul\Hom_G(k,k),
$$
the equivalence on the right being a consequence of the projection formula (after having written a natural algebra map between the two).
In fact, $E \simeq \ccO_s[u]$ with $u$ a free variable in degree $2$.
The pushforward along $j(2)$ induces a dg-functor
$$
  j(2)_{*}^{\enh}:
  \Mod_{\ccO_s[u]}\bigt{ \Cohminus(s)}
    \to 
  \Mod_{\ccO_{s(2)}[u]}\bigt{ \Cohminus(s(2)) }.
$$
Composing the two dg-functors above, we obtain 
$$
  \sB^+
    \to  
  \Mod_{\ccO_{s(2)}[u]}\bigt{ \Cohminus(s(2)) }
$$
which is easily seen to land in the full subcategory
$$
\Mod^{\etp}_{\ccO_{s(2)}[u]}\bigt{ \Cohminus(s(2)) }.
$$
By Lemma \ref{lem: Coh(A1[-1]) vs Cohminus etp}, the latter is equivalent to $\sB^+(2)$. This completes the construction of $\xi^+_{1,0}$.

\sssec*{The homotopy, part 1}

Our next task is to provide a homotopy $\xi_{1,0}^+ \circ \on{m} \simeq \xi_{1,0}^+ \circ \on{m}^{\rev}$. To this end, consider the following commutative diagram, with cartesian squares:
\begin{equation*}
    \begin{tikzcd}
     G
      \rar["\sigma"]
      \arrow[d]
      &
      G
      \rar["\pr{1}"]
      \arrow[d]
        &
      s
      \rar["j(n+2)"]
\dar["\delta_{s/S}"]
        &
      s(n+2)
      \\
         G \times_{\pr1, s, \pr2} G
     \dar["\can"]
      \rar
      &
      G \times_{\pr2, s, \pr1} G
      \dar["\can"]
      \rar["\pr{1,3}"]
        &
      G
        &
      \\
      G \times_S G
      \rar["\rho"]
      &
      G \times_S G.
        &
        &
    \end{tikzcd}
\end{equation*}
There are two compositions $G \to G \times_S G$, the left vertical one and the central one. We denote them by $\tdelta_\sigma$ and $\tdelta$, respectively.\footnote{We observe that these maps are the same, both equivalent to
$$
  G 
    = 
  s \times_S s
    \xto{\pr1 \times \pr2 \times \pr2 \times \pr1}
  s \times_S
  s \times_S
  s \times_S
  s.
$$
We keep different notations for a reason that will become clear momentarily.}
Arguing as above, we see that pull-push along 
$
  G \times_S G 
    \xleftarrow{\tdelta}
  G 
  \xto{\pr1} s
$
induces a dg-functor
$$
  \Cohext(G \times_S G)
    \xto{\tdelta^{*,\enh}} 
  \Mod_{\pr1^*(E)}(\Cohminus(G))
    \xto{\pr{1*}^{\enh}}
  \Mod_E(\Cohminus(s)).
$$
To be totally explicit, $\pr{1*}^{\enh}$ is the composition 
$$
  \Mod_{\pr1^*(E)}(\Cohminus(G))
    \xto{\pr{1*}}
  \Mod_{\pr{1*}\pr1^*(E)}(\Cohminus(s))
    \xto{\mathit{forget}}
  \Mod_E(\Cohminus(s)).
$$
Symmetrically, using the fact that $\pr1 \circ \sigma = \pr2$, the correspondence 
$
  G \times_S G 
    \xleftarrow{\tdelta_\sigma}
  G 
    \xto{\pr2}
  s
$ induces a dg-functor
$$
  \Cohext(G \times_S G)
    \xto{(\tdelta_\sigma)^{*,\enh}} 
  \Mod_{\pr2^*(E)}(\Cohminus(G))
    \xto{\pr{2*}^{\enh}}
  \Mod_E(\Cohminus(s)).
$$
Chasing along the above diagram yields natural transformations
$$
  \xi_{1,0}^+ \circ \on{m}
    \simeq 
  j(2)_*^{\enh}
    \circ 
  \pr{1*}^{\enh}
    \circ 
  (\tdelta)^{*,\enh}
$$
$$
  \xi_{1,0}^+ \circ \on{m}^{\rev}
    \simeq 
  j(2)_*^{\enh}
    \circ 
  \pr{1*}^{\enh}
    \circ 
  (\tdelta)^{*,\enh}
    \circ 
  \rho_*.
$$

\sssec*{The homotopy, part 2}

It follows from base-change that 
$$
  (\tdelta)^{*,\enh}
    \circ 
  \rho_*
    \simeq 
  (\tdelta_\sigma)^{*,\enh}.
$$
This is the reason why we kept distinct notations for $\tdelta_\sigma$ and $\tdelta$: their enhanced pullbacks are a priori distinct. However, we now show that $(\tdelta_\sigma)^{*,\enh}
\simeq 
(\tdelta)^{*,\enh}$.
For this, it suffices to prove that the algebras $\pr1^*(E)$ and $\pr2^*(E)$ are equivalent. Consider the cartesian diagrams
\begin{equation} 
\label{eqn:base-change-example for B(2)}
    \begin{tikzcd}
      G
      \rar["\pr{1}"]
      \arrow[d]
        &
      s
      \rar["i"]
\dar["\delta_{s/S}"]
        &
      S \arrow[d]
      \\
      G \times_{\pr2, s, \pr1} G
      \rar["\pr{1,3}"]
        &
      G \rar
        &
        \bbA^1_S[-1].
    \end{tikzcd}
\end{equation}
The fiber square on the right shows that $E \simeq i^*(E_{\comm})$ where 
$$
E_{\comm} := \End_{\bbA^1_S[-1]}(\ccO_S).
$$
Hence,
$$
  \pr1^*(E) 
    \simeq 
  \pr1^* i^*(E_{\comm})
    \simeq 
  \pr2^* i^*(E_{\comm})
    \simeq 
  \pr2^*(E).
$$
Under the induced equivalence
$$
  \tau: 
  \Mod_{\pr1^*(E)}(\Cohminus(G))
    \xto{\simeq} 
  \Mod_{\pr2^*(E)}(\Cohminus(G)),
$$
the two functors $(\tdelta)^{*,\enh}$ and $(\tdelta_\sigma)^{*,\enh}$ are identified.

\sssec*{The homotopy, part 3}

Combining everything together, we obtain 
$$
  \xi_{1,0}^+ \circ \on{m}^{\rev}
    \simeq 
  j(2)_*^{\enh}
     \circ 
  \pr{1*}^{\enh}
    \circ 
  (\tdelta)^{*,\enh}
    \circ 
  \rho_*
    \simeq 
  j(2)_*^{\enh}
    \circ 
  \pr{1*}^{\enh}
    \circ
  (\tdelta)^{*,\enh}
    \simeq 
  \xi_{1,0}^+ \circ \on{m}.
$$

\end{example}

\sssec{}

Now we treat the general case, which is not much different from the previous case. 
To begin with, we construct a dg-functor
$$
\xi^+_{n,m}:
  \sB^{+,\otimes_Am+1}
    \to
  \sB^+(n+2)
$$
for each $0\leq m \leq n$.
Consider the commutative diagram
\begin{equation*}
    \begin{tikzcd}
      s^{\times_Sm+1}
      \rar["\pr{1}"]
      \arrow[dd, swap, bend right = 3.5cm, "\tdelta"]
      \arrow[d]
        &
      s
      \rar["j(n+2)"]
      \dar["\delta_{s/S}"]
        &
      s(n+2)
      \\
      G^{\times_s m+1}
        \simeq
      s^{\times_S m+2}
      \dar["\can"]
      \rar["\pr{1,m+2}"]
        &
      G
        \simeq
      s\times_Ss
        &
      \\
      G^{\times_S m+1}
        &
        &
    \end{tikzcd}
\end{equation*}
with cartesian square and set
$$
  \xi^{+}_{n,m}
    = 
  j(n+2)_*^{\enh}
    \circ 
  \pr{1*}^{\enh}
    \circ 
  (\tdelta)^{*,\enh}.
$$

\sssec{}

The similar diagram for $s(n+2)$ shows that $\xi^{+,'}_{n,m}$ factors as
$$
  (\sB^+)^{\otimes_A m+1}
    \simeq
  \Cohext(G^{\times_S m+1})
    \xto{j(n+2)_*\circ \pr{1*} \circ \tdelta^*}
  \Mod_{\ccO_{s(n+2)}[u]}
  \Bigt{
        \Cohminus \bigt{s(n+2)}
        }.
$$
Moreover, it is immediate to see that its essential image lands in
$$
 \Mod^{\etp}_{\ccO_{s(n+2)}[u]}
  \Bigt{
  \Cohminus \bigt{s(n+2)}
 }
 \subseteq 
\Mod_{\ccO_{s(n+2)}[u]}
    \Bigt{
   \Cohminus \bigt{s(n+2)}
    }.
$$
Finally, using Lemma \ref{lem: Coh(A1[-1]) vs Cohminus etp}, we obtain a dg-functor
$$
\xi^{+}_{n,m}:
  (\sB^+)^{\otimes_A m+1}
    \simeq
  \Cohext(G^{\times_S m+1})
    \to
  \sB^+(n+2)
    =
  \Coh(\bbA^1_{s(n+2)}[-1]).
$$

\sssec{}

To make the proof of the following crucial lemma as clear as possible, we explicitly write the dg-functors that appear in the diagram
$$
  \on{Bar}_n(\sB^+/A):
  (\Delta^{\leq n}_{\on{inj}})^{\op}
    \to 
  \dgCat_A.
$$
Let $0\leq m \leq n$ and $k\leq n-m$.
Then there are exactly $\prod_{i=1}^{k}(m+i+1)!$ dg-functors
$$
  \sB^{+,\otimes_A m+k+1}
    \to
  \sB^{+,\otimes_A m+1}
$$
to consider, namely the compositions
\begin{equation*}
    \begin{tikzcd}
    \sB^{+,\otimes_A m+k+1}
    \arrow[r,shift right=3]
    \arrow[r, draw=none, "\raisebox{+0.5ex}{\dots}" description]
    \arrow[r,shift left=3]
    &
    \sB^{+,\otimes_A m+k}
    \arrow[r,shift right=3]
    \arrow[r, draw=none, "\raisebox{+0.5ex}{\dots}" description]
    \arrow[r,shift left=3]
    &
    \cdots
    \arrow[r,shift right=3]
    \arrow[r, draw=none, "\raisebox{+0.5ex}{\dots}" description]
    \arrow[r,shift left=3]
    &
    \sB^{+,\otimes_A m+2}
    \arrow[r,shift right=3]
    \arrow[r, draw=none, "\raisebox{+0.5ex}{\dots}" description]
    \arrow[r,shift left=3]
    &
    \sB^{+,\otimes_A m+1}.
    \end{tikzcd}
\end{equation*}
Indeed, for each $\ul d = (d_{m+1},\dots,d_{m+k})\in \prod_{s=m+1}^{m+k}[s]$ we have a composition
$$
  \sB^{+,\otimes_A m+k+1}
    \xto{\on{m}_{d_{m+k}}}
  \sB^{+,\otimes_A m+k}
    \xto{\on{m}_{d_{m+k-1}}}
  \dots
    \xto{\on{m}_{d_{m+2}}}
  \sB^{+,\otimes_A m+2}
    \xto{\on{m}_{d_{m+1}}}
  \sB^{+,\otimes_A m+1}.
$$
Here, for $l\leq k$ and $d\in [m+l]$, the dg-functor
$$
  \sB^{+,\otimes_A m+l+1}
    =
  \Cohext(G^{\times_Sm+l+1})
    \xto{\on{m}_{d}}
  \sB^{+,\otimes_A m+l}
    =
  \Cohext(G^{\times_Sm+l})
$$
is given by pull-push along the following correspondence:
$$
  G^{\times_Sm+l+1}
    \leftarrow
  V(u_{d+1}^{(\star)}-t_{\rho_{m+l+1}(d+1)}^{(\star)})
    \to
  G^{\times_S m+l}.
$$
If $0\leq d\leq m+l-1$,
the second arrow
corresponds to 
$$
  \underbrace{G^{\times_Sd}\times_S (s\times_Ss\times_Ss)\times_S G^{\times_Sm+l-d-1}}_{
    \simeq 
    \; 
  V(u_d^{(\star)}-t_{\rho_{m+l+1}(d)}^{(\star)})
  }
    \xto{\id \times_S \pr{13}\times_S \id}
  \underbrace{G^{\times_Sd}\times_S (s\times_Ss)\times_S G^{\times_Sm+l-d-1}}_{
    \simeq 
    \;
  G^{\times_S m+l}}.
$$
If $d=m+l$, it corresponds to
$$
V(u_{m+l+1}^{(\star)}-t_{1}^{(\star)})
\xto{\vartheta^{(m+l)}}
G^{\times_Sm+l},
$$
where $\vartheta^{(m+l)}$ is defined in Remark \ref{rmk: morphims f^(c)_r}.

\begin{lem} \label{lem:defn of xi-n-plus}
For any $n \geq 0$, the dg-functors $\xi^+_{n,m}$ (with $0 \leq m\leq n$) induce a natural transformation
\begin{equation}\label{eqn: xi_n^+}
\xi_n^+:\on{Bar}_n(\sB^+/A) \to \sB^+(n+2)
\end{equation}
of functors $(\Delta^{\leq n}_{\on{inj}})^{\op}\to \dgCat_A$.
These natural transformations are functorial in $n$, hence they yield a dg-functor 
$
  \xi^+ 
    := 
  \colim_{\bbN} \xi_n^+: 
  \HH(\sB^+/A) 
    \to 
  \sB^+(\infty)$.
\end{lem}

\begin{proof}

For each $\ul d \in \prod_{i=1}^{n}[i]$ we need to supply a coherently homotopic diagram

\begin{equation*}
    \begin{tikzcd}[column sep = 1cm, row sep = 1cm]
    \sB^{+,\otimes_A n+1}
    \rar["\on{m}_{d_{n+1}}"]
    \arrow[rrd]
    &
    \sB^{+,\otimes_A n}
    \rar["\on{m}_{d_{n}}"]
    \arrow[rd]
    &
    \cdots
    \rar["\on{m}_{d_{2}}"]
    &
    \sB^{+,\otimes_A 2}
    \rar["\on{m}_{d_{1}}"]
    \arrow[ld]
    &
    \sB^+
    \arrow[lld]
    \\
    &
    &
    \sB^+(n+2),
    &
    &
    \end{tikzcd}
\end{equation*}
where the downward pointing arrows are the previously defined $\xi^+_{n,m}$.
Theorem \ref{thm: master homotopies} yields the following diagram, commutative up to coherent homotopy:
\begin{equation*}
    \begin{tikzcd}[column sep = 0.5cm]
      G^{\times_S n+1}
      &
      \lar[]
      V(t^{(\star)}_{\rho(d_{n+1})}-u^{(\star)}_{d_{n+1}})
      \ar[dr, phantom, "\square"]
      \rar
      &
      G^{\times_S n}
      &
      \lar
      \cdots
      \rar
      &
      G^{\times_S 2}
      &
      \lar
      V(t^{(\star)}_{\rho(d_{1})}-u^{(\star)}_{d_{1}})
      \rar
      \ar[dr, phantom, "\square"]
      &
      G
      \\
      \uar
      G^{\times_s n+1}
      &
      \uar
      \lar["\id"]
      G^{\times_s n+1}
      \rar
      \ar[dr, phantom, "\square"]
      &
      \uar
      G^{\times_s n}
      &
      \lar
      \cdots
      \rar
      &
      \uar
      G^{\times_s 2}
      &
      \uar
      \lar["\id"]
      G^{\times_s 2}
      \rar
      \ar[dr, phantom, "\square"]
      &
      \uar
      G
      \\
      \uar
      s^{\times_S n+1}
      \dar
      &
      \uar
      \lar["\id"]
      s^{\times_S n+1}
      \rar
      \dar
      &
      \uar
      s^{\times_S n}
      \dar
      &
      \lar
      \cdots
      \rar
      &
      \uar
      s^{\times_S 2}
      \dar
      &
      \uar
      \lar["\id"]
      s^{\times_S 2}
      \rar
      \dar
      &
      \uar
      s
      \dar
      \\
      s(n+2)
      &
      \lar["\id"]
      s(n+2)
      \rar
      &
      s(n+2)
      &
      \lar
      \cdots
      \rar
      &
      s(n+2)
      &
      \lar["\id"]
      s(n+2)
      \rar
      &
      s(n+2)
      \end{tikzcd}
\end{equation*}
where the squares with the symbol $\square$ in the middle are derived cartesian. Applying $\Cohminus(-)$, we immediately obtain a commutative diagram
\begin{equation*}
    \begin{tikzcd}[column sep = 1cm, row sep = 1cm]
    \sB^{+,\otimes_A n+1}
    \rar["\on{m}_{d_{n+1}}"]
    \arrow[rrd,swap,"\xi^+_{n,n}"]
    &
    \sB^{+,\otimes_A n}
    \rar["\on{m}_{d_{n}}"]
    \arrow[rd,"\xi^+_{n,n-1}"]
    &
    \cdots
    \rar["\on{m}_{d_{2}}"]
    &
    \sB^{+,\otimes_A 2}
    \rar["\on{m}_{d_{1}}"]
    \arrow[ld,swap,"\xi^+_{n,1}"]
    &
    \sB^+
    \arrow[lld,"\xi^+_{n,0}"]
    \\
    &
    &
    \Cohminus \bigt{s(n+2)},
    &
    &
    \end{tikzcd}
\end{equation*}
where the downward pointing arrows are the $\xi^{+}_{n,m}$ with the action of $\ccO_{s(n+2)}[u]$ forgotten. 
To see that such actions match, recall thier origin: they come from pullback along the middle vertical arrows, and the fact that such arrows are base-changes of $s \to G$.
Then the assertion follows from the fact that the latter map is a base-change of $S \to \bbA^1_S[-1]$ as in \eqref{eqn:base-change-example for B(2)}.

\end{proof}

\begin{rmk}\label{rmk: composition B^+-->HH(B^+/A)-->Coh(A^1_S[-1])}

By construction and Lemma \ref{lem: Coh(A1[-1]) vs Cohminus etp}, the dg-functor
  $$
    \xi^+_0:
    F_0\HH(\sB^*/A)
      =
    \sB^+
      \to
    \sB^+(2)
  $$
  identifies with
  $$
    (\bbA^1_s[-1]\to \bbA^1_{s(2)}[-1])_*:
    \Coh(G)
        \simeq
    \Coh(\bbA^1_{s}[-1])
        \to
    \Coh(\bbA^1_{s(2)}[-1]).
  $$
  In particular, the composition
  $$
    \sB^+
      \xto{\xi^+_0}
    \sB^+(2)
      \xto{r(2)_*}
    \Coh(\bbA^1_S[-1])
  $$
  is just the pushforward along $G\simeq \bbA^1_s[-1]\to \bbA^1_S[-1]$.
\end{rmk}

\subsection{The construction of the functor \texorpdfstring{$\oint$}{oint}} \label{ssec:oint-construction II}

We use the above construction of $\oint^+$ to define the dg-functor
$$
  \oint: 
  \HH(\sB/A) 
    \to 
  \Sing(\bbA^1_S[-1])_s
$$
as the composition
$$
  \HH(\sB/A)
    \xto{\;\;\xi\;\;} 
  \sB(\infty)
    \xto{\;\;r(\infty)_*\;\;}
  \Sing(\bbA^1_S[-1])_s.
$$

\sssec{}

Let $n\geq 1$ and denote by $\sB^+_{\on{pe}}(n)$ the full subcategory $\Angles{\pr{}^*\Coh(\bbA^1_{s(n)})}$ of 
$
  \sB^+(n)
    =
  \Coh(\bbA^1_{s(n)}[-1])
$ 
(see Section \ref{sssec: pr^*Coh}).
A routine computation shows that \eqref{eqn:functor Bplus(n)} restricts to a functor
$
  n 
    \mapsto 
  \sB^+_{\on{pe}}(n)
$.
Hence, we can consider the colimit dg-category 
$$
  \sB^+_{\on{pe}}(\infty) 
    := 
  \varinjlim_{\bbN} 
    \;
  \sB^+_{\on{pe}}(n+1),
$$
which is a full subcategory of $\sB^+(\infty)$.

\sssec{} 

Define $\sB(n)$ as the quotient
$$
  \sB(n)
    := 
  \frac{\sB^+(n)}{\sB^+_{\on{pe}}(n)}
    \in 
  \dgCat_A.
$$
We obtain a functor 
$
  \sB(\bullet): 
  \bbN_{\geq 1} 
    \mapsto 
  \dgCat_A
$ 
and set
$$
  \sB(\infty)
    :=
  \varinjlim_{\bbN} 
    \;
  \sB(n+1).
$$

\begin{lem}\label{lem: B(oo)=B^+(oo)/B_pe(oo)}
The naturally defined dg-functor
$$
  \sB(\infty)
    \to 
  \frac
       {
        \sB^+(\infty)
       }
       {
        \sB^+_{\on{pe}}(\infty)
       }
$$
is an equivalence. 
\end{lem}

\begin{proof}
Unraveling the definitions, we need to show that the dg-functor
$$
  \varinjlim_{\bbN}
    \biggt{
           \frac{\sB^+(n+1)}{\sB^+_{\on{pe}}(n+1)}
           }
    \to 
  \frac{
        \Bigt{\varinjlim_{\bbN}\sB^+(n+1)}
        }
       {
        \Bigt{\varinjlim_{\bbN}\sB^+_{\on{pe}}(n+1)}
        }
$$
is an equivalence.
Since the quotient functor
$$
  \Fun (\Delta^1,\dgCat_A)
    \to 
  \dgCat_A
    \;\;\;\; 
  (\sU' \to \sU) 
    \mapsto 
  \sU/\sU'
$$
is left adjoint to
$$
  \dgCat_A 
    \to 
  \Fun (\Delta^1,\dgCat_A), 
    \;\;\;\; 
  \sU 
    \mapsto 
  (0\to \sU),
$$
it commutes with colimits.
Thus, the claim follows.
\end{proof}

\sssec{}

Having defined $\sB(n)$ and $\sB(\infty)$, we show that 
$$
r(\oo)_*
:
\sB^+(\infty) 
\to
\Coh (\bbA^1_S[-1])_s
$$
descends to an analogous dg-functor (denoted in the same way) at the level of singularity categories. 

\begin{lem}
The following square commutes:
\begin{equation*}
    \begin{tikzcd}
      \sB^+_{\on{pe}}(\infty)
      \rar["r(\oo)_*"]
      \dar[hook]
      &
      \Perf(\bbA^1_S[-1])_s
      \dar[hook]
      \\
      \sB^+(\infty)
      \rar["r(\oo)_*"]
      &
      \Coh(\bbA^1_S[-1])_s.
    \end{tikzcd}
\end{equation*}
In particular, $r(\oo)_*:\sB^+(\oo)\to \Coh(\bbA^1_S[-1])_s$ induces a dg-functor
$$
r(\oo)_*:
\sB(\oo)
\to 
\Sing(\bbA^1_S[-1])_s.
$$
\end{lem}

\begin{proof}
By Lemma \ref{lem: B(oo)=B^+(oo)/B_pe(oo)} it suffices to check that \eqref{eqn:push from G(n) to shifted A1} sends $\sB^+_{\on{pe}}(n)$ to $\Perf(\bbA^1_S[-1])$. 
This follows immediately from the definitions and the fiber square
\begin{equation*}
    \begin{tikzcd}
      \bbA^1_{s(n)}[-1] 
        \arrow[d,"r(n)"] 
        \arrow[r,"\pr{}"] 
        & 
      s(n)  
        \arrow[d,"i(n)"] 
        \\
      \bbA^1_S[-1]
        \arrow[r,"\pr{}"] 
        &  
      S.
    \end{tikzcd}
\end{equation*}
Then $r(n)_*\pr{1}^*E \simeq \pr{1}^*i(n)_*E$, which is perfect since $S$ is regular. The second statement follows readily from the first and the previous lemma.
\end{proof}

\sssec{}

We now proceed with the construction of $\xi: \HH(\sB/A) \to \sB(\infty)$. 
As in the case of $\xi^+$, it suffices to construct a compatible family of dg-functors
$$
  \xi_n : 
  F_n\HH(\sB/A)
    \to
  \sB(n+2).
$$
Equivalently, we will construct a family of compatible natural transformations
\begin{equation}\label{eqn: xi_n}
  \xi_n:\on{Bar}_n(\sB/A) 
    \to 
  \sB(n+2),
\end{equation}
where $\sB(n+2)$ is regarded as a (constant) functor 
$
  (\Delta^{\leq n}_{\on{inj}})^{\op}
    \to 
  \dgCat_A.
$

\begin{cor}

The family of compatible natural transformations \eqref{eqn: xi_n^+} induces a family of compatible natural transformations \eqref{eqn: xi_n}.
\end{cor}

\begin{proof}

It suffices to show that each of the dg-functors
\begin{equation} \label{eqn:mult-to-Bplus(n)}
  \Cohext(G^{\times_Sm})
    \simeq 
  (\sB^+)^{\otimes_Am}
    \to 
  \sB^+(n+2)
\end{equation}
of the above Lemma \ref{lem:defn of xi-n-plus} induces a dg-functor
$$
  \Perfext(G^{\times_Sm})
    \to 
  \sB^+_{\on{pe}}(n+2).
$$
To see this, recall that \eqref{eqn:mult-to-Bplus(n)} can be decomposed as
$$
  (\sB^+)^{\otimes_Am}
    \to 
  \sB^+ 
    \to 
  \sB^+(n+2).
$$
Since $\sB^+$ is generated by its monoidal unit, it follows that $\Perfext(G^{\times_Sm})$ gets mapped to $\sB^+_{\on{pe}}$ under the first arrow. 
It remains to verify that $\sB^+ \to \sB^+(n+2)$ sends  $\sB^+_{\on{pe}}$ to $\sB^+_{\on{pe}}(n+2)$: this readily follows from the definitions and from base change along the cartesian square
\begin{equation*}
    \begin{tikzcd}
      G
        \arrow[d,"r(n+2)"] 
        \arrow[r,"\pr{1}"] 
        & 
      s
        \arrow[d,"i(n+2)"] 
        \\
      \bbA^1_{s(n+2)}[-1]
        \arrow[r,"\pr{}"] 
        &  
      s(n+2).
    \end{tikzcd}
\end{equation*}
Indeed, under the equivalence of Lemma \ref{lem: Coh(A1[-1]) vs Cohminus etp}, the pushforward along $s \to s(n+2)$ corresponds to the pushforward along $G\simeq \bbA^1_s[-1] \to \bbA^1_{s(n+2)}[-1]$
\end{proof}

\sssec{}

Let 
$
  \xi_n: 
  F_n\HH(\sB/A)
  \to 
  \sB(n+2)
$
denote the dg-functor corresponding to \eqref{eqn: xi_n}.
To summarize, we have defined
$$
  \oint: 
  \HH(\sB/A)
    \to 
  \Sing(\bbA^1_S[-1])_s
$$
as the composition
$$
  \HH(\sB/A)
    \simeq
  \varinjlim_{\bbN}
  F_n \HH(\sB/A)
    \xto{
         \xi:=\varinjlim_{\bbN} \xi_n
         }
  \varinjlim_{\bbN} \sB(n+1)
    =
  \sB(\oo)
    \xto{r(\oo)_*
         }
  \Sing(\bbA^1_S[-1])_s.
$$

\begin{rmk}\label{rmk: composition B-->HH(B/A)-->Sing(A^1_S[-1])}
  It follows immediately from Remark \ref{rmk: composition B^+-->HH(B^+/A)-->Coh(A^1_S[-1])} that the composition
  $$
    \sB
      \to
    \HH(\sB/A)
      \xto{\oint}
    \Sing(\bbA^1_S[-1])_s
  $$
  identifies with the functor induced by the pushforward along 
  $$
  G
    \simeq
  \bbA^1_s[-1]
    \to
  \bbA^1_S[-1].
  $$
\end{rmk}

\subsection{Motivic realization of the Drinfeld cocenter} \label{ssec:realization of HH}

We conclude this section with the following result, which may be regarded as the first main result of this paper.

\begin{thm}\label{thm: M(B) is a retract of M(HH(B/A))}

The morphism of motivic spectra
$$
  \Mv_S(\sB)
    \to 
  \Mv_S \bigt{\HH(\sB/A)}
$$
induced by the canonical dg-functor $\sB \to \HH(\sB/A)$
admits a natural retraction.
In particular, there are natural retractions 
\begin{equation*}
  \begin{tikzcd}
    \bbZ 
      \arrow[r,hook]
      \arrow[d,hook]
      &
    \HK_0\bigt{\HH(\sB/A)}
      \arrow[d,"\chern"]
      \arrow[r,->>] 
      &
    \bbZ
      \arrow[d,hook]
      \\
    \Qell 
      \arrow[r,hook]
      &
      \uH^0_{\et}\Bigt{S,\rl_S\bigt{\HH(\sB/A)}} 
      \arrow[r,->>]
      &
    \Qell.
  \end{tikzcd}
\end{equation*}
\end{thm}

\begin{proof}
By Remark \ref{rmk: composition B-->HH(B/A)-->Sing(A^1_S[-1])}, the composition
\begin{equation} \label{eqn:composition in a lemma}
  \sB 
    \to
  \HH(\sB/A)
    \xto{\oint}
  \Sing(\bbA^1_S[-1])_s   
\end{equation}
is homotopic to the dg-functor $\Sing(G) \to \Sing(\bbA^1_S[-1])_s$ induced by pushforward along 
\begin{equation} \label{eqn:from G to shifted A1}
  G
    =
  s\times_Ss
    \to 
  S\times_{\bbA^1_S}S
    =
  \bbA^1_S[-1].
\end{equation}
Let us show that this dg-functor induces the identity map in $\SH_S$.
Thanks to \cite[Proposition 3.30]{brtv18}, we know that 
$
  \Mv_T\bigt{\Sing(\bbA^1_T[-1])}
    \simeq 
  \BU_T \oplus \BU_T[1]
$
whenever $T$ a regular scheme of Krull dimension $\leq 1$.
Let $i:s\hto S$. 
Combining the localization sequence
$$
  \Sing(\bbA^1_S[-1])_s 
    \hto 
  \Sing(\bbA^1_S[-1]) 
    \to 
  \Sing(\bbA^1_{\eta}[-1])
$$
with the purity equivalence $i^!\BU_S \simeq \BU_s$,
we deduce that
$$
  \Mv_S\bigt{\Sing(\bbA^1_S[-1])_s} 
    \simeq 
  i_*\BU_s \oplus i_* \BU_s[1].
$$
On the other hand, the derived $S$-scheme $G$ is equivalent to the derived $S$-scheme $\bbA^1_s[-1]$ and, under this identification, the morphism \eqref{eqn:from G to shifted A1} goes over to the closed embedding $\bbA^1_s[-1] \to \bbA^1_S[-1]$ induced by $i$.
It is then clear that the image of the dg-functor \eqref{eqn:composition in a lemma} under $\Mv_S$ is the identity.
The latter statement of the theorem is an immediate consequence of the former.
\end{proof}

\begin{rmk}
When $S$ is of pure characteristic, there is a dg-functor
$$
  \Sing(\bbA^1_S[-1])_s
    \to 
  \Sing(G)
$$
induced by pushforward along the morphisms $\bbA^1_S[-1]\to G$ induced by the morphism $S\to s$.
Then the composition
$$
  \HH(\sB/A)
    \xto{\oint} 
  \Sing(\bbA^1_S[-1])_s
    \to 
  \Sing(G) 
    = 
  \sB
$$
recovers the dg-functor $\on{m}:\HH(\sB/A)\to \sB$ of Section \ref{sssec: retraction B to HH(B/A) in pure char}.
\end{rmk}

\begin{rmk}
While the above Theorem does not fully confirm the expectations of \TV \footnote{They conjectured that the canonical dg-functor $\sB \to \HH(\sB/A)$ is an $\bbA^1$-homotopy equivalence, see \cite[Remark 5.2.3]{toenvezzosi22}.}, it is enough to make their approach to intersection theory on arithmetic schemes successful. We will see how this works in Sections \ref{sec: categorical intersection theory} and \ref{sec: proof of gBCC}.
\end{rmk}


\section{Categorical intersection theory}\label{sec: categorical intersection theory}

In this section, we use the dg-functor $\oint: \HH(\sB/A)\to \Sing(\bbA^1_S[-1])_s$ to study the intersection theory of arithmetic schemes\footnote{Recall that an \emph{arithmetic scheme} is an $S$-scheme $X$ as in Section \ref{hypothesis gBCC}.} as follows.
In \cite{toenvezzosi22},
\TV{} introduced a dg-functor
$$
\ev^{\HH}:
\Sing(X\times_SX)
\to
\HH(\sB/A)
$$
which can be regarded as a 
\emph{categorical intersection product with the diagonal}.
As observed there, the fact that the motivic and $\ell$-adic realization of $\HH(\sB/A)$ are unknown (at least in mixed characteristic) makes it hard to extract numerical invariants out of $\ev^{\HH}$.
To bypass this issue, we consider the \emph{integration dg-functor}, defined as the composition
$$
\int_{X/S}:
\Sing(X\times_SX)
\xto{\ev^{\HH}}
\HH(\sB/A)
\xto{\oint}
\Sing(\bbA^1_S[-1])_s.
$$
In this case, the realizations of $\Sing(\bbA^1_S[-1])_s$ are known after \cite{brtv18} and it remains to identify the numerical invariants captured by $\int_{X/S}$. To this end, we construct a commutative diagram
\begin{equation}\label{diagram: alternative expressions integration functor}
    \begin{tikzcd}
      \Sing(X\times_SX)
      \arrow[rr,bend left = 0.5cm,"\Upsilon"]
      \arrow[d,"\ev^{\HH}"]
      \arrow[r,"\Xi"]
      &
      \varinjlim_{\bbN} 
      \MFcoh(X_{s(n+2)},0)_{\red{Z}}
      \arrow[d,"\ev_\infty"]
      \arrow[r,"R(\oo)_*"]
      &
      \Sing(\bbA^1_X[-1])_{\red{Z}}
      \arrow[d,"p_*"]
      \\
      \HH(\sB/A)
      \arrow[r,"\xi"]
      \arrow[rr,swap,bend right = 0.5cm, "\oint"]
      &
      \sB(\oo)
      \arrow[r,"r(\oo)_*"]
      &
      \Sing(\bbA^1_S[-1])_s,
    \end{tikzcd}
\end{equation}
thus obtaining an alternative expression for the integration dg-functor
 $$
 \int_{X/S}:=
\oint \circ \ev^{\HH} 
\simeq 
p_* \circ \Upsilon.
$$
As shown later in Section \ref{sec: proof of gBCC},
this new expression connects to Kato--Saito's localized intersection product and will eventually let us prove Conjecture \ref{conj:gBCC}.

\subsection{The \texorpdfstring{$\sB$}{B}-action on \texorpdfstring{$\sT$}{T} and the categorical Thom--Sebastiani theorem}

In this preliminary subsection, we first review the definition of the convolution actions of $\sB^+$ on $\Coh(X_s)$ and of $\sB$ on $\Sing(X_s)$. 
Afterwards, we recall the statement and the proof of the Thom--Sebastiani theorem for coherent and singularity categories. The latter is due to {\TV{}}. 
The reader may consult \cite{toenvezzosi22, beraldopippi24} for more details.

\sssec{}

Let $p:X\to S$ be as in Section \ref{hypothesis gBCC}.
The special fiber $X_s$ is Zariski locally of the form $\Spec{R[\eps]}$, where $\Spec{R}\hto X$ is an open subscheme and $R[\eps]$ is the simplicial commutative $R$-algebra freely generated by an element $\eps$ in degree $1$ mapped to $\pi$ by the differential.

\sssec{}

Consider $\dgMod(R[\eps])$, the dg-category of $R[\eps]$-dg-modules endowed with the projective model structure, and let $\cofdgMod(R[\eps])$ be the subcategory of cofibrant $R[\eps]$-dg-modules whose underlying $R$-module is perfect.
Recall now the Hopf algebroid $A[\eps_1, \eps_2]$ defined in Section \ref{sssec:Hopf-algebroid}.
We consider the dg-functor
$$
  -\odot-:
  \cofdgMod(A[\eps_1,\eps_2])\otimes_A \cofdgMod(R[\eps]) 
    \to
  \cofdgMod(R[\eps]) 
$$
$$
  (M,E)
    \mapsto
  M\odot E:= \tau_*M\otimes_{A[\eps]}E,
$$
which is well-defined since the model and tensor structures on $\dgMod(R[\eps])$ are compatible.

\sssec{}

Denote by $W_{\on{qi}}$ the class of quasi-isomorphisms and by $W_{\on{pe}}$ the class of morphisms whose homotopy cone is a perfect $R[\eps]$-dg-module. 
It is easy to verify that $-\odot-$ is compatible with both $W_{\on{qi}}$ and $W_{\on{pe}}$. 
In view of
$$
  L_{W_{\on{qi}}}\bigt{\cofdgMod(R[\eps])} 
    \simeq
  \Coh \bigt{\Spec{R[\eps]}}
$$
$$
  L_{W_{\on{pe}}}\bigt{\cofdgMod(R[\eps])}
    \simeq
  \Sing \bigt{\Spec{R[\eps]}},
$$
we see that $\Coh \bigt{\Spec{R[\eps]}}$ is endowed with a left $\sB^+$-module structure and $\Sing \bigt{\Spec{R[\eps]}}$ with a left $\sB$-module structure. 
These structures are natural in $R$.

\sssec{}

The $\infty$-categories of left $\sB^+$-modules and of left $\sB$-modules both admit limits, which commute with the forgetful functors to $\dgCat_A$.
In view of the equivalences
$$
  \Coh(X_s)
    \simeq 
  \varprojlim_{\Spec{R}\hto X}\Coh\bigt{\Spec{R[\eps]}}
$$
$$
  \Sing(X_s)
    \simeq 
  \varprojlim_{\Spec{R}\hto X}\Sing \bigt{\Spec{R[\eps]}},
$$
we see that $\Coh(X_s)$ can be endowed with the structure of a left $\sB^+$-module and $\Sing(X_s)$ with that of a left $\sB$-module.

Passing to the opposite dg-categories, we obtain a right $\sB^+$-action on $\Coh(X_s)^{\op}$ and a right $\sB$-action on $\Sing(X_s)^{\op}$.

\begin{rmk}

Notice that the action of $\sB^+$ on $\Coh(X_s)$ and $\Perf(X_s)$ is compatible with imposing support conditions:
if $K \subseteq X_s$ is a closed subset, then $\Coh(X_s)_K$ and $\Perf(X_s)_K$ are left $\sB^+$-modules.
The same is true when passing to the quotient: the left $\sB$-module structure  on $\Sing(X_s)$ induces a left $\sB$-modules structure on $\Sing(X_s)_K$.
\end{rmk}

\sssec{}

We are now ready to recall the following fundamental theorem, due to \TV.

\begin{thm}[{\cite[Theorem 4.2.1]{toenvezzosi22}}]\label{thm: Thom--Sebastiani for MF}
Let $X$ and $Y$ be two $S$-schemes as in Section \ref{hypothesis gBCC} and denote by $Z$ and $W$ the singular loci of $X/S$ and $Y/S$, respectively.
Denote by $\red{Z}$ and $\red{W}$ the associated reduced schemes.
Then there are canonical equivalences
$$
  \ffF^+: 
  \Coh(X_s)_{\red{Z}}^{\op}
  \underset{\sB^+}\otimes
  \Coh(Y_s)_{\red{W}}
    \xto{\simeq}
  \Coh(X\times_S Y)_{\red{Z} \times_s\red{ W}}
$$
$$
  \ffF: 
  \Sing(X_s)^{\op}
  \underset{\sB}\otimes
  \Sing(Y_s)
    \xto{\simeq}
  \Sing(X\times_S Y)
$$
in
$\dgCat_A$.
\end{thm}

In view of the importance of this theorem, we sketch the proof given in \cite{toenvezzosi22}.

\begin{proof}[Sketch of the proof]
Let $E\mapsto E^{\vee}$ denote Grothendieck duality $\Coh(X_s)\xto{\sim} \Coh(X_s)^{\op}$ and consider the dg-functor
$$
  \ol \ffF:
  \Coh(X_s)^{\op}\otimes_A \Coh(Y_s)
    \to
  \Coh(X\times_S Y)_{X_s \times_s Y_s}
$$
$$
  (E,F)
    \mapsto 
  (X_s\times_s Y_s \to X\times_SY)_*(E^\vee \boxtimes_s F)
$$
and observe that is it $(\sB^+)^\env$-invariant. Hence, we obtain a dg-functor
$$
  \ffF^+:
  \Coh(X_s)^{\op}
  \otimes_{\sB^+} 
  \Coh(Y_s)
    \to
  \Coh(X\times_S Y)_{X_s \times_s Y_s},
$$
which is shown to be fully faithful in \cite[Lemma 4.2.3]{toenvezzosi22}. 
Since the external tensor product 
$$
  -\boxtimes_s-: 
  \Coh(X_s)\otimes_A \Coh(Y_s)
    \to 
  \Coh(X_s \times_s Y_s)
$$ 
and the pushforward
$$
  (X_s \times_s Y_s \to X\times_S Y)_*:
  \Coh(X_s\times_s Y_s)
    \to
  \Coh(X\times_S Y)_{X_s\times_s Y_s}
$$
are both Karoubi-essentially surjective, so is $\ffF^+$.
It follows that $\ffF^+$ is an equivalence in $\dgCat_A$.
Taking supports into account, it is clear that $\ffF^+$ restricts to an equivalence
$$
  \ffF^+:
  \Coh(X_s)^{\op}_{\red{Z}}
  \otimes_{\sB^+} 
  \Coh(Y_s)_{\red{W}}
    \to
  \Coh(X\times_S Y)_{\red{Z} \times_s\red{ W}}.
$$
Next, one checks that $\ol\ffF$ induces equivalences
$$
  \Angles{
          \Bigt{
                \Coh(X_s)^{\op}
                \otimes_{\sB^+} 
                \Perf(Y_s) 
                }
  \cup
  \Bigt{
        \Perf(X_s)^{\op}
        \otimes_{\sB^+} 
        \Coh(Y_s)
        }
         }
    \simeq
    \Perf(X\times_S Y)_{X_s \times_s Y_s}
$$
$$
  \Angles{
  \Bigt{
        \Coh(X_s)^{\op}_{\red{Z}}
        \otimes_{\sB^+} 
        \Perf(Y_s)_{\red{W}}
        }
    \cup
  \Bigt{
        \Perf(X_s)^{\op}_{\red{Z}}
        \otimes_{\sB^+} 
        \Coh(Y_s)_{\red{W}}
        }
         }
    \simeq
  \Perf(X\times_S Y)_{\red{Z} \times_s\red{ W}}.
$$
Since $X\setminus Z \to S$ and $Y \setminus W \to S$ are smooth morphisms, we have
$$
  \Sing(X_s)_{\red{Z}} 
    \simeq
  \Sing(X_s),
    \;\;\;\;
  \Sing(Y_s)_{\red{W}} 
    \simeq
  \Sing(Y_s).
$$
Taking quotients, we thus obtain the equivalence
$$
  \ffF: 
  \Sing(X_s)^{\op} \otimes_{\sB} \Sing(Y_s)
    \xto{\simeq }
  \frac{\Coh(X\times_SX)_{\red{Z}\times_s\red{W}}}{\Perf(X\times_SX)_{\red{Z}\times_s\red{W}}}
    \simeq 
  \Sing(X\times_SX),
$$
where the last step follows from the fact that $X\times_S Y \setminus (\red{Z} \times_s \red{W})$ is smooth over $S$.
\end{proof}

\sssec{}

For the proof of Conjecture \ref{conj:gBCC}, we just need to consider the above theorem in the special case $X = Y$. 
We adopt the following notations for the dg-categories
$$
  \sT^+ 
    := 
  \Coh(X_s)_{\red{Z}},
    \hspace{.4cm}
  \sT^+_{\on{pe}} 
    := 
  \Perf(X_s)_{\red{Z}},
    \hspace{.4cm}
  \sT 
    := 
  \Sing(X_s)\simeq \Sing(X_s)_{\red{Z}}.
$$
The first two will be regarded as left modules over $\sB^+$, 
while the third one as a left module over $\sB$.

\sssec{}

The relative tensor product
$
  (\sT^+)^{\op}
    \otimes_{\sB^+}
  \sT^+
$
is computed as the colimit of the following semi-simplicial diagram
$$
  \on{Bar}(\sT^+):
  (\Delta_{\inj})^{\op}
    \to
  \dgCat_A
$$
\begin{equation*}
\begin{tikzcd}
    \cdots
    \;\;
    \arrow[r,shift right=3]
    \arrow[r,shift right=1]
    \arrow[r,shift left=1]
    \arrow[r,shift left=3]
    &
    (\sT^+)^{\op} \otimes_A \sB\otimes_A \sB \otimes_A \sT^{+}
    \arrow[r,shift right=2]
    \arrow[r]
    \arrow[r,shift left=2]
    &
    (\sT^+)^{\op} \otimes_A \sB \otimes_A \sT^{+}
    \arrow[r,shift right=1]
    \arrow[r,shift left=1]
    &
    (\sT^+)^{\op}\otimes_A\sT^{+}.
\end{tikzcd}
\end{equation*}

We will denote 
$$
  \on{Bar}_{n}(\sT^+)
    :=
  \restr{\on{Bar}(\sT^+)}{(\Delta_{\inj}^{\leq n})^{\op}}:
  (\Delta_{\inj}^{\leq n})^{\op}
    \to
  \dgCat_A
$$
its restriction to 
$
  (\Delta_{\inj}^{\leq n})^{\op}
    \subseteq 
  (\Delta_{\inj})^{\op}
$
and 
$$
  F_n\bigt{(\sT^+)^{\op}\otimes_{\sB^+}\sT^+}
    :=
  \varinjlim \on{Bar}_n(\sT^+)
    \in
  \dgCat_A.
$$

We will use an analogous notation for the $\sB$-module $\sT$:
we let
$$
  \on{Bar}(\sT):
  (\Delta_{\inj})^{\op}
    \to
  \dgCat_A
$$
denote the bar complex of $\sT$ relative to $\sB$,
$$
  \on{Bar}_{n}(\sT)
    :=
  \restr{\on{Bar}(\sT)}{(\Delta_{\inj}^{\leq n})^{\op}}:
  (\Delta_{\inj}^{\leq n})^{\op}
    \to
  \dgCat_A
$$
its $n^{th}$ truncation and
$$
  F_n\bigt{\sT^{\op}\otimes_{\sB}\sT}
    :=
  \varinjlim \on{Bar}_n(\sT)
    \in
  \dgCat_A
$$
the corresponding colimit.

\begin{rmk}
  We believe that, using the coherent homotopies of Section \ref{sec: the E_2 structure on G} it is possible to enhance Theorem \ref{thm: Thom--Sebastiani for MF} as follows: for each $n \geq 0$ there are Morita equivalences
  $$
  \ffF^+_n: 
  F_n\bigt{\Coh(X_s)_{\red{Z}}^{\op}
  \underset{\sB^+}\otimes
  \Coh(Y_s)_{\red{W}}}
    \xto{\simeq}
  \Coh(X\times_S s(n+1)\times_S Y)_{\red{Z} \times_s\red{ W}}
$$
$$
  \ffF_n: 
  F_n \bigt{\Sing(X_s)^{\op}
  \underset{\sB}\otimes
  \Sing(Y_s)}
    \xto{\simeq}
  \Sing(X\times_Ss(n+1)\times_S Y)
$$
such that
$$
  \varinjlim_n \ffF^+_n \simeq \ffF^+,
  \;\;\;\;
  \varinjlim_n \ffF_n \simeq \ffF.
$$

\end{rmk}

\subsection{A functor out of the coherent Thom--Sebastiani category} \label{ssec:functor from TBT-plus}

Our current goal is to construct a dg-functor
\begin{equation} \label{eqn: from TBTplus to colim Coh(A^1[-1])}
  \Xi^+:
  (\sT^+)^{\op}
  \underset{\sB^+}\otimes
  \sT^+
    \to 
  \colim_{\bbN} 
  \Coh(\bbA^1_{X_s(n)}[-1])_{\red{Z}}.
\end{equation}

Very much in the spirit of this paper, the construction exploits the realization of the tensor product on the left as a colimit of truncated bar diagrams.

Namely, we plan to construct \eqref{eqn: from TBTplus to colim Coh(A^1[-1])} as the composition 
$$
  (\sT^+)^\op \otimes_{\sB^+}\sT^+
    \xto{\on{can}^{-1}}
  \colim_n
  F_n 
  \bigt{
        (\sT^+)^\op \otimes_{\sB^+}\sT^+
        }
    \xto{\colim_n \Xi^+_n}
  \colim_n
  \Coh(\bbA^1_{X_s(n+2)}[-1])_{\red{Z}}.
$$
In this section, we define the $\Xi_n^+$'s (for $n \geq 0$) and assemble them into the colimit.

\sssec{}

Defining $\Xi_n^+$ is equivalent to exhibiting a natural transformation 
$$
  \on{Bar}_n(\sT^+)
    \to 
  \Coh(\bbA^1_{X_s(n+2)}[-1])_{\red{Z}},
$$
where 
$
  \Coh(\bbA^1_{X_{s(n+2)}}[-1])_{\red{Z}}
$
is considered as a constant diagram $(\Delta_{\on{inj}}^{\leq n})^{\op}\to \dgCat_A$.

\sssec{}

We first need to find geometric descriptions of the nodes of the bar complex $\on{Bar}(\sT^+)$.
This is done once again via the dg-categories introduced in Section \ref{ssec: categories external products}.

\begin{lem}\label{lem: description T+otimes B+otimesn otimes T+}
For $n\geq 0$, there is a canonical equivalence
$$
  (\sT^+)^\op 
    \otimes_A 
  (\sB^+)^{\otimes_An}
    \otimes_A 
  \sT^+
    \simeq 
  \Cohext(X_s \times_S G^{\times_S n}\times_S X_s)_{\red{Z}\times_s\red{Z}},
$$
with the convention that $G^{\times_S 0} = S$.

The dg-category 
$$
  \Cohext(X_s \times_S G^{\times_S n}\times_S X_s)_{\red{Z}\times_s\red{Z}}
    \subseteq 
  \Coh(X_s \times_S G^{\times_S n}\times_S X_s)
$$
is the full subcategory Karoubi-generated by external tensor products (over $S$) of objects 
$$
  E 
    \in 
  \Coh(X_s)_{\red{Z}},
    \;\;\;\;
  M_1
    \in 
  \Coh(G),
    \dots, 
  M_n 
    \in 
  \Coh(G), 
    \;\;\;\;
  F
    \in 
  \Coh(X_s)_{\red{Z}}.
$$
\end{lem}

\begin{proof}
From the definitions, we have
$$
  (\sT^+)^\op 
    \otimes_A 
  (\sB^+)^{\otimes_An}
    \otimes_A 
  \sT^+
    =
  \Coh(X_s)^{\op}_{\red{Z}}\otimes_A \bigt{\Coh(G)}^{\otimes_An}\otimes_A\Coh(X_s)_{\red{Z}}.
$$
Lemma \ref{lem:Cohext contained in Coh} guarantees that 
$\Cohext(X_s \times_S G^{\times_Sn}\times_SX_s)_{\red{Z}\times_s\red{Z}}$
is a full subcategory of $\Coh(X_s \times_S G^{\times_S n}\times_S X_s)_{\red{Z}\times_s\red{Z}}$:
use for instance the closed embeddings $X_s \hto X$ and $G \hto S$ into regular $S$-schemes. 

Consider the external tensor product
$$
  \Coh(X_s)^{\op}_{\red{Z}}
    \otimes_A 
  \bigt{\Coh(G)}^{\otimes_An}
    \otimes_A
  \Coh(X_s)_{\red{Z}}
    \to
  \Coh(X_s \times_S G^{\times_S n}\times_S X_s)_{\red{Z}\times_s\red{Z}}
$$
$$
  (E,M_1,\dots,M_n,F)
    \mapsto 
  E^{\vee}
    \boxtimes_S
  M_1 
    \boxtimes_S 
  \dots 
    \boxtimes_S 
  M_n 
    \boxtimes_S 
  F.
$$
This dg-functor obviously lands in $\Cohext(X_s \times_S G^{\times_Sn}\times_SX_s)_{\red{Z}\times_s\red{Z}}$ and Karoubi-generates it. 
It remains to show fully faithfulness: this follows from \cite[Theorem 1.2]{ben-zvifrancisnadler10} upon embedding $\Coh$ inside $\QCoh$.
\end{proof}

\sssec{}
We now make explicit the dg-functors that appear in the diagram 
$$
\on{Bar}_n(\sT^+):
(\Delta_{\inj}^{\leq n})
\to 
\dgCat_A.
$$

Let $0\leq m \leq n$ and $k\leq n-m$.
Then there are exactly $\prod_{i=1}^{k}(m+i+1)!$ dg-functors
$$
(\sT^+)^{\op}\otimes_A\sB^{+,\otimes_A m+k} \otimes \sT^+
\to
(\sT^+)^{\op}\otimes_A\sB^{+,\otimes_A m} \otimes \sT^+
$$
which appear in the diagram $\on{Bar}_n(\sT^+)$:
\begin{equation*}
    \begin{tikzcd}
    (\sT^+)^{\op}\otimes_A\sB^{+,\otimes_A m+k} \otimes \sT^+
    \arrow[r,shift right=3]
    \arrow[r, draw=none, "\raisebox{+0.5ex}{\dots}" description]
    \arrow[r,shift left=3]
    &
    \!\!
    \cdots
    \!\!
    \arrow[r,shift right=3]
    \arrow[r, draw=none, "\raisebox{+0.5ex}{\dots}" description]
    \arrow[r,shift left=3]
    &
    (\sT^+)^{\op}\otimes_A\sB^{+,\otimes_A m+1} \otimes \sT^+
    \arrow[r,shift right=3]
    \arrow[r, draw=none, "\raisebox{+0.5ex}{\dots}" description]
    \arrow[r,shift left=3]
    &
    (\sT^+)^{\op}\otimes_A\sB^{+,\otimes_A m} \otimes \sT^+.
    \end{tikzcd}
\end{equation*}
Indeed, for each $\ul d = (d_{m+1},\dots,d_{m+k})\in \prod_{s=m+1}^{m+k}[s]$ we have a composition
$$
(\sT^+)^{\op}\otimes_A\sB^{+,\otimes_A m+k} \otimes \sT^+
\xto{\on{m}_{d_{m+k}}}
\dots
\xto{\on{m}_{d_{m+2}}}
(\sT^+)^{\op}\otimes_A\sB^{+,\otimes_A m+1} \otimes \sT^+
\xto{\on{m}_{d_{m+1}}}
(\sT^+)^{\op}\otimes_A\sB^{+,\otimes_A m} \otimes \sT^+.
$$
Here, for $1\leq l\leq k$ and $d\in [m+l]$, the dg-functor
$$
\underbrace{(\sT^+)^{\op}\otimes_A\sB^{+,\otimes_A m+l} \otimes \sT^+}_{
\simeq 
\;
\Cohext(X_s\times_SG^{\times_Sm+l}\times_SX_s)
}
\xto{\on{m}_{d}}
\underbrace{(\sT^+)^{\op}\otimes_A\sB^{+,\otimes_A m+l-1} \otimes \sT^+}_{
\simeq
\;
\Cohext(X_s\times_SG^{\times_Sm+l-1}\times_SX_s)
}
$$
is given by pull-push along the following correspondence
\begin{equation*}
        \begin{tikzcd}[column sep = 1.7cm]
          \overbrace{X\times_Ss\times_S G^{\times_Sd}\times_s G^{\times_S (m+l-d)}\times_S s\times_S X}^{
          X\times_Ss\times_S G^{\times_Sd-1}\times_S(G\times_sG)\times_S G^{\times_S (m+l-d-1)}\times_S s\times_S X}
          \dar["\id \times \on{can} \times \id"]
          \rar["\id \times \pr{13}\times \id"]
          &
          \overbrace{X\times_Ss \times_S G^{\times_S (m+l-1)}\times_Ss\times_SX}^{
          X\times_Ss\times_S G^{\times_Sd-1}\times_S (G)\times_S G^{\times_S (m+l-d-1)}\times_Ss\times_SX}
          \\
          \underbrace{X\times_Ss\times_S G^{\times_Sd}\times_S G^{\times_S (m+l-d)}\times_S s\times_S X}_{
          X\times_Ss \times_S G^{\times_S (m+l)}\times_Ss\times_SX}.
          &
        \end{tikzcd}
    \end{equation*}

\sssec{}

For each $0\leq m \leq n$ we construct a dg-functor
$$
\Xi^+_{n,m}:
  (\sT^+)^{\op}\otimes_A \sB^{+,\otimes_Am} \otimes_A \sT^+
    \to
  \Coh (\bbA^1_{X_s(n+2)}[-1])_{\red{Z}}
$$
as follows.
Consider the commutative diagram
\begin{equation*}
    \begin{tikzcd}[column sep = 2cm]
      X\times_Ss^{\times_Sm+1}
      \rar["\id_X \times \pr{1}"]
      \arrow[dd, swap, bend right = 3.5cm, "\tdelta"]
      \arrow[d, "\delta_{X/S}\times \delta_{s/S}\times \id_{s^{\times_Sm}}"]
        &
      X\times_Ss
      \rar["j(n+2)"]
      \dar["\delta_{X/S}\times \delta_{s/S}"]
        &
      X\times_Ss(n+2)
      \\
      \overbrace{X_s\times_S G^{\times_s m} \times_s X_s}^{
        \simeq
      X\times_S X\times_Ss^{\times_S m+2}}
      \dar["\can"]
      \rar["\id_{X\times_SX}\times \pr{1,2} "]
        &
      \overbrace{X_s\times_SX_s}^{
        \simeq
      X\times_SX\times_S s\times_Ss}
        &
      \\
      X_s\times_S G^{\times_S m} \times_S X_s
        &
        &
    \end{tikzcd}
\end{equation*}
with cartesian square.
Notice that the morphism $\delta_{X/S}\times \delta_{s/S}$ can be decomposed as
$$
  X\times_Ss
    \xto{\id_X \times \delta_{s/S}}
  X\times_S s\times_Ss
    \xto{\delta_{X/S} \times \id}
  X\times_SX \times_S s\times_Ss.
$$
Then, it is easy to see that the dg-functor
$$
  \bigt{\id_X \times_S \bigt{j(n+2)\circ \pr{1}}_*} \circ \tdelta^*:
  \Cohext(X_s\times_S G^{\times_Sm}\times_S X_s)_{\red{Z}\times_s \red{Z}}
    \to
  \Cohminus(X_s(n+2))
$$
enhances to
$$
  \Cohext(X_s\times_S G^{\times_Sm}\times_S X_s)_{\red{Z}\times_s \red{Z}}
    \to
  \Mod_{\ccO_{X_s(n+2)}[u]}^{\etp}\bigt{\Cohminus(X_s(n+2))}.
$$

By Lemma \ref{lem: Coh(A1[-1]) vs Cohminus etp}, we obtain a dg-functor
$$
\Xi^{+}_{n,m}:
  (\sT^+)^{\op}\otimes_A (\sB^+)^{\otimes_A m}\otimes_A \sT^+
    \simeq
  \Cohext(X_s\times_S G^{\times_Sm}\times_S X_s)_{\red{Z}\times_s \red{Z}}
    \to
  \Coh(\bbA^1_{X_s(n+2)}[-1])_{\red{Z}}.
$$

\begin{lem}\label{lem: functor Barn(T)--> Coh(A^1[-1])}

For each $n \geq 0$, the dg-functors $\{\Xi^+_{n,m}\}_{0\leq m \leq n}$ define a natural transformation
\begin{equation}\label{eqn: ol Xi_n^+}
  \Xi_n^+:\on{Bar}_n(\sT^+) 
    \to 
  \Coh \bigt{\bbA^1_{X_s(n+2)}[-1]}_{\red{Z}}
\end{equation}
of functors $(\Delta^{\leq n}_{\on{inj}})^{\op}\to \dgCat_A$.
These natural transformations are functorial in $n$.
\end{lem}

\begin{proof}

The proof follows the same lines as Lemma \ref{lem:defn of xi-n-plus}, where now the main geometric input is provided by Theorem \ref{thm: master homotopies - variant}.
\end{proof}

\sssec{}
Notice that the two lemmas above immediately give us the desired dg-functors
$$
  \Xi^+_n: 
  F_n \bigt{(\sT^+)^{\op}\otimes_{\sB^+}\sT^+}
    \to
  \Coh \bigt{\bbA^1_{X_s(n+2)}[-1]}_{\red{Z}}
$$
which are moreover functorial in $n$.
Thus, passing to the colimit over $\bbN$ yields the desired functor \eqref{eqn: from TBTplus to colim Coh(A^1[-1])}.

\subsection{A functor out of the Thom--Sebastiani category}

Here we adapt the contents of Section \ref{ssec:functor from TBT-plus} to the singularity categories. 
Recall the definition of the dg-category of coherent matrix factorization, see Section \ref{sssec: defn MFcoh}.
The goal is to write down a dg-functor
\begin{equation} \label{eqn:from TBT to Singext-weird}
  \Xi:
  \sT^\op \otimes_{\sB} \sT 
    \to 
  \colim_{\bbN} 
  \MFcoh \bigt{X_s(n+1),0}_{\red{Z}}
  .
\end{equation}
The treatment is parallel to the one of Section \ref{ssec:functor from TBT-plus}.

\sssec{} 

The canonically defined dg-functor 
$$
  \on{can}:
  \varinjlim_{\bbN} 
  F_n \bigt{\sT^\op \otimes_{\sB}\sT}
    \to 
  \sT^\op \otimes_{\sB}\sT
$$
is an equivalence in $\dgCat_A$. Hence, as before, our dg-functor \eqref{eqn:from TBT to Singext-weird} will be the composition 
$$
  \sT^\op \otimes_{\sB}\sT
    \xto{\on{can}^{-1}}
  \varinjlim_n 
  F_n \bigt{\sT^\op \otimes_{\sB}\sT}
    \xto{\colim_n \Xi_n}
  \colim_n 
  \MFcoh \bigt{X_s(n+1),0}_{\red{Z}}.
$$
The definition of the $\Xi_n$'s is obtained from that of the $\Xi_n^+$'s. 
We first need the following

\begin{lem}
For $n \geq 0$, we have
$$
  \sT^\op \otimes_A \sB^{\otimes_An}\otimes_A \sT
    \simeq 
  \Singext(X_s \times_S G^{\times_S n}\times_S X_s)_{\red{Z}\times_s\red{Z}},
$$
with the convention that $G^{\times_S 0 } = S$.
\end{lem}

\begin{proof}
Thanks to \cite[Appendix A]{toenvezzosi22}, we know that  
$$
  \sT^\op \otimes_A \sB^{\otimes_An}\otimes_A \sT
$$
is the dg-quotient of 
$$
  (\sT^+)^\op \otimes_A (\sB^+)^{\otimes_An}\otimes_A \sT^+
$$
by its full subcategory Karoubi-generated by
$$
  \bigt{(\sT^+_{\on{pe}})^\op \otimes_A (\sB^+)^{\otimes_An}\otimes_A \sT^+}
    \cup 
  \bigt{(\sT^+)^\op \otimes_A\sB^+_{\on{pe}}\otimes_A (\sB^+)^{\otimes_An-1}\otimes_A \sT^+}
    \cup 
  \dots
$$
$$
  \dots 
    \cup  
  \bigt{(\sT^+)^\op \otimes_A (\sB^+)^{\otimes_An-1}\otimes_A\sB^+_{\on{pe}}\otimes_A \sT^+}
    \cup 
  \bigt{(\sT^+)^\op \otimes_A (\sB^+)^{\otimes_An}\otimes_A \sT^+_{\on{pe}}},
$$
It is immediate to observe that, under the equivalence of Lemma \ref{lem: description T+otimes B+otimesn otimes T+}, this subcategory corresponds to 
$$
  \Perfext(X_s \times_S G^{\times_Sn}\times_SX_s)_{\red{Z}\times_s\red{Z}}
    \subseteq 
  \Cohext(X_s \times_S G^{\times_Sn}\times_SX_s)_{\red{Z}\times_s\red{Z}}.
$$
\end{proof}

\sssec{}

To construct
\begin{equation}\label{eqn: Xi_n}
\nonumber
    \Xi_n: 
    F_n \bigt{\sT^\op \otimes_{\sB}\sT} 
      \to 
    \MFcoh \bigt{X_s(n+2),0}_{\red{Z}}
\end{equation}
we just need to verify that each functor 
$$
  \Cohext(X_s \times_S G^{\times_S m}\times_S X_s)_{\red{Z}\times_s\red{Z}}
    \to
  \Coh \bigt{\bbA^1_{X_s(n+2)}[-1]}_{\red{Z}}
$$
as in Lemma \ref{lem: functor Barn(T)--> Coh(A^1[-1])} induces a dg-functor
$$
  \Perfext(X_s \times_S G^{\times_S m}\times_S X_s)_{\red{Z}\times_s\red{Z}}
    \to
  \Angles{\pr{}^*\Coh \bigt{X_s(n+2)}_{\red{Z}}}.
$$
By definition, each functor
$$
  \Perfext(X_s \times_S G^{\times_S m}\times_S X_s)_{\red{Z}\times_s\red{Z}}
    \to
  \Coh \bigt{\bbA^1_{X_s(n+2)}[-1]}_{\red{Z}}
$$
can be written as the composition below 
$$
  \Perfext(X_s \times_S G^{\times_S m}\times_S X_s)_{\red{Z}\times_s\red{Z}}
    \to
  \Perfext(X_s \times_S X_s)_{\red{Z}\times_s\red{Z}}
    \to
  \Coh \bigt{\bbA^1_{X_s(n+2)}[-1]}_{\red{Z}}.
$$
Therefore, we just need to verify that the latter functor factors through
$$
  \Perfext(X_s  \times_S X_s)_{\red{Z}\times_s\red{Z}}
    \to
  \Angles{\pr{}^*\Coh \bigt{X_s(n+2)}_{\red{Z}}}.
$$
This follows from the definitions, from base change along the cartesian square
\begin{equation*}
    \begin{tikzcd}
      X_s
      \rar
      \dar
        &
      X_s(n+2)
      \dar
        \\
      X_s\times_SX_s
      \rar
        &
      \bbA^1_{X_s(n+2)}[-1]
    \end{tikzcd}
\end{equation*}
and from Lemma \ref{lem: pr^* vs etp u-tors}.

\begin{rmk}
Notice that the functors $\Xi_n$ above are clearly functorial in $n$:
for each $n\geq 1$ it follows from the constructions that the square below commutes
\begin{equation*}
    \begin{tikzcd}
      \Bar_n(\sT) \simeq \restr{\Bar_{n+1}(\sT)}{(\Delta_{\on{inj}}^{\leq n})^{\op}}
      \arrow[r]
      \arrow[d]
      &
      \MFcoh\bigt{X_s(n+2),0}_{\red{Z}}
      \arrow[d]
      \\
      \Bar_{n+1}(\sT)
      \arrow[r]
      &
      \MFcoh\bigt{X_s(n+3),0}_{\red{Z}}.
    \end{tikzcd}
\end{equation*}
In particular, by passing to the colimits of these diagrams, we get a functor
$$
  (\bbN, \leq)
    \to
  \Fun \bigt{\Delta^1,\dgCat_A}.
$$
Loosely, this sends the morphism $n \leq n+m$ to the commutative square
\begin{equation*}
    \begin{tikzcd}
      F_n \bigt{\sT^{\op}\otimes_{\sB}\sT}
      \arrow[r,"\Xi_n"]
      \arrow[d]
      &
      \MFcoh\bigt{X_s(n+2),0}_{\red{Z}}
      \arrow[d]
      \\
      F_{n+m} \bigt{\sT^{\op}\otimes_{\sB}\sT}
      \arrow[r,"\Xi_{n+m}"]
      &
      \MFcoh\bigt{X_s(n+m+2),0}_{\red{Z}}.
    \end{tikzcd}
\end{equation*}
Therefore, passing to the colimit over $\bbN$ we get the desired functor \eqref{eqn:from TBT to Singext-weird}.
\end{rmk}

\subsection{The integration dg-functor}

In this section, we begin the study of the \emph{integration dg-functor} $\int_{X/S}$ defined as the composition
$$
  \int_{X/S}:
  \Sing(X\times_SX)
    \simeq 
  \sT^\op \otimes_{\sB}\sT
    \xto{\;\;\ev^{\HH}\;}
  \HH(\sB/A)
    \xto{\oint} 
  \Sing (\bbA^1_S[-1])_s.
$$
Upon taking realizations, this improves the categorical intersection product introduced by \TV{} in \cite{toenvezzosi22}. 
We will exploit this later to recover Bloch's intersection number.

\sssec{} 

The concrete goal of this section is to rewrite $\int_{X/S}$ so as to avoid the passage through $\HH(\sB/A)$. We will do so as follows: recalling the definition of $\oint$, the dg-functor $\int_{X/S}$ can be written as the composition
\begin{equation*}
\begin{tikzcd}
    \sT^\op\otimes_{\sB}\sT
    \arrow[r,"\ev^{\HH}"]
    &
    \HH(\sB/A)
    \arrow[r,"\xi"]
    \arrow[rr,bend right = 0.5cm,swap,"\oint"]
    &
    \sB(\infty)
    \arrow[r] 
    &
    \Sing(\bbA^1_S[-1])_s.
\end{tikzcd}
\end{equation*}
We will construct a commutative diagram
\begin{equation}\label{eqn: rewriting integration functor part I}
\begin{tikzpicture}[scale=1.5]
\node (00) at (0,0) {$\HH(\sB/A) $};
\node (10) at (4,0) {$\sB(\infty)$};
\node (01) at (0,1) {$
\sT^{\op} \otimes_{\sB} \sT$};
\node (11) at (4,1) {$
\colim_{\bbN} \MFcoh \bigt{X_s(n+1),0}_{\red{Z}}$}; 
\path[->,font=\scriptsize,>=angle 90]
(00.east) edge node[above] {$\xi$}  (10.west); 
\path[->,font=\scriptsize,>=angle 90]
(01.east) edge node[above] {$\Xi$} (11.west); 
\path[->,font=\scriptsize,>=angle 90]
(01.south) edge node[right] {$\ev^{\HH}$} (00.north);
\path[->,font=\scriptsize,>=angle 90]
(11.south) edge node[right] {$ $} (10.north);
\end{tikzpicture}
\end{equation}
where the right vertical map is induced by pushforward along the morphisms
$$
  \bbA^1_{X_s(n)}[-1]
    \to
  \bbA^1_{s(n)}[-1].  
$$

\sssec{} 

As usual, before dealing with the singularity categories, we focus on the coherent categories. 

\sssec{}\label{sssec: defn evHH}

We begin by recalling the definition of the left vertical arrow in the above diagram. 
Since $\sT^+$ is a proper left $\sB^+$-module, there is a $\sB^+$-bilinear evaluation
$$
  \ev: 
  \sT^+\otimes_A(\sT^+)^{\op} 
    \to
  \sB^+.
$$
Following \cite[Section 5.2]{toenvezzosi22}, one considers the dg-functor
$$
  \ev^{\HH,+}: (\sT^+)^{\op}\otimes_{\sB^+}\sT^+ 
    = 
  \bigt{
        \sT^+ \otimes_A(\sT^+)^{\op}
        }
        \otimes_{\sB^{+,\env}}\sB^+ 
          \xto{
               \ev \otimes_{\sB^{+,\env}}\sB^+
               } 
        \sB^+ \otimes_{\sB^{+,\env}}\sB^+
          =
        \HH(\sB^+/A).
$$
A similar construction for $\sT$ and $\sB$ yields the dg-functor 
$$
  \ev^{\HH}:
  \Sing(X\times_S X)
    \simeq 
  \sT^{\op}\otimes_{\sB}\sT 
    = 
  \bigt{\sT\otimes_A\sT^{\op}}\otimes_{\sB^{\env}}\sB  
    \xto{\ev \otimes_{\sB^{\env}}\sB} 
  \sB \otimes_{\sB^{\env}}\sB
    =
  \HH(\sB/A).
$$

\sssec{} \label{sssec:evHH as natural trans}

By construction, $\ev^{\HH,+}$ is induced by the natural transformation 
$$
  \on{Bar}(\sT^+)
    \to 
  \on{Bar}(\sB^+/A): 
  (\Delta_{\on{inj}})^{\op}
    \to 
  \dgCat_A
$$
given value-wise by the dg-functor
$$
  \ev^{\HH,+}_n:
  (\sT^+)^\op \otimes_A (\sB^+)^{\otimes_An}\otimes_A \sT^+
    \xto{\on{swap}}
  (\sT^+)^\op \otimes_A \sT^+ \otimes_A (\sB^+)^{\otimes_An}
    \xto{\ev^+ \otimes \id^{\otimes n}}
  (\sB^+)^{\otimes_A n+1}.
$$
The dg-functor $\ev_n^{\HH}$ is defined analogously. Here are geometric descriptions of $\ev^{\HH,+}_n$ and $\ev_n^{\HH}$.

\begin{lem}
Under the equivalences
$$
  (\sT^+)^\op \otimes_A (\sB^+)^{\otimes_An}\otimes_A \sT^+
    \simeq 
  \Cohext(X_s \times_S G^{\times_Sn}\times_S X_s)_{\red{Z}\times_s\red{Z}}
$$
$$
  (\sB^+)^{\otimes_A n+1}
    \simeq 
  \Cohext(G^{\times_Sn+1}),
$$
the dg-functor 
$$
  \ev^{\HH,+}_n
    :
  (\sT^+)^\op \otimes_A (\sB^+)^{\otimes_An}\otimes_A \sT^+
    \to
  (\sB^+)^{\otimes_A n+1}
$$
is induced by pull-push along the natural correspondence
$$
  X_s\times_S G^{\times_Sn}\times_S X_s 
    \xleftarrow{\on{swap}}
  (X_s\times_S X_s) \times_S G^{\times_Sn}
    \leftarrow
  (X_s\times_X X_s) \times G^{\times_Sn}
    \to
  G\times_S G^{\times_Sn}=G^{\times_S n+1}.
$$
\end{lem}

\begin{proof}
This readily follows from the fact that 
$$
  \ev^+:
  (\sT^+)^\op \otimes_A \sT^+
    \to
  \sB^+
$$
$$
  (E,F)
    \mapsto 
  (X_s\times_X X_s\to G)_*(E^{\vee}\boxtimes_XF)
$$
corresponds to pull-push along the correspondence
$
  X_s\times_S X_s 
    \leftarrow 
  X_s\times_X X_s 
  \to s\times_Ss
    =
  G.
$
\end{proof}

\begin{cor}
Under the equivalences
$$
  \sT^\op \otimes_A \sB^{\otimes_An}\otimes_A \sT
    \simeq 
  \Singext(X_s \times_S G^{\times_Sn}\times_S X_s)_{\red{Z}\times_s\red{Z}}
$$
$$
  \sB^{\otimes_A n+1}
    \simeq 
  \Singext(G^{\times_Sn+1}),
$$
the dg-functor 
$$
  \ev^{\HH}_n:
  \sT^\op \otimes_A \sB^{\otimes_An}\otimes_A \sT
    \to
  \sB^{\otimes_A n+1}
$$
is induced by pull-push along the natural correspondence
$$
  X_s\times_S G^{\times_Sn}\times_S X_s 
    \xleftarrow{\on{swap}}
  (X_s\times_S X_s) \times_S G^{\times_Sn}
    \leftarrow
  (X_s\times_X X_s) \times G^{\times_Sn}
    \to
  G\times_S G^{\times_Sn}=G^{\times_S n+1}.
$$
\end{cor}

\sssec{}

Recall from Section \ref{sssec:evHH as natural trans} that $\ev^{\HH}$ arises as the colimit $\colim_{\bbN} \ev^{\HH}_{\leq n}$ of the naturally defined dg-functors
$$
  \ev^{\HH}_n: 
  F_n \bigt{\sT^\op \otimes_{\sB}\sT } 
    \to 
  F_n \HH(\sB/A).
$$
It follows that $\int_{X/S}$ is the colimit of the dg-functors
$$
  \int_{X/S,n}:
  F_n(\sT^\op\otimes_{\sB}\sT)
    \xto{\ev^{\HH}_{n}}
  F_n\bigt{\HH(\sB/A)}
    \xto{\xi_n}
  \sB(n+2)
    \xto{r(n+2)_*} 
  \Sing(\bbA^1_S[-1])_s.
$$
We now describe the composition of the first two arrows in a different way.

\sssec{}

Consider the dg-functor
\begin{equation} \label{eqn:from Cohext to Bplus}
  \Coh \bigt{\bbA^1_{X_s(n+2)}[-1]}_{\red{Z}} 
    \to 
  \sB^+(n+2)    
    =
  \Coh (\bbA^1_{s(n+2)}[-1])
\end{equation}
induced by pushforward along 
$
  \bbA^1_{X_s(n+2)}
    \to
  \bbA^1_{s(n+2)}.
$

\begin{lem}
The diagram
\begin{equation} 
\nonumber
\begin{tikzpicture}[scale=1.5]
\node (00) at (0,0) {$F_n \bigt{ \HH(\sB^+/A) } $};
\node (10) at (4,0) {$\sB^+(n+2) $ };
\node (01) at (0,1) {$F_n \bigt{
(\sT^+)^{\op} \otimes_{\sB^+} \sT^+
}$};
\node (11) at (4,1) {$\Coh \bigt{\bbA^1_{X_s(n+2)}[-1]}_{\red{Z}}$}; 
\path[->,font=\scriptsize,>=angle 90]
(00.east) edge node[above] {$\xi_n^+$}  (10.west); 
\path[->,font=\scriptsize,>=angle 90]
(01.east) edge node[above] {$\Xi_n^+ $} (11.west); 
\path[->,font=\scriptsize,>=angle 90]
(01.south) edge node[right] {$\ev^{\HH,+}_{n}$} (00.north);
\path[->,font=\scriptsize,>=angle 90]
(11.south) edge node[right] {$ $} (10.north);
\end{tikzpicture}
\end{equation}
is commutative.
\end{lem}

\begin{proof}
Since this square is induced by natural transformations of diagrams $(\Delta_{\on{inj}})^\op \to \dgCat_A$, it suffices to verify that for each $0\leq m \leq n$ the square
\begin{equation*}
    \begin{tikzcd}
      (\sT^+)^{\op}\otimes_A(\sB^+)^{\otimes_Am}\otimes_A\sT^+ 
      \arrow[d] 
      \arrow[r,"\Xi^+_{n,m}"]
      &
      \Coh \bigt{\bbA^1_{X_s(n+2)}[-1]}_{\red{Z}}
      \arrow[d] 
      \\
      (\sB^+)^{\otimes_A m+1} 
      \arrow[r,"\xi^+_{n,m}"]
      & 
      \sB^+(n+2)
    \end{tikzcd}
\end{equation*}
commutes.
Under the equivalences of Lemmas \ref{lem: nodes Bar(B^+/A)} and \ref{lem: description T+otimes B+otimesn otimes T+}, the horizontal arrows correspond to
\begin{align*}
  \Cohext(G^{\times_S m+1})
    &
    \to
  \Coh(G)
    \\
    &
    \xto{\delta_{s/S}^*}
  \Mod_{\ccO_{s}[u]}^{\etp}\bigt{\Cohminus(s)}
    \\
    &
    \xto{j(n+2)_*}
  \Mod_{\ccO_{s(n+2)}[u]}^{\etp}\bigt{\Cohminus(s(n+2))}
    \\
    &
    \simeq 
  \Coh \bigt{\bbA^1_{s(n+2)}[-1]},    
\end{align*}
\begin{align*}
  \Cohext(X_s\times_S G^{\times_S m}\times_S X_s)_{\red{Z}\times_s\red{Z}}
    &
    \to
  \Cohext(X_s\times_SX_s)_{\red{Z}\times_s \red{Z}}
    \\
    &
    \xto{\delta_{X_s/S}^*}
  \Mod_{\ccO_{X_s}[u]}^{\etp}\bigt{\Cohminus(X_s)}_{\red{Z}}
    \\
    &
    \xto{j(n+2)_*}
  \Mod_{\ccO_{X_s(n+2)}[u]}^{\etp}\bigt{\Cohminus(X_s(n+2))}_{\red{Z}}
    \\
    &
    \simeq 
  \Coh \bigt{\bbA^1_{X_s(n+2)}[-1]}_{\red{Z}}.
\end{align*}
The claim then follows from base-change along the cartesian squares
\begin{equation*}
    \begin{tikzcd}
      X_s(n+2)
      \dar
      &
      \lar["j(n+2)"]
      X_s
      \rar["\delta_{X_s/S}"]
      \dar
        &
      X_s\times_SX_s
      \dar
      \\
      s(n+2)
      &
      \lar["j(n+2)"]
      s
      \rar["\delta_{s/S}"]
        &
      G.
    \end{tikzcd}
\end{equation*}
\end{proof}

\begin{cor}
The dg-functor defined in \eqref{eqn:from Cohext to Bplus} induces an analogous dg-functor
\begin{equation} \label{eqn:fromSingext-weird to B}
  \MFcoh(X_{s(n+2)},0)_{\red{Z}}
    \to 
  \sB(n+2)
\end{equation}
that fits in the commutative diagram
\begin{equation} 
\nonumber
\begin{tikzpicture}[scale=1.5]
\node (00) at (0,0) {$F_n \bigt{ \HH(\sB/A) } $};
\node (10) at (4,0) {$\sB(n+2)$. };
\node (01) at (0,1) {$F_n \bigt{
\sT^{\op} \otimes_{\sB} \sT
}$};
\node (11) at (4,1) {$\MFcoh(X_{s(n+2)},0)_{\red{Z}}$}; 
\path[->,font=\scriptsize,>=angle 90]
(00.east) edge node[above] {$\xi_n$}  (10.west); 
\path[->,font=\scriptsize,>=angle 90]
(01.east) edge node[above] {$\Xi_n $} (11.west); 
\path[->,font=\scriptsize,>=angle 90]
(01.south) edge node[right] {$\ev^{\HH}_{n}$} (00.north);
\path[->,font=\scriptsize,>=angle 90]
(11.south) edge node[right] {${}$} (10.north);
\end{tikzpicture}
\end{equation}

Taking the colimit, we obtain the commutative diagram \eqref{eqn: rewriting integration functor part I}
that we were looking for.
\end{cor}

\begin{proof}
We just need to verify that our dg-functor 
$
  \Coh \bigt{\bbA^1_{X_s(n+2)}[-1]}_{\red{Z}}
    \to 
  \sB^+(n+2)
$ 
induces a dg-functor
$$
  \MFcoh (X_{s(n+2)},0)_{\red{Z}}
    \to 
  \sB(n+2)
$$
Thus, we need to check that the dg-functor restricted to 
$
  \Angles{\pr{}^*\Coh(X_{s(n+2)})_{\red{Z}}}
    \subseteq 
  \Coh \bigt{\bbA^1_{X_s(n+2)}[-1]}_{\red{Z}}
$ 
factors through $\sB^+_{\on{pe}}(n+1)$.
This follows immediately base change along
\begin{equation*}
    \begin{tikzcd}
      \bbA^1_{X_s(n+2)}[-1]
      \rar["\pr{}"]
      \dar
        &
      X_{s(n+2)}
      \dar
      \\
      \bbA^1_{s(n+2)}[-1]
      \rar["\pr{}"]
        &
      s(n+2).
    \end{tikzcd}
\end{equation*}
\end{proof}

\sssec{}

So far, we have constructed the left  commutative square in diagram \eqref{diagram: alternative expressions integration functor}.
To construct the right half of the diagram, it suffices to notice that there is a commutative diagram
\begin{equation*}
    \begin{tikzcd}
      X_s\times_SX_s
      \rar
      \dar
        &
      \bbA^1_{X_s(2)}[-1]
      \rar
      \dar
        &
      \cdots
      \rar
        &
      \bbA^1_{X_s(n)}[-1]
      \dar
      \rar
      \arrow[rr,bend left = 0.5cm,"R(n)"]
        &
      \cdots
        &
      \bbA^1_{X}[-1]
      \dar
      \\
      G
      \rar
        &
      \bbA^1_{s(2)}[-1]
      \rar
        &
      \cdots
      \rar
        &
      \bbA^1_{s(n)}[-1]
      \arrow[rr,swap,bend right = 0.5cm,"r(n)"]
      \rar
        &
      \cdots
        &
      \bbA^1_{S}[-1].
    \end{tikzcd}
\end{equation*}
Passing to the colimits, we get the following commutative squares in $\dgCat_A$:
\begin{equation}\label{eqn: R(oo) vs r(oo)}
    \begin{tikzcd}
      \varinjlim_n \Coh(\bbA^1_{X_s(n)}[-1])_{\red{Z}}
      \rar["R(\oo)_*"]
      \dar
      &
      \Coh(\bbA^1_{X}[-1])_{\red{Z}}
      \dar
      \\
      \varinjlim_n \Coh(\bbA^1_{s(n)}[-1])
      \rar["r(\oo)_*"]
      &
      \Coh(\bbA^1_{S}[-1])_{s}
    \end{tikzcd}
    \;\;\;\;
    \begin{tikzcd}
      \varinjlim_n \MFcoh(X_s(n),0)_{\red{Z}}
      \rar["R(\oo)_*"]
      \dar
      &
      \MFcoh(X,0)_{\red{Z}}
      \dar
      \\
      \varinjlim_n \MFcoh(s(n),0)
      \rar["r(\oo)_*"]
      &
      \MFcoh(S,0)_s.
    \end{tikzcd}
\end{equation}

\begin{rmk}\label{rmk: Upsilon^+ vs delta^*}
Denote by $t:X\to \bbA^1_X[-1]$ the closed embedding of the classical truncation. 
The composition 
\begin{equation*}
  \begin{tikzcd}
    \Coh(X\times_SX)_{\red{Z}\times_s\red{Z}}
      \arrow[rr,bend left = 0.5cm,"\Upsilon^+"]
      \rar["\Xi^+"]
      &
    \varinjlim_n \Coh(\bbA^1_{X_s(n)}[-1])_{\red{Z}}
      \rar["R(\oo)_*"]
      &
    \Coh(\bbA^1_X[-1])_{\red{Z}}
      \rar["t^*"]
      &
    \QCoh(X)_{\red{Z}}
  \end{tikzcd}
\end{equation*}
is equivalent to the pullback
$$
  \delta_{X/S}^*: \Coh(X\times_SX)_{\red{Z}\times_s\red{Z}}
    \to 
  \QCoh(X)_{\red{Z}}
$$
along the diagonal $\delta_{X/S}$.
This follows immediately from the definition of $\Xi^+$, from the sequence of cartesian squares \eqref{eqn: R(oo) vs r(oo)} and from the observation that
$$
  t^*:
  \Coh(\bbA^1_X[-1])_{\red{Z}}
    \to
  \QCoh(X)_{\red{Z}}
$$
corresponds, under the equivalence 
$
  \Coh(\bbA^1_X[-1])_{\red{Z}}
    \simeq
  \Mod_{\ccO_X[u]}^{\etp}\bigt{\Cohminus(X)_{\red{Z}}},
$
to the forgetful functor
$$
  \Mod_{\ccO_X[u]}^{\etp}\bigt{\Cohminus(X)_{\red{Z}}}
    \to
  \Cohminus(X)_{\red{Z}}
    \subseteq
  \QCoh(X)_{\red{Z}}.
$$
\end{rmk}


\section{Proof of the generalized Bloch conductor conjecture}\label{sec: proof of gBCC}

In this section, we explain how the categorical intersection product defined by {\TV} (\cite[Definition 5.2.1]{toenvezzosi22}) relates to the localized intersection product defined by Kato--Saito in \cite[Definition 5.1.5]{katosaito04}.
This, combined with our earlier work \cite{beraldopippi24}, leads to the proof of the generalized Bloch conductor conjecture (Conjecture \ref{conj:gBCC}).

\subsection{A commutative diagram}

In this section, we construct the commutative diagram below:
\begin{equation}\label{diag: key diag}
    \begin{tikzcd}[column sep = 1.5cm]
      \Mod^{\etp}_{\ccO_X[u]}\bigt{\Cohminus(X)_{\eZ}}
      \arrow[r,"\on{forget}"]
      \arrow[d]
      &
      \Cohminus(X)_{\eZ}^{\Getp}
      \arrow[d]
      \\
      \frac{\Mod^{\etp}_{\ccO_X[u]}\bigt{\Cohminus(X)_{\eZ}}}{\Mod_{\ccO_X[u]}\bigt{\Coh(X)}}
      \arrow[r,"\on{forget}"]
      &
\frac{\Cohminus(X)_{\eZ}^{\Getp}}{\Coh(X)}.
    \end{tikzcd}
\end{equation}
We now proceed to define all the categories and functors involved.

\sssec{}

We say that an object $E \in \Cohminus(X)$ is \emph{eventually supported on $\red{Z}$} if there exists an integer $n_0 \in \bbN$ such that $\ccH^{-n}(E)$ is a coherent $\ccO_X$-module set theoretically supported on $Z$ for all $n \geq n_0$.
Denote by  $\Cohminus(X)_{\eZ}$ the full subcategory of objects that are eventually supported on $\red{Z}$.
This is an idempotent-complete full subcategory that is closed under finite limits (and finite colimits).
There is an obvious fully-faithful embedding $\Coh(X)\subseteq \Cohminus(X)_{\eZ}$.

\sssec{} 

In particular, consider $\Mod_{\ccO_X[u]}\bigt{\Cohminus(X)_{\eZ}}$
and its full subcategory
$$
  \Mod^{\etp}_{\ccO_X[u]}\bigt{\Cohminus(X)_{\eZ}}
    \subseteq 
  \Mod_{\ccO_X[u]}\bigt{\Cohminus(X)_{\eZ}}
$$
spanned by those objects that are eventually two-periodic and eventually supported on $\red{Z}$.
There is an obvious fully faithful embedding
$$
 \Mod_{\ccO_X[u]}\bigt{\Coh(X)}
 \simeq 
  \Mod^{\etp}_{\ccO_X[u]}\bigt{\Coh(X)}
    \subseteq
  \Mod^{\etp}_{\ccO_X[u]}\bigt{\Cohminus(X)_{\eZ}}.
$$

\sssec{}

We denote by $\Cohminus(X)_{\eZ}^{\Getp}$ the full subcategory of $\Cohminus(X)_{\eZ}$ spanned by those objects $M \in \Cohminus(X)_{\eZ}$ with the following property: there exists $n_0 \in \bbN$ such that
$$
  [\ccH^{-n}(M)]=[\ccH^{-n-2}(M)] 
    \in 
  \G_0(X)_{\red{Z}}
$$
for all $n \geq n_0$.
There is an inclusion 
$$
  \Coh(X)
    \subseteq 
  \Cohminus(X)^{\Getp}_{\eZ}
$$ and, obviously, the forgetful functor $\Mod_{\ccO_X[u]}\to \QCoh(X)$ restricts to
\begin{equation}\label{eqn: Mod^etp_eZ --> Cohminus^Getp_eZ}
  \Mod^{\etp}_{\ccO_X[u]}\bigt{\Cohminus(X)_{\eZ}}
    \xto{\on{forget}}
  \Cohminus(X)_{\eZ}^{\Getp}.
\end{equation}

\sssec{}

In \eqref{diag: key diag}, the two vertical arrows are just the canonical quotient dg-functors and the bottom horizontal dg-functor is induced by the top one upon passing to the quotients.

\subsection{Rewriting the forgetful functor} 
Our next goal is to rewrite the dg-functor
$$
\frac{\Mod^{\etp}_{\ccO_X[u]}\bigt{\Cohminus(X)_{\eZ}}}{\Mod_{\ccO_X[u]}\bigt{\Coh(X)}}
\xto{\on{forget}}
\frac{\Cohminus(X)_{\eZ}^{\Getp}}{\Coh(X)}
$$
that appears in \eqref{diag: key diag}. Namely, we will place it in a commutative diagram as follows:
\begin{equation}\label{diag: key diag2}
    \begin{tikzcd}
\frac{\Mod^{\etp}_{\ccO_X[u]}\bigt{\Cohminus(X)_{\eZ}}}{\Mod_{\ccO_X[u]}\bigt{\Coh(X)}}
      \arrow[r,"\on{forget}"]
      \arrow[d,"\ffL"]
      &
      \frac{\Cohminus(X)_{\eZ}^{\Getp}}{\Coh(X)}
      \\
      \Mod_{\ccO_X[u,u^{-1}]}\bigt{\Cohinfty(X)_{\red{Z}}}
      \arrow[r,"\on{forget}"]
      &
      \Cohinfty(X)_{\red{Z}}^{\tp}.
      \arrow[u,"\ffT"]
    \end{tikzcd}
\end{equation}

\sssec{}
We begin by defining the bottom right corner.
Denote by $\Cohinfty(X)_{\red{Z}}$ the full subcategory of $\QCoh(X)$ spanned by those complexes of $\ccO_X$-modules $M$ such that, for any $n \in \bbZ$, the cohomology sheaf $\ccH^n(M)$ is a coherent $\ccO_X$-module set theoretically supported on $\red{Z}$. 
Let $\Cohinfty(X)^{\tp}_{\red{Z}}$ the full subcategory of $\Cohinfty(X)_{\red{Z}}$ consisting of objects $M$ with two-periodic (\emph{tp}) cohomology sheaves: for every $n \in \bbZ$, there is an isomorphism
$$
\ccH^n(M)\simeq \ccH^{n-2}(M).
$$
We stress that these isomorphisms are not part of the data.

\sssec{} 
We now look at the bottom left corner.
Let $\ccO_{X}[u,u^{-1}]$ be the obvious localization of $\ccO_X[u]$, that is,
$$
  \ccO_{X}[u,u^{-1}]
    :=
  \varinjlim \bigt{\ccO_X[u]\xto{u}\ccO_X[u][2]\xto{u} \cdots}
$$
and $\Mod_{\ccO_X[u,u^{-1}]}$ the dg-category of $\ccO_X[u,u^{-1}]$-modules in $\QCoh(X)$.
There is an adjunction
$$
  \ccO_X[u,u^{-1}] \otimes_{\ccO_X} -:
  \QCoh(X)
    \leftrightarrows
  \Mod_{\ccO_X[u,u^{-1}]}
  :
  \on{forget}.
$$
As before, we set
$$
  \Mod_{\ccO_X[u,u^{-1}]}\bigt{\Cohinfty(X)_{\red{Z}}}
    :=
  \varprojlim \Bigt{\Mod_{\ccO_X[u,u^{-1}]}\xto{\on{forget}}\QCoh(X)\hookleftarrow \Cohinfty(X)_{\red{Z}}}.
$$
\sssec{} 

For $M \in \Mod_{\ccO_X[u,u^{-1}]}\bigt{\Cohinfty(X)_{\red{Z}}}$, the arrow $u: M \to M[2]$ obviously induces an isomorphism
$$
\ccH^n(M)\simeq \ccH^{n+2}(M)
$$
of coherent $\ccO_X$-modules (supported on $Z$).
This yields the bottom arrow 
$$
  \on{forget}:
  \Mod_{\ccO_X[u,u^{-1}]}\bigt{\Cohinfty(X)_{\red{Z}}}
    \to
  \Cohinfty(X)^{\tp}_{\red{Z}}
$$
appearing in \eqref{diag: key diag2}.

\sssec{}
Next, we define the left vertical arrow. Let $M\in \Mod^{\etp}_{\ccO_X[u]}\bigt{\Cohminus(X)_{\eZ}}$.
We denote by $(M,u)$ the diagram
$$
\cdots  \xto{u} M[-4] \xto{u}M[-2] \xto{u} M
$$
and consider the dg-functor
$$
  \Mod^{\etp}_{\ccO_X[u]}\bigt{\Cohminus(X)_{\eZ}}
    \to
  \Cohinfty(X)
$$
$$
  M 
    \mapsto 
  \varprojlim (M,u)
    :=
  \varprojlim \bigt{ \cdots \xto{u}M[-2] \xto{u} M}.
$$
Clearly, this limit carries an action of $\ccO_X[u]$, so that we have a dg-functor
\begin{equation}\label{eqn: dg-functor lim (M,u)}
    \Mod^{\etp}_{\ccO_X[u]}\bigt{\Cohminus(X)_{\eZ}}
    \to
    \Mod_{\ccO_X[u]}\bigt{\Cohinfty(X)}.
\end{equation}
Our next task will be to show that this dg-functor factors through $\Mod_{\ccO_X[u,u^{-1}]}\bigt{\Cohinfty(X)_{\red{Z}}}$.
\sssec{}
For $a \in \bbZ$, let $\Cohlevel{a}(X)\subseteq \Cohinfty(X)$ (resp. $\Cohlevelplus{a}(X)\subseteq \Cohinfty(X)$) denote the full subcategory spanned by those complexes of $\ccO_X$-modules cohomologically concentrated in degrees $\leq a$ (resp. $\geq a$).
Recall that for every $a\in \bbZ$ the t-structure on $\Cohinfty(X)$ induces a co-localization
$$
  \iota^{\leq a}:
  \Cohlevel{a}(X)
    \leftrightarrows
  \Cohinfty(X)
  : \tau^{\leq a}
$$
and a localization
$$
  \tau^{\geq a}:
  \Cohinfty(X)
    \leftrightarrows
  \Cohlevelplus{a}(X)
  : \iota^{\geq a}.
$$
We start with the following observation, which will help us to compute the cohomology groups of $\varprojlim (M,u)$.

\begin{lem}\label{lem: etp implies ileqa commutes with directed limits}
Let $M \in \Mod^{\etp}_{\ccO_X[u]}\bigt{\Cohminus(X)_{\eZ}}$.
There is a canonical equivalence
$$
  \iota^{\leq a} \varprojlim \tau^{\leq a} (M,u)
    \xto{\simeq}
  \varprojlim \iota^{\leq a} \tau^{\leq a} (M,u).
$$
\end{lem}

\begin{proof}
We consider the morphism
$$
  \iota^{\leq a} \varprojlim \tau^{\leq a} (M,u)
    \to
  \varprojlim \iota^{\leq a} \tau^{\leq a} (M,u)
$$
which corresponds under adjuntion to
$$
  \id:
  \varprojlim \tau^{\leq a} (M,u)
    \to
  \tau^{\leq a}
  \varprojlim  \iota^{\leq a} \tau^{\leq a} (M,u)
    \simeq 
  \varprojlim \tau^{\leq a} \iota^{\leq a} \tau^{\leq a} (M,u)
    \simeq
  \varprojlim \tau^{\leq a} (M,u).
$$
Since $M \in \Mod^{\etp}_{\ccO_X[u]}\bigt{\Cohminus(X)_{\eZ}}$, we have that $\tau^{\leq a - n}(M[-2]\xto{u}M)$ is an equivalence for $n\gg 0$.
This means that the diagram
$$
  \cdots 
    \xto{u}
  \tau^{\leq a}(M[-2n])
    \xto{u}
  \tau^{\leq a}(M[-2n+2])
    \xto{u}
  \cdots
    \xto{u}
  \tau^{\leq a}(M)
$$
stabilizes, that is, that the arrow
$$
  \tau^{\leq a}(M[-2n])
    \xto{u}
  \tau^{\leq a}(M[-2n+2])
$$
is a quasi-isomorphism for all $n\gg 0$.
In particular, we obtain that
$$
  \varprojlim \tau^{\leq a} (M,u)
    \simeq 
  \tau^{\leq a}(M[-2n])
$$
for $n \gg 0$.
Similarly, one sees that 
$$
  \varprojlim \iota^{\leq a} \tau^{\leq a} (M,u)
    \simeq
  \iota^{\leq a} \tau^{\leq a}(M[-2n'])
$$
for $n'\gg 0$.

Under these equivalences, choosing $n=n'\gg0$, the morphism
$$
  \iota^{\leq a} \varprojlim \tau^{\leq a} (M,u)
    \to
  \varprojlim \iota^{\leq a} \tau^{\leq a} (M,u)
$$
corresponds to the identity of $\iota^{\leq a} \tau^{\leq a}(M[-2n])$.
\end{proof}

\begin{cor}\label{cor: cohomology groups lim (M,u)}
Let $M \in \Mod^{\etp}_{\ccO_X[u]}\bigt{\Cohminus(X)_{\eZ}}$.
There is a canonical equivalence
$$
  \iota^{\geq a}  \tau^{\geq a} \varprojlim(M,u)
    \simeq 
  \varprojlim \iota^{\geq a} \tau^{\geq a} (M,u).
$$
In particular,
$$
  \ccH^{a}\bigt{\varprojlim (M,u)}
    \simeq
  \ccH^{a-2n}(M)
$$
for $n\gg 0$.
\end{cor}
\begin{proof}
The commutative diagram 
\begin{equation*}
    \begin{tikzcd}
      \iota^{\leq a-1}\tau^{\leq a-1} \varprojlim (M,u)
      \arrow[r]
      \arrow[d," \simeq \;\; \on{Lemma} \ref{lem: etp implies ileqa commutes with directed limits}"]
      &
      \varprojlim (M,u)
      \arrow[r]
      \arrow[d,"\id"]
      &
      \iota^{\geq a}\tau^{\geq a}\varprojlim (M,u)
      \arrow[d]
      \\
      \varprojlim \iota^{\leq a-1}\tau^{\leq a-1} (M,u)
      \arrow[r]
      &
      \varprojlim (M,u)
      \arrow[r]
      &
      \varprojlim \iota^{\geq a}\tau^{\geq a} (M,u),
    \end{tikzcd}
\end{equation*}
immediately implies the first claim.
As for the second one, we consider the following chain of isomorphisms:
\begin{align*}
    \ccH^{a}\bigt{\varprojlim (M,u)} & := \tau^{\leq a} \tau^{\geq a}\varprojlim (M,u)
                                     \\
                                     & \simeq \tau^{\leq a} \tau^{\geq a} \iota^{\geq a} \tau^{\geq a}\varprojlim (M,u) && \tau^{\geq a} \iota^{\geq a}\simeq \id
                                     \\
                                     & \simeq \tau^{\leq a} \tau^{\geq a}  \varprojlim \iota^{\geq a} \tau^{\geq a} (M,u) && \text{first claim}
                                     \\
                                     & \simeq \tau^{\leq a} \tau^{\geq a} \iota^{\geq a} \varprojlim \tau^{\geq a} (M,u) && \iota^{\geq a}\; \text{is a right adjoint}
                                     \\
                                     & \simeq \tau^{\leq a} \varprojlim \tau^{\geq a} (M,u) && \tau^{\geq a} \iota^{\geq a}\simeq \id
                                     \\
                                     & \simeq \varprojlim \tau^{\leq a} \tau^{\geq a} (M,u) && \tau^{\leq a} \; \text{is a right adjoint}
                                     \\
                                     & \simeq \varprojlim_n \ccH^{a}(M[-2n]).
\end{align*}
However, since $M \in \Mod^{\etp}_{\ccO_X[u]}\bigt{\Cohminus(X)_{\eZ}}$, the inverse system
$$
  \cdots 
    \xto{u}
  \ccH^{a}(M[-2n])
    \xto{u}
  \ccH^{a}(M[-2n+2])
    \xto{u}
  \cdots
    \xto{u}
  \ccH^{a}(M)
$$
stabilizes to $\ccH^{a-2n}(M)$ for $n\gg0$.
\end{proof}

\begin{cor}\label{cor: u invertible on lim(M,u)}
The dg-functor
$$
    \Mod^{\etp}_{\ccO_X[u]}\bigt{\Cohminus(X)_{\eZ}}
    \to
    \Mod_{\ccO_X[u]}\bigt{\Cohinfty(X)}
$$
appearing in \eqref{eqn: dg-functor lim (M,u)} factors through 
$$
  \Mod_{\ccO_X[u,u^{-1}]}\bigt{\Cohinfty(X)_{\red{Z}}}
    \to
  \Mod_{\ccO_X[u]}\bigt{\Cohinfty(X)}.
$$
In particular, we obtain a dg-functor
$$
  \ol \ffL:
  \Mod^{\etp}_{\ccO_X[u]}\bigt{\Cohminus(X)_{\eZ}}
    \to
  \Mod_{\ccO_X[u,u^{-1}]}\bigt{\Cohinfty(X)_{\red{Z}}}.
$$
\end{cor}
\begin{proof}
Let $M \in \Mod^{\etp}_{\ccO_X[u]}\bigt{\Cohminus(X)_{\eZ}}$. 
Since $M$ is eventually supported on $\red{Z}$, Corollary \ref{cor: cohomology groups lim (M,u)} immediately implies that $\varprojlim (M,u)\in \Mod_{\ccO_X[u]}\bigt{\Cohinfty(X)_{\red{Z}}}$. It remains to show that the action of $u$ on $\varprojlim (M,u)$ is invertible. Equivalently, we need to verify that
$$
  \varprojlim (M,u)
    \xto{u}
  \varprojlim (M,u)[2]
$$
is a quasi-isomorphism.
By Corollary \ref{cor: cohomology groups lim (M,u)}, there exists some $n \gg 0$ such that
$$
  \ccH^{a}\bigt{\varprojlim (M,u)}
    \simeq 
  \ccH^{a-2n}(M).
$$
Then the morphism
$$
  \ccH^{a}\bigt{\varprojlim (M,u)}
    \xto{u}
  \ccH^{a+2}\bigt{\varprojlim (M,u)}
$$
corresponds to 
$$
  \ccH^{a-2n}(M)
    \xto{u}
  \ccH^{a-2n+2}(M)
$$
and the assertion follows from the \emph{etp} assumption on $M$.
\end{proof}

\sssec{}
Suppose now that $M \in \Mod_{\ccO_X[u]}\bigt{\Coh(X)}$. Then $\ul \Hom_X(M,M)$ is a bounded complex (indeed, $X$ is regular, hence $M$ is perfect). 
It follows that $u^n: M[-2n] \to M$ is homotopic to zero for $n \gg 0$ and thus the composition
$$
  \Mod_{\ccO_X[u]}\bigt{\Coh(X)}
    \subseteq 
  \Mod^{\etp}_{\ccO_X[u]}\bigt{\Cohminus(X)_{\eZ}}
    \xto{\ol \ffL}
  \Mod_{\ccO_X[u,u^{-1}]}\bigt{\Coh(X)_{\red{Z}}}
$$
is the zero functor. This gives the dg-functor
$$
  \ffL:
  \frac{\Mod^{\etp}_{\ccO_X[u]}\bigt{\Cohminus(X)_{\eZ}}}{\Mod_{\ccO_X[u]}\bigt{\Coh(X)}}
    \to
  \Mod_{\ccO_X[u,u^{-1}]}\bigt{\Coh(X)_{\red{Z}}}
$$
we wanted. 
\sssec{}
It remains to define the dg-functor $\ffT: \Cohinfty(X)^{\tp}_{\red{Z}}\to \frac{\Cohminus(X)_{\eZ}^{\Getp}}{\Coh(X)}$ on the right of diagram \eqref{diag: key diag2}.
Consider the right adjoint of the embedding $\Cohlevel{n}(X)\subseteq \Cohinfty(X)$:
$$
  \tau^{\leq n}: 
  \Cohinfty(X)
    \to
  \Cohlevel{n}(X).
$$
We define $\ffT: \Cohinfty(X)^{\tp}_{\red{Z}}\to \frac{\Cohminus(X)_{\eZ}^{\Getp}}{\Coh(X)}$ simply as the composition
$$ 
  \Cohinfty(X)^{\tp}_{\red{Z}}
    \xto{\tau^{\leq 0}}
  \Cohminus(X)_{\eZ}^{\Getp}
    \to
  \frac{\Cohminus(X)_{\eZ}^{\Getp}}{\Coh(X)}.
$$

\begin{rmk}
The above dg-functor $\ffT$ is canonically equivalent to the composition
$$
  \Cohinfty(X)^{\tp}_{\red{Z}}
    \xto{\tau^{\leq n}}
  \Cohminus(X)_{\eZ}^{\Getp}
    \to
  \frac{\Cohminus(X)_{\eZ}^{\Getp}}{\Coh(X)}
$$
for every $n \in \bbZ$.
Indeed, suppose that $n \geq 0$ (the case $n \leq 0$ is analogous). 
Then there is a natural transformation
$$
  \tau^{\leq 0}\to \tau^{\leq n}: 
  \Cohinfty(X)^{\tp}_{\red{Z}}
    \to
  \Cohminus(X)_{\eZ}^{\Getp}
$$
and 
$$
  \Fib \bigt{\tau^{\leq 0}(M) 
    \to 
  \tau^{\leq n}(M)} 
    \in 
  \Coh(X)
$$
for any $M \in \Cohinfty(X)^{\tp}_{\red{Z}}$.
Hence, the above natural transformation becomes an equivalence when post-composed with the quotient functor
$$
  \Cohminus(X)_{\eZ}^{\Getp}
    \to
  \frac{\Cohminus(X)_{\eZ}^{\Getp}}{\Coh(X)}.
$$
\end{rmk}

\begin{lem}
The diagram \eqref{diag: key diag2} is commutative.
\end{lem}
\begin{proof}
Consider the (non-commutative) diagram
\begin{equation*}
    \begin{tikzcd}
      \Mod^{\etp}_{\ccO_X[u]}\bigt{\Cohminus(X)_{\eZ}}
      \arrow[r,"\on{forget}"]
      \arrow[d,"\ffL"]
      &
      \Cohminus(X)_{\eZ}^{\Getp}
      \\
      \Mod_{\ccO_X[u,u^{-1}]}\bigt{\Cohinfty(X)_{\red{Z}}}
      \arrow[r,"\on{forget}"]
      & 
      \Cohinfty(X)^{\tp}_{\red{Z}}.
      \arrow[u,"\tau^{\leq 0}"]
    \end{tikzcd}
\end{equation*}
We first exhibit a natural transformation
$$
  \alpha:
  \tau^{\leq 0} \circ \on{forget}\circ \ffL 
    \to 
  \on{forget}
$$
between the two circuits
$$
  \Mod^{\etp}_{\ccO_X[u]}\bigt{\Cohminus(X)_{\eZ}}
    \to
  \Cohminus(X)_{\eZ}^{\Getp}
$$
of the above diagram. Since $\tau^{\leq 0}$ commutes with limits, we have arrows
$$
  \tau^{\leq 0}\circ \on{forget}\circ \ffL (M)
    \simeq
  \varprojlim \bigt{\cdots \xto{\tau^{\leq 0}u}\tau^{\leq 0}(M[-2])\xto{\tau^{\leq 0}u} \tau^{\leq 0}M }
    \to
  \tau^{\leq 0} \circ \on{forget}(M)
    \to
  \on{forget}(M)
$$
that are natural in $M \in \Mod^{\etp}_{\ccO_X[u]}\bigt{\Cohminus(X)_{\eZ}}$. This defines $\alpha$.

To prove the lemma, it remains to show that $\alpha$ induces a natural equivalence when post-composed with 
$$
  \Cohminus(X)_{\eZ}^{\Getp}
    \to 
  \frac{\Cohminus(X)_{\eZ}^{\Getp}}{\Coh(X)}.
$$
For $M \in \Mod^{\etp}_{\ccO_X[u]}\bigt{\Cohminus(X)_{\eZ}}$, choose $n_0 \in \bbN$ such that $u: M[-2] \to M$ induces equivalences
$$
  \ccH^{-n-2}(M)
    \xto{\simeq}
  \ccH^{-n}(M)
$$
for all $n \geq n_0$.
It follows Corollary \ref{cor: cohomology groups lim (M,u)} that
$$
  \ccH^{-n} \bigt{\tau^{\leq 0} \circ \on{forget}\circ \ffL (M)}   
    \simeq
  \ccH^{-n}\bigt{\on{forget}(M)}
$$
for all $n \geq n_0$. Moreover, it is clear that
$$
\Fib \bigt{\alpha_M: \tau^{\leq 0} \circ \on{forget}\circ \ffL (M) \to \on{forget}(M)} \in \Coh(X)
$$
as desired.
\end{proof}

\subsection{Localized intersection product with the diagonal} \label{ssec:KS localized product with diag}

We are now ready to prove the following result.

\begin{thm}\label{thm: intersection with the diagonal}
The dg-functor $\Upsilon: \Sing(X\times_SX)\to \Sing(\bbA_X^{1}[-1])_{\red{Z}}$ induces the following morphism of abelian groups:
$$
  \G_0(X\times_SX)
    \to
  \HK_0^{\on{sg}}(X\times_SX)
    \xto{\Upsilon}
  \HK_0^{\on{sg}}(\bbA_X^{1}[-1])_{\red{Z}}
    \simeq  
  \G_0(X)_{\red{Z}}
$$
$$
  [M]
    \mapsto
  [\![\Delta_X,M]\!]_{X\times_SX}=
  \sum_{i=0,1}(-1)^i[\ul \sTor^{X\times_SX}_i(\Delta_X,M)].
$$
\end{thm}
\begin{proof}
Recall from \cite{brtv18} the equivalence between $\Sing(\bbA^1_X[-1])_{\red{Z}}$ and the dg-category of matrix factorizations $\MF(X,0)_{\red{Z}}$. 
In particular, every object $E \in \Sing(\bbA^1_X[-1])_{\red{Z}}$ yields two coherent $\ccO_X$-modules supported on $\red{Z}$, to be denoted $s\ccH^0(E)$ and $s\ccH^1(E)$. Recall also that the isomorphism
$$
  \HK_0^{\on{sg}} (\bbA^1_X[-1])_{\red{Z}}
    \xto{\sim}
  \G_0(X)_{\red{Z}}
$$
is given by the map
$$
  [E]
    \mapsto
  [s\ccH^0(E)]-[s\ccH^1(E)].
$$
Thus, for $M \in \Coh(X\times_SX)$, we need to compute the G-theory classes
$$
  [s\ccH^i(\Upsilon(M))]
    \in 
  \G_0(X)_{\red{Z}}
    \;\;\;\;
  i=0,1.
$$
By \cite[Theorem 3.2.1]{katosaito04}, pullback along $\delta_{X/S}$ restricts to a dg-functor
$$
  \Coh(X\times_SX)
    \xto{\delta^*_{X/S}}
  \Cohminus(X)^{\Getp}_{\eZ}.
$$
Consider the following commutative diagram
\begin{equation*}
    \begin{tikzcd}
    \Perf(X \times_S X)_{\red{Z} \times_s \red{Z}}
    \arrow[r,hook]
    \arrow[d]
    &
   \Perf(X \times_S X)
    \arrow[r, "\delta_{X/S}^*"]
    \arrow[d]
    &
    \Perf(X) \simeq \Coh(X)
    \arrow[d]
    \\
    \Coh(X \times_S X)_{\red{Z} \times_s \red{Z}}
    \arrow[r,hook]
    \arrow[d ]
    &
   \Coh(X \times_S X)
     \arrow[r, "\delta_{X/S}^*"]
    \arrow[d]
    &
    \Cohminus(X)_{\eZ}^{\Getp}
    \arrow[d]
    \\
    \Sing(X \times_S X)_{\red{Z} \times_s \red{Z}}
    \arrow[r,hook]
    &
    \Sing(X \times_S X)
    \arrow[r, "\delta_{X/S}^*"]
    &
     \frac{\Cohminus(X)_{\eZ}^{\Getp}}{\Coh(X)}.
    \end{tikzcd}
\end{equation*}
Using the fact that $\Sing(X\times_SX)_{\red{Z}\times_s \red{Z}}\hto \Sing(X\times_SX)$ is Karoubi-essentially surjective, and thanks to Remark \ref{rmk: Upsilon^+ vs delta^*}, we see that the following diagram commutes:
\begin{equation*}
    \begin{tikzcd}[column sep = 2.5cm]
      \Coh(X\times_SX)
      \arrow[r,"\delta_{X/S}^*"]
      \arrow[d]
      &
      \Cohminus(X)^{\Getp}_{\eZ}
      \arrow[d]
      \\
      \Sing(X\times_SX)
      \arrow[r,"\wt{\Upsilon}"]
      &
      \frac{\Cohminus(X)_{\eZ}^{\Getp}}{\Coh(X)}.
    \end{tikzcd}
\end{equation*}

More precisely, the map at the bottom is the composition
\begin{align*}
  \Sing(X\times_SX)
    &
    \xto{\Upsilon}
  \MF(X,0)_{\red{Z}}
    \\
    &
    \simeq
  \frac{
        \Mod_{\ccO_X[u]}^{\etp}\bigt{\Cohminus(X)_{\red{Z}}}
        }
        {
        \Mod_{\ccO_X[u]}\bigt{\Coh(X)_{\red{Z}}}
        }
      && \text{Lemma } \ref{cor: MFcoh vs Cohminus etp/Cohb}
      \\
      &
      \xto{\iota}
  \frac{
        \Mod_{\ccO_X[u]}^{\etp}\bigt{\Cohminus(X)_{\eZ}}
        }
        {
        \Mod_{\ccO_X[u]}\bigt{\Coh(X)}
        }
      && \text{induced by the inclusion}
      \\
      &
      \xto{\eqref{eqn: Mod^etp_eZ --> Cohminus^Getp_eZ}}
    \frac{
          \Cohminus(X)_{\eZ}^{\Getp}
          }
          {
           \Coh(X)
           }.
\end{align*}
Thus, for any $M \in \Coh(X\times_SX)$, we have a canonical equivalence
\begin{equation}\label{eqn: Psi = delta^*}
\wt{\Upsilon}(M)
\simeq 
\delta_{X/S}^*(M)
\in
\frac{\Cohminus(X)_{\eZ}^{\Getp}}{\Coh(X)}.
\end{equation}
By the commutativity of diagram \eqref{diag: key diag2}, we obtain
$$
\wt{\Upsilon}(M)
\simeq
\ffT \circ \on{forget} \circ \ffL \circ \iota \circ \Upsilon(M).
$$

One computes that the composition
\begin{align*}
    \Sing(\bbA^1_X[-1])_{\red{Z}}
      &
      \simeq
    \frac{
        \Mod^{\etp}_{\ccO_X[u]}\bigt{\Cohminus(X)_{\red{Z}}}
        }
        {
         \Mod_{\ccO_X[u]}\bigt{\Coh(X)_{\red{Z}}}
         }
      \\
      &
      \xto{\iota}
    \frac{
          \Mod^{\etp}_{\ccO_X[u]}\bigt{\Cohminus(X)_{\eZ}}
          }
          { 
           \Mod_{\ccO_X[u]}\bigt{\Coh(X)}
           }
      \\ 
      &
      \xto{\ffL} 
    \Mod_{\ccO_X[u,u^{-1}]}\bigt{\Cohinfty(X)_{\red{Z}}}
      \\
      &
      \xto{\on{forget}}
    \Cohinfty(X)^{\tp}_{\red{Z}}
\end{align*}
sends an object $E \in \Sing(\bbA^1_X[-1])_{\red{Z}}$ to an object $E'$ in $\Cohinfty(X)^{\tp}_{\red{Z}}$ such that
$$
\ccH^{2n+i}(E')\simeq s\ccH^i(E)
$$
for all $n \in \bbZ$ and $i=0,1$.

Using \eqref{eqn: Psi = delta^*}, we conclude that 
$$
  [s\ccH^i\bigt{\Upsilon(M)}]
    \simeq
  \bigq{\ul \Tor^{X\times_SX}_{2n+i}(\Delta_X,M)}
    \in 
  \G_0(X)_{\red{Z}}.
$$
for every $M \in \Coh(X\times_SX)$, provided that $n \gg 0$.

\end{proof}

\subsection{The generalized Bloch conductor formula}

As a consequence of the above theorem, we obtain a categorification of the localized intersection product defined by Kato--Saito. As we explain presently, this leads to the proof of the generalized Bloch conductor conjecture.

\begin{cor}\label{cor: intersection with the diagonal}
The dg-functor $\int_{X/S}:\Sing(X\times_SX)\to \Sing(\bbA_S^{1}[-1])_s$ induces the following morphism of abelian groups:
$$
  \G_0(X\times_SX)
    \to
  \HK_0^{\on{sg}}(X\times_SX)
    \xto{\int_{X/S}}
  \HK_0^{\on{sg}}(\bbA_S^{1}[-1])_s
    \simeq 
  \bbZ
$$
$$
  [M]
    \mapsto
  [\![\Delta_X,M]\!]_S=
  \sum_{i=0,1}(-1)^i\deg 
    \Big[ 
        \ul \sTor^{X\times_SX}_i(\Delta_X,M)
    \Big].
$$
In particular, we have
$$
\int_{X/S}[\Delta_X]= \Bl(X/S).
$$
\end{cor}
\begin{proof}
By diagram \eqref{diagram: alternative expressions integration functor} the integration dg-functor
$\int_{X/S}:\Sing(X\times_SX) \to \Sing(\bbA^1_S[-1])_s$ factors as
$$
  \Sing(X\times_SX) 
    \xto{\Upsilon} 
  \Sing(\bbA^1_X[-1])_{\red{Z}}
    \to
  \Sing(\bbA^1_S[-1])_s
$$
where the rightmost dg-functor is induced by the pushforward along $\bbA^1_X[-1] \to \bbA^1_S[-1]$.
Then the first claim follows immediately from Theorem \ref{thm: intersection with the diagonal}.
The last claim then follows readily from \cite[Corollary 3.4.5]{katosaito04}
\end{proof}

\sssec{}

In \cite{beraldopippi24}, we proved the following formula (see \eqref{eqn: categorical gBCF}):
\begin{equation}\label{eqn: categorical gBCC}
\Bl^{\cat}(X/S)= - \dimtot^{\cat}(\Phi) \in \uH^0_\et\Bigt{S,\rl_S\bigt{\HH(\sB/A)}}.
\end{equation}
We will now show how this formula implies the following
\begin{thm}\label{thm: gBCC}
With the same notation as above,
$$
\Bl(X/S)=-\dimtot(\Phi)\in \bbZ.
$$
\end{thm}

\begin{proof}
Apply
$$
  \oint:
  \uH^0_\et\Bigt{S,\rl_S\bigt{\HH(\sB/A)}} 
    \to 
  \uH^0_\et\Bigt{S,\rl_S\bigt{\Sing(\bbA^1_S[-1])}}
    \simeq 
  \Qell
$$
to both members of \eqref{eqn: categorical gBCC}.
By definition, $\Bl^{\cat}(X/S)= \chern \bigt{\bigq{\ev^{\HH}(\Delta_X)}} \in \uH^0_{\et}\Bigt{S,\rl_S \bigt{\HH(\sB/A)}}$.
Since $\int_{X/S}$ is by definition the composition
$$
  \Sing(X\times_SX)
    \xto{\ev^{\HH}}
  \HH(\sB/A)
    \xto{\oint} 
  \Sing(\bbA^1_S[-1]),
$$
Corollary \ref{cor: intersection with the diagonal} shows that 
$$
\oint\bigt{\Bl^{\cat}(X/S)}=\int_{X/S}[\Delta_X]=\Bl(X/S).
$$
On the other hand, by the definition given in \cite{beraldopippi24}, we have
$$
-\dimtot^{\cat}(\Phi)=\id^{\cat}\bigt{\Tr_{\Qell^{\uI}}(\id,\Phi[1])}-\arcat \bigt{\Tr_{\QellIGred}(\id,\Phi^\IL/\Phi^\IK)},
$$
where $\uK \subseteq \uL$ denotes a finite Galois extension with Galois group $\GLK$ such that the inertia group $\IL$ of $\uL$ acts unipotently on $\Phi$.
Then Lemma \ref{lem: idcat and arcat} below immediately implies that
$$
\oint \bigt{-\dimtot(\Phi)}=\Tr_{\Qell^{\uI}}(\id,\Phi[1])- \ar \bigt{\Tr_{\QellIGred}(\id,\Phi^\IL/\Phi^\IK)}=-\dimtot(\Phi)
$$
as claimed.
\end{proof}

\begin{lem}\label{lem: idcat and arcat}
Using the notation of \cite{beraldopippi24}, we have:
\begin{enumerate}
    \item The dg-functor 
    $$
    \sB \xto{\ev^{\HH}_{\sB/\sB}}
    \HH(\sB/A)
    \xto{\oint} \Sing(\bbA^1_S[-1])_s
    $$
    induces the identity at the level of homotopy-invariant non-connective algebraic K-theory:
    $$
    \id:\HK_0 (\sB)\simeq \bbZ \xto{\id^{\cat}} \HK_0\bigt{\HH(\sB/A)} \to \HK_0^{\on{sg}}(\bbA^1_S[-1])_s\simeq \bbZ.
    $$
    \item The dg-functor 
    $$
    \sC \xto{\ev^{\HH}_{\sC/\sC}} \HH(\sC/A)\simeq \HH(\sB/A)\xto{\oint} \Sing(\bbA^1_S[-1])_s
    $$
    induces the opposite of the Artin character at the level of homotopy-invariant non-connective algebraic K-theory:
    $$
    -\ar:\HK_0 (\sC)\simeq \bbZ(\GLK) \to \HK_0\bigt{\HH(\sB/A)} \to \HK_0^{\on{sg}}(\bbA^1_S[-1])_s\simeq \bbZ.
    $$
\end{enumerate}
\end{lem}
\begin{proof}
The first statement is an immediate consequence of Theorem \ref{thm: M(B) is a retract of M(HH(B/A))}. Let us prove the second statement. By Theorem $\ref{thm: intersection with the diagonal}$, the composition
$$
  \G_0(\sC)
    \to 
  \HK_0(\sC)  
    \xto{-\ar^{\cat}}
  \HK_0\bigt{\HH(\sC/A)}
    \simeq 
  \HK_0\bigt{\HH(\sB/A)}
    \to 
  \HK_0\bigt{\Sing(\bbA^1_S[-1])_s}
$$
is computed by
$$
  [M]
    \mapsto
  \sum_{i=0,1}(-1)^i\bigq{\sTor^{S'\times_SS'}_i(M,\Delta_{S'})}.
$$
By \cite[Proposition 6.3.5]{beraldopippi24}, we know that we have a surjection
$$
  \G_0(\sC)\simeq \bbZ[\GLK] 
    \tto 
  \bbZ(\GLK)\simeq
  \HK_0(\sC)
$$
and that 
$$
\sum_{i=0,1}(-1)^i\bigq{\sTor^{S'\times_SS'}_i(\Gamma_g,\Delta_{S'})}=-\ar(g)
$$
for every $g\in \GLK$.
\end{proof}

\subsection{Twisted Bloch conductor formula}

So far, we have only considered a small bit of intersection theory for arithmetic schemes.
Namely, we have restricted our attention to the intersection \emph{with the diagonal}.
However, our methods allow to understand the intersection with any object in $\Coh(X\times_SX)$. We briefly discuss this here.

\sssec{}

An object $M \in \Coh(X\times_SX)$ gives rise to an $A$-linear dg-functor
$$
  F_M: 
  \Perf(A)
    \to
  \Sing(X\times_SX).
$$
Since $\sT$ is a dualizable $\bigt{\sB,\Perf(A)}$-bimodule and $\Sing(X\times_SX)\simeq \sT^\op \otimes_{\sB}\sT$, the dg-functor $F_M$ corresponds to a $(\sB,\sB)$-bilinear dg-functor
$$
  \ev_M:
  \sT \otimes_A \sT^\op 
    \to
  \sB.
$$
Similarly, one gets a $(\sC,\sC)$-bilinear dg-functor
$$
  \ev_M':
  \sU \otimes_A \sU^\op 
    \to
  \sC,
$$
where we use the same notation as in \cite{beraldopippi24}.
As usual, this yields the dg-functor
$$
  \ev^{\HH}_M:
  \Sing(X\times_SX)
    \simeq 
  \sT^\op \otimes_{\sB}\sT
    \to
  \HH(\sB/A)
$$
sitting in a commutative
\begin{equation*}
    \begin{tikzcd}
    \sT \otimes_A \sT^\op 
    \rar
    \dar["\ev_M"]
    &
    \sT^\op \otimes_{\sB}\sT
    \simeq
    \sU^\op \otimes_{\sC}\sU
    \dar["\ev^{\HH}_M"]
    &
    \lar
    \sU \otimes_A \sU^\op 
    \dar["\ev_M'"]
    \\
    \sB
    \rar["\ev^{\HH}_{\sB}"]
    &
    \HH(\sB/A)
    \simeq
    \HH(\sC/A)
    &
    \lar["\ev^{\HH}_{\sC}"]
    \sC.
    \end{tikzcd}
\end{equation*}

\sssec{} 

Applying $\HK_0$ to the composition
$$
  \Coh(X\times_SX)
    \to
  \Sing(X\times_SX)
    \xto{\ev^{\HH}_M}
  \HH(\sB/A)
    \xto{\oint}
  \Sing(\bbA^1_S[-1])_s,
$$
one gets the homomorphism
$$
  \sG_0(X\times_SX)
    \to
  \bbZ
$$
$$
  [N] 
    \mapsto
  \sum_{i=0,1}(-1)^i\deg \bigq{\ul \sTor^{X\times_SX}_i(M,N)},
$$
which was denoted by $[\![M,-]\!]_S$ in \cite{katosaito04}.

\sssec{}\label{rmk: very general BCC}

Let $M \in \Coh(X\times_SX)$ and $\uK \subseteq \uL$ a sufficiently large finite Galois extension as in the proof of Theorem \ref{thm: gBCC}.
Let us use the same notation as in \cite{beraldopippi24}. 
In particular, let
$$
  \nu 
  :=
  p_* \Phi_{X/S}\bigt{\Ql{,X}(\beta)}
    \in
  \Mod_{\Ql{,S}(\beta)}\bigt{\Shv^{\IK}(S)}.
$$

Recall the equivalence 
$$
  \rl_S \bigt{\Sing(X\times_SX)}
    \simeq 
  \Bigt{
        \nu^{\IK}[-1]\otimes_{\Ql{,S}^{\uI}(\beta)}\nu^{\IK}[-1]
        }
    \oplus 
  \Bigt{
        (\nu^\IL/\nu^\IK)\otimes_{\Ql{,S}^{\uI}(\beta)(\GLK)}(\nu^\IL/\nu^\IK)
        }
$$
proven in \cite[Proposition 5.6.3]{beraldopippi24},
and the fact that $\nu^{\IK}[-1]$ (respectively, $\nu^\IL/\nu^\IK$) is a dualizable $\Ql{,S}^{\uI}(\beta)$-module 
(respectively, $\bigt{\Ql{,S}^{\uI}(\beta)(\GLK),\Ql{,S}^{\uI}(\beta)}$-bimodule),
see \cite[Proposition 5.4.1]{beraldopippi24} (respectively, \cite[Proposition 5.5.5]{beraldopippi24}).

Using these, we can define the \emph{$M$-twisted evaluations}
$$
  \ev^M_{\nu^{\IK}[-1]}: 
  \nu^{\IK}[-1]\otimes_{\Ql{,S}^{\uI}(\beta)}\nu^{\IK}[-1] 
    \to 
  \Ql{,S}^{\uI}(\beta),
$$
$$
  \ev^M_{\nu^{\IL}/\nu^\IK}: 
  (\nu^\IL/\nu^\IK)\otimes_{\Ql{,S}^{\uI}(\beta)}(\nu^\IL/\nu^\IK)
    \to
  \Ql{,S}^{\uI}(\beta)(\GLK).
$$

\sssec{} 

Similarly, for any $N \in \Coh(X\times_SX)$, we obtain \emph{$N$-twisted coevaluations} as follows: 
$$
  \coev^N_{\nu^{\IK}[-1]}:\Ql{,S}^{\uI}(\beta)
    \to
  \nu^{\IK}[-1]\otimes_{\Ql{,S}^{\uI}(\beta)}\nu^{\IK}[-1],
$$
$$
  \coev^N_{\nu^{\IL}/\nu^\IK}: 
  \Ql{,S}^{\uI}(\beta)
    \to 
  (\nu^\IL/\nu^\IK)\otimes_{\Ql{,S}^{\uI}(\beta)(\GLK)}(\nu^\IL/\nu^\IK).
$$

\sssec{} 
We can then consider the compositions
$$
  \Tr^{(M,N)}_{\Ql{,S}^{\uI}(\beta)}(\nu^\IK[-1]):
  \Ql{,S}^\IK(\beta)
    \xto{\coev^N_{\nu^{\IK}[-1]}}
  \nu^{\IK}[-1]\otimes_{\Ql{,S}^{\uI}(\beta)}\nu^{\IK}[-1] 
    \xto{\ev^M_{\nu^{\IK}[-1]}} 
  \Ql{,S}^\IK(\beta),
$$
\begin{align*}
    \Tr^{(M,N)}_{\Ql{,S}^{\uI}(\beta)(\GLK)}(\nu^\IL/\nu^\IK):\Ql{,S}^\IK(\beta)
    & 
    \xto{\coev^N_{\nu^{\IL}/\nu^\IK}}(\nu^\IL/\nu^\IK)\otimes_{\Ql{,S}^{\uI}(\beta)(\GLK)}(\nu^\IL/\nu^\IK) 
    \\
    & 
    \xto{\ev^{M,\HH}_{\nu^{\IL}/\nu^\IK}}
    \HH\bigt{\Ql{,S}^{\uI}(\beta)(\GLK)/\Ql{,S}^{\uI}(\beta)}. 
\end{align*}

\sssec{} 
Repeating the steps of the above proof, one obtains the following equality:
$$
  [\![ M,N ]\!]_S
    = 
  -\dimtot^{(M,N)} (\Phi)
    \in 
  \bbZ,
$$
where
$$
  -\dimtot^{(M,N)} (\Phi)
    \\
    := 
  \Tr^{(M,N)}_{\Ql{,S}^{\uI}(\beta)}(\nu^\IK[-1])
  -\ar \bigt{ \Tr^{(M,N)}_{\Ql{,S}^{\uI}(\beta)(\GLK)}(\nu^\IL/\nu^\IK)}.
$$

\printbibliography

\end{document}